\documentclass[11pt,a4paper] {article}

 \usepackage{a4}
\usepackage{amsmath} 
\usepackage{amssymb}
\usepackage{amsthm}
\usepackage{graphicx}
\usepackage{tikz}
\usepackage{subfigure}
\usepackage{float} 

\def\+{\oplus}

\newcommand{\R}{{\mathbb R}}

\newcommand{\N}{{\mathbb N}}

\newcommand{\cG}{{\mathcal G}}

\newcommand{\cS}{{\mathcal S}}

\newcommand{\cC}{{\mathcal C}}

\newcommand{\cP}{{\mathcal P}}

\newcommand{\cT}{{\mathcal T}}

\newcommand{\cR}{{\mathcal R}}

\renewcommand{\epsilon}{\varepsilon}
\newcommand{\ph}{\varphi}
\renewcommand{\a}{{\alpha}}

\newcommand{\Ga}{{\Gamma}}

\newcommand{\fl}{ {\widetilde{f\!\ell}}}
\newcommand{\FL}{{\rm{FL}}}

\newcommand{\tH}{\widetilde{\rm{H}}}

\newcommand{\ds}{\displaystyle}
\def\squareforqed{\hbox{\rlap{$\sqcap$}$\sqcup$}}
\def\qed{\ifmmode\else\unskip\quad\fi\squareforqed}
\def\smartqed{\def\qed{\ifmmode\squareforqed\else{\unskip\nobreak\hfil
\penalty50\hskip1em\null\nobreak\hfil\squareforqed
\parfillskip=0pt\finalhyphendemerits=0\endgraf}\fi}}
\newcommand{\supp}{\mathrm{supp}\;}

\newtheorem{remark}{\textbf{Remark}}[section]
\newtheorem{property}{\textbf{Property}}[section]
\newtheorem{lemma}{\textbf{Lemma}}[section]
\newtheorem{theorem}{\textbf{Theorem}}[section]
\newtheorem{corollary}{\textbf{Corollary}}[section]
\newtheorem{proposition}{\textbf{Proposition}}[section]

\newtheorem{definition}{\textbf{Definition}}[section]

\numberwithin{equation}{section}
\title{Hamilton-Jacobi equations for optimal control on multidimensional junctions
}
\author{Salom{\'e} Oudet \thanks {IRMAR, Universit{\'e} de Rennes 1, Rennes, France, salome.oudet@univ-rennes1.fr}
}

\begin{document}

\maketitle

\abstract{We consider continuous-state and continuous-time control problems   
 where  the  admissible  trajectories of the system are constrained to remain on a union of half-planes which share a common straight line. This set will be named a junction. We define a notion of constrained viscosity solution of Hamilton-Jacobi equations on the junction
 and we propose a comparison principle whose proof is based on arguments from the optimal control theory.
}


\section{Introduction}

We are interested in optimal control problems whose trajectories  are constrained to remain on
a multidimensional junction. We define a {\sl junction} in $\R^d$, $d\ge 2$, as  
a union of half-hyperplanes sharing an affine space of dimension $d-2$, see Figure \ref{fig: the_geometry_half_plans} for a junction in $\R^3$. 
For simplicity, we shall limit ourselves to junctions in $\R^3$, although all what follows can be generalized for $d\ge 3$.
We shall name {\sl interface}  the straight line $\Gamma$  shared by the half-planes.
\begin{figure}[H]
\begin{center}
\begin{tikzpicture}[scale=0.65]
\draw (0,0) -- ++(30:5);
\draw (0,0) -- (0:5);
\draw (0:5) -- ++(30:4);
\draw (0,0) -- (90:5);
\draw (90:5) -- ++(30:4);
\draw (0,0) -- (135:4);
\draw (135:4) -- ++(30:3.26);
\draw (0,0) -- (240:3);
\draw (240:3) -- ++(30:3.15);
\draw (0,0) -- (320:4);
\draw (320:4) -- ++(30:4);
\draw[->] (255:2.4) to[bend right] (315:3.5);
\draw (30:5.2) node[right]{$\Gamma$} ;
\draw (5:5) node[left]{$\mathcal{P}_1$} ;
\draw (89:4.5) node[right]{$\mathcal{P}_2$} ;
\draw (132:3.6) node[right]{$\mathcal{P}_3$} ;
\draw (265:1.8) node[above]{$\mathcal{P}_4$} ;
\draw (324:3.9) node[above]{$\mathcal{P}_N$} ;

\draw[red, thick] [->] (0,0) -- (30:1);
\draw (30:1.2) node[above, red]{$e_0$} ;
\draw[blue, thick] [->] (0,0) -- (0:1);
\draw (-7:1.4) node[above, blue]{$e_1$} ;
\draw[blue, thick] [->]  (0,0) -- (90:1);
\draw (95:1.2) node[right, blue]{$e_2$} ;
\draw[blue, thick] [->] (0,0) -- (135:1);
\draw (135:1.6) node[right, blue]{$e_3$} ;
\draw[blue, thick] [->] (0,0) -- (240:1);
\draw (240:0.9) node[left, blue]{$e_4$} ;
\draw[blue, thick] [->] (0,0) -- (320:1);
\draw (322:1.7) node[above, blue]{$e_N$} ;
\end{tikzpicture}
\caption{The junction $\cS$ in $\R^3$}
\label{fig: the_geometry_half_plans}
\end{center}
\end{figure}
The junctions in $\R^d$  are particular {\sl ramified sets}, which we define as closed and connected subsets of $\R^d$ obtained as the union of embedded manifolds with dimension strictly smaller than $d$. 
 Figure \ref{fig: examples_of_ramified_spaces} below supplies  examples of ramified sets.
\begin{figure}[H]
\begin{center}
\subfigure[]{
    \label{fig: network}
\begin{tikzpicture}[scale=0.40]
\draw (0,0) -- (-1,3) -- (3,1) -- (0,0);
\draw (0,0) -- (-1.5,-2) -- (-1,3);
\draw (3,1) -- (1,-1) -- (3,-2.5);
\draw (0,0) node{$\bullet$};
\draw (3,1) node{$\bullet$};
\draw (-1.5,-2) node{$\bullet$};
\draw (-1,3) node{$\bullet$};
\draw (1,-1) node{$\bullet$};
\draw (3,-2.5) node{$\bullet$};
\end{tikzpicture}
}
\quad \quad \quad \quad
\subfigure[]{
    \label{fig: cylinder}
   \begin{tikzpicture}[scale=0.50]
\draw ({1.5*cos(-15)},0,{1.5*sin(-15)}) -- ++ (0,3,0);
\draw ({1.5*cos(160)},0,{1.5*sin(160)}) -- ++ (0,3,0);
\draw [domain=pi:2*pi+0.2, samples=80, smooth]
plot ( {1.5*cos(-\x r)},0, {1.5*sin(-\x r)}) ;
\draw [dashed, domain=-0.2:-pi-0.2, samples=80, smooth]
plot  ( {1.5*cos(\x r)},0, {1.5*sin(\x r)}) ;
\draw [domain=0:2*pi, samples=80, smooth]
plot ( {1.5*cos(\x r)},3, {1.5*sin(\x r)}) ;
\end{tikzpicture}
}
\quad \quad \quad \quad
\subfigure[]{
    \label{fig: cube}
    \begin{tikzpicture}[scale=0.50]
\draw [dashed] (0,3,0)--(0,0,0)--(3,0,0); 
\draw (3,0,0) -- (3,3,0) -- (0,3,0); 
\draw (0,0,3)--(3,0,3)--(3,3,3)--(0,3,3)--cycle; 
\draw [dashed] (0,0,0) -- (0,0,3); 
\draw (3,0,0) -- (3,0,3); 
\draw (3,3,0) -- (3,3,3); 
\draw (0,3,0) -- (0,3,3); 
\end{tikzpicture}
}
\caption[Optional caption for list of figures]{Examples of ramified sets}
\label{fig: examples_of_ramified_spaces}
\end{center}
\end{figure}
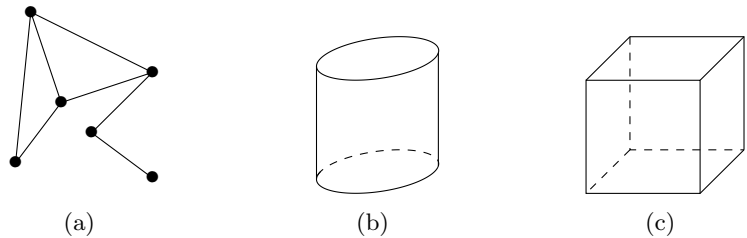
While there is a wide literature on control problems with state constrained in closures of open sets  (\cite{MR838056,MR861089}, \cite{MR951880}, \cite{IK}),  the interest on  problems with state constrained in  
closed sets with empty interior is more recent.   The results of Frankowska and Plaskacz \cite{MR1794772l,MR1713859} do apply to some closed sets with empty interior, but not to ramified sets except in very particular cases.  \\
The case of Hamilton-Jacobi equations on networks, see Figure \ref{fig: examples_of_ramified_spaces} \subref{fig: network}, is now well understood.  The first work 
seems to be the thesis of Schieborn in 2006,  \cite{Sch}. It was focused on eikonal equations. These results were improved later in \cite{MR3018167}.
The notion of viscosity solutions in these works are restricted to eikonal equations and cannot be used for more general control problems.
 The first two articles on optimal control problems  whose dynamics are constrained to a network were published in 2013: in \cite{MR3057137} Achdou, Camilli, Cutr{\`\i} and Tchou  proposed a Hamilton-Jacobi equation and  a definition of viscosity solution.  Independently, in \cite{MR3023064}  Imbert, Monneau and Zidani proposed a Hamilton-Jacobi approach to junction problems and traffic flows and gave an equivalent notion of viscosity solution.
 Both  \cite{MR3057137} and  \cite{MR3023064} contain the first comparison and uniqueness results for Hamilton-Jacobi equations on networks,
 but these results needed rather strong assumptions on the Hamiltonians behaviour.  A general comparison result has finally been obtained in the recent paper  by Imbert-Monneau \cite{imbert2013vertex}. In the latter, the Hamiltonians  in the edges are completely independent from each other; the main assumption is that the Hamiltonian in each edge, say $H_i(x,p)$ for the edge indexed $i$, is coercive and bimonotone, i.e. non increasing (resp. non decreasing)  for $p $ smaller (resp. larger) than a given threshold $p_i^0(x)$. Of course, convex Hamiltonian coming from optimal control theory
are bimonotone. Moreover,  in \cite{imbert2013vertex},  the authors consider more general  transmission conditions  than in  \cite{MR3057137,MR3023064}, allowing an additional running cost at the junctions. Soon after, Y. Achdou, S. Oudet and N. Tchou,  \cite{uniqueness2013},  proposed a different  proof of a general  comparison result in the case of control problems.  In \cite{uniqueness2013},  the dynamics and running costs be different on each edges, and the Hamiltonians  in the edges are a priori completely independent from each other, as in \cite{imbert2013vertex}. Whereas the proof of the 
comparison result  in \cite{imbert2013vertex} it is only based on arguments from the theory of partial differential equations, the proof in \cite{uniqueness2013} is  based on arguments from the theory of control which were first introduced by G. Barles, A. Briani and E. Chasseigne
 in \cite{barles2011bellman, barles2013bellman}. In the latter articles, the authors study control  problems in $\R^N$ whose
 dynamics and running costs  may be discontinuous across an hyperplane. The problems studied in \cite{barles2011bellman, barles2013bellman} 
and in \cite{uniqueness2013} have the common point that the data are discontinuous across a low dimensional subregion.
\\
There is even less literature on  Hamilton-Jacobi equations posed on more general ramified spaces. In their recent article  \cite{MR3062576},  F. Camilli, D. Schieborn and C. Marchi deal with eikonal equations, generalize the special notion of viscosity solutions proposed in  \cite{Sch,MR3018167}, and prove existence and uniqueness theorems.  The work by Y. Giga, N. Hamamuki and A. Nakayasu, see \cite{MR3271253}, is devoted to  eikonal equations in general metric spaces, and their results apply to ramified  spaces. The difficulty with
 general metric space $(\chi, d)$ is that the gradient $D u$ of a function $u:\chi \to \R$ is not well-defined in general.  Yet,  eikonal equation can be studied since a definition for the 
modulus of the gradient can be given.  More general results in geodesic metric spaces have been recently given by L. Ambrosio and J. Feng, see \cite{MR3160441}, see also 
the recent paper of \cite{Nakayasu}, who considered  evolutionary Hamilton-Jacobi  equations of  the form $u_t+H(x,|Du|)=0$  in a metric space.  
For  optimal control problems on ramified sets, we can mention the recent article by C. Hermosilla and H. Zidani \cite{Hermosilla_Zidani2014}, 
in which they study infinite horizon problems whose trajectories are constrained to remain in a set with a stratified structure. The authors  obtain existence and uniqueness results 
 with  weak controllability assumptions, but they assume that the dynamics is continuous.
\\
The present work is a continuation of \cite{uniqueness2013} (which was focused on networks),
but since the interface $\Gamma$ is now a straight line instead of a point, the  trajectories that stay on $\Gamma$ have a richer structure than in  \cite{uniqueness2013}.
We will have to introduce a tangential Hamiltonian $H_\Gamma^T$ to take these admissible trajectories into account. 
Different controllability assumptions can be made
\begin{enumerate}
\item strong controllability  in a neighborhood of $\Gamma$
\item a weaker controllability assumption  in a neighborhood of $\Gamma$, namely normal controllability to $\Gamma$,  see $[\tH_3]$ in \S~\ref{sec: framework1}.
\end{enumerate} 
As in  \cite{uniqueness2013}, the proof of the comparison results will be inspired by the arguments contained in \cite{barles2011bellman, barles2013bellman}.\\

In \S~\ref{section: classical_controllability}, we discuss the case when strong controllability  is assumed in a neighborhood of $\Gamma$. We 
propose a Bellman equation and a notion of viscosity solutions, prove that the value function is indeed a continuous viscosity solution of this equation, and give a comparison result. In \S~\ref{section_normal_controllability}, the same program is carried out when only normal controllability holds in a neighborhood of $\Gamma$. As in \cite{barles2013bellman}, we first prove that the value function is a discontinuous viscosity solution of the Bellman equation. We then prove a comparison result. The latter implies 
 the continuity of the  value function. Finally, in \S~\ref{sec:extens-more-gener}, we extend the results by assuming that in addition to the dynamics and costs related to the hyperplanes, there is a pair of  tangential dynamics  and tangential running cost defined on $\Gamma$. 
\\
Although we will not discuss it, the results obtained below can be generalized to ramified sets for which the interfaces are non intersecting manifolds of dimension $d-2$, see for example  Figure \ref{fig: examples_of_ramified_spaces} \subref{fig: cylinder}. On the contrary, it is not obvious to apply them to the ramified sets for  which 
interfaces of dimension $d-2$ cross each other, see Figure \ref{fig: examples_of_ramified_spaces} \subref{fig: cube}. This topic will hopefully be discussed in a forecoming work.

\section{First case : full controllability near the interface}\label{section: classical_controllability}

\subsection{Setting of the problem and basic assumptions}\label{sec:setting}

\subsubsection{The geometry}

We are going to study optimal control problems in $\R^d$, $d=3$, with constraints on the state of the system. The state is constrained to lie in the union $\cS$ of $N$ half-planes, $N>1$. Let $(e_i)_{i=0,\dots, N}$,  be some respectively distinct unit vectors in $\R^d$ such that $e_i.e_0=0$ for any $i\in\{1, \dots, N\}$, where for $x,y\in \R^d$, $x.y$ denotes the usual scalar product of the Euclidean space $\R^d$. The notation $|.|$ will be used for the usual Euclidean norm in $\R^d$. For $i=1, \ldots , N$, $\cP_i$ is the closed half-plane $\R e_0 \times \R^+ e_i$. We denote by $\Ga$ the straight line $\R e_0$.
The half-planes $\cP_i$ are glued at the straight line $\Ga$ to form the set $\cS$, see Figure \ref{fig: the_geometry_half_plans}:
  \begin{displaymath}
  \cS=\bigcup_{i=1}^N \cP_i.
\end{displaymath}

For any $x\in \cS$, we denote by $T_x(\cS)\subset \R^d$ the set of the tangent directions to $\cS$, i.e.
$T_x(\cS)=\R e_0 \times \R e_i$, for any $x\in \cP_i\backslash \Ga$ and $T_x(\cS)=\cup_{i=1}^N (\R e_0 \times \R e_i)$ for any $x\in \Ga$. \\
If $x\in \cS\setminus \Ga$, $\exists !\; i\in \{1,\ldots,N \}$, $\exists !\; x_0\in \R$ and $\exists !\; x_i\in \R_+\setminus \{0\}$ such that 
\begin{equation}\label{eq: decomposition_of_x1}
x=x_0e_0+x_ie_i.
\end{equation}
We will use often this decomposition. Sometime, it will be convenient to extend this decomposition to the whole set $\cS$, by writing $x=x_0e_0+0e_i$ for any $i\in \{1,\ldots,N\}$ if $x\in \Ga$. When we will not want to specify in which half-plan $\cP_i$ is $x$ belonging to $\cS$, we will use the notation 
\begin{equation}\label{eq: decomposition_of_x2}
x=x_0e_0+x',
\end{equation}
where $x'$ denotes $x_ie_i$ if $x=x_0e_0+x_ie_i$.\\
The geodesic distance $d(x,y)$ between two points $x,y$ of $\cS$ is 
\begin{equation}
\label{geodesic_distance}
    d(x,y)=\left\{  \begin{array}[c]{ll}
|x-y|\quad \hbox{if } x,y \hbox{ belong to the same half-plane } \cP_i \\
\min_{z\in \Ga}\{|x-z| + |z-y| \}  \quad \hbox{if } x,y \hbox{ belong to different half-planes } \cP_i \hbox{ and } \cP_j.
  \end{array}\right.
\end{equation}
More classically, if $x\in \R^d$ and $C$ is a closed subset of $\R^d$, $dist(x,C)$ will denote
\begin{displaymath}
dist(x,C)= \inf\{|x-z|: z\in C \},
\end{displaymath}
the distance between $x$ and $C$. The notation $B(\Ga,r)$ will be used to denote the set $\{x\in \R^d : dist(x,\Ga)<r\}$.  

\subsubsection{The optimal control problem}\label{sec: framework1}
We consider infinite-horizon optimal control problems which have different dynamics and running costs in the half-planes. We are going to describe the assumptions on the dynamics and costs in each half-plane $\cP_i$: the sets of controls are denoted by $A_i$, the system is driven by the dynamics $f_i$ and the running costs are given by $\ell_i$. Our main assumptions are as follows

\begin{description}
\item{[H0]} $A$ is a metric space (one can take $A=\R^m$). For $i=1,\dots,N$, $A_i$ is a non empty compact subset of $A$ and $f_i: \cP_i\times A_i \to \R e_0\times \R e_i$ is a continuous bounded function. The sets $A_i$ are disjoint.
 Moreover, there exists $L_f>0$ such that for any $i$, $x,y\in \cP_i$ and $a\in A_i$,
  \begin{displaymath}
    |f_i(x,a)-f_i(y,a)|\le L_f |x-y|.
  \end{displaymath}
 We note $M_f$ the minimal constant such that for any $x\in \cS$, $i\in \{1, \ldots, N\}$ and $a\in A_i$,  \begin{displaymath}
| f_i(x,a)| \le M_f.
  \end{displaymath}
   We will use also the notation $F_i(x)$ for the set $\{f_i(x,a), a\in A_i\} $.
\item{[H1]} For $i=1,\dots,N$, the function $\ell_i: \cP_i\times A_i \to \R$ is a continuous and bounded function. 
There is a modulus of continuity $\omega_\ell$ 
such that for all $i\in \{1, \ldots, N\}$, $x,y\in \cP_i$ and $a\in A_i$, 
\begin{displaymath}
|\ell_i(x,a)-\ell_i(y,a)|\le \omega_{\ell} (|x-y|).
\end{displaymath}
 We denote $M_\ell$ the minimal constant such that for any $i\in \{1, \ldots, N\}$,  $x\in \cP_i$ and $a\in A_i$, 
 \begin{displaymath}
| \ell_i(x,a)| \le M_\ell.
  \end{displaymath}
\item{[H2] }  For $i=1,\dots,N$, $x\in \cP_i$, the non empty and closed  set \[\FL_i(x)= \{ (f_i(x,a), \ell_i(x,a) ) , a\in A_i\} \] is convex. 
\item{[H3] } There is a real number $\delta>0$ such that for any $i=1,\dots,N$ and for all $x\in \Ga$, 
  \begin{displaymath}
    B(0,\delta)\cap \left(\R e_0 \times \R e_i\right) \subset F_i(x).
  \end{displaymath}
  In \S~\ref{section_normal_controllability} below, we will weaken assumption [H3] and use only the assumption on normal controllability
  \item{[$\tH$3] } There is a real number $\delta>0$ such that for any $i=1,\dots,N$ and for all $x\in \Ga$, 
  \begin{displaymath}
    [-\delta,\delta] \subset \left\lbrace f_i(x,a).e_i : a\in A_i \right\rbrace.
  \end{displaymath}
\end{description}
\begin{remark}
In $[\rm{H}0]$ the assumption that the sets $A_i$ are disjoint is  not restrictive: it is made only for simplifying the proof of Theorem \ref{sec:optim-contr-probl} below. The assumption $[\rm{H}2]$ is not essential : it is made in order to avoid the use of relaxed controls.
\end{remark}

Thanks to the Filippov implicit function lemma, see  \cite{MR0208590}, we obtain:
\begin{theorem}\label{sec:optim-contr-probl-1}
  Let $I$ be an interval of $\R$ and $\gamma: I\to \R^d \times \R^d$ be a measurable function. Let $K$ be a closed subset of $\R^d\times  A$ and $\Psi: K\to  \R^d \times \R^d$ be continuous.
Assume that $\gamma(I)\subset \Psi(K)$, then there is a measurable function $\Phi: I\to K$ with 
\begin{displaymath}
  \Psi\circ \Phi(t)=\gamma(t) \quad  \hbox{for a.a. } t\in I.
\end{displaymath}
\end{theorem}
Let $M$ denote the set:
\begin{equation}
  \label{eq:1}
  M=\left\{(x,a);\; x\in \cS,\quad  a\in A_i  \hbox{ if } x\in \cP_i\backslash \Ga,  \hbox{ and } a\in \cup_{i=1}^N A_i   \hbox{ if } x \in \Ga \right\}.
\end{equation}
The set $M$ is closed. We also define the function $f$ on $M$ by 
\begin{equation}\label{def_f}
\forall (x,a)\in M,\quad\quad   f(x, a)=\left\{
    \begin{array}[c]{ll}
      f_i(x,a) \quad &\hbox{ if } x\in \cP_i\backslash \Ga, \\
      f_i(x,a) \quad &\hbox{ if } x\in \Ga \hbox{ and }     a\in A_i.
    \end{array}
\right.
\end{equation}
\begin{remark}\label{rmk: f bounded continuous}
The function $f $ is well defined on $M$, because the sets $A_i$ are disjoint, and is continuous on $M$.
\end{remark}
  Let $\tilde F(x)$ be defined by
\begin{displaymath}
  \tilde F(x)=\left\{ 
\begin{array}{ll}
F_i(x)\quad  &\hbox{if } x \hbox{ belongs to the open half-plane } \cP_i\backslash\Ga \\
\cup_{i=1}^N F_i(x) \quad  &\hbox{if } x\in \Ga.
\end{array}
\right. 
\end{displaymath}
For $x\in \cS$, the set of admissible  trajectories starting from $x$ is 
\begin{equation}
  \label{eq:2}
Y_x=\left\{  y_x\in \rm{Lip}(\R^+; \cS)\;:\left|
    \begin{array}[c]{ll}
       \dot y_x(t)   \in \tilde F(y_x(t))  ,\quad& \hbox{for a.a. } t>0,\\
       y_x(0)=x,
    \end{array}\right.  \right\},
\end{equation}
where $\rm{Lip}(\R^+; \cS)$ is the set of the Lipschitz continuous functions from $\R_+$ to $\cS$.

As in \cite{barles2011bellman} and \cite{uniqueness2013} the following statement holds :
\begin{theorem}\label{sec:optim-contr-probl}
 Assume  $[\rm{H}0]$, $[\rm{H}1]$, $[\rm{H}2]$ and $[\rm{H}3]$. Then 
  \begin{enumerate}
  \item For any $x\in \cS$, $Y_x$ is not empty.
  \item  For any $x\in \cS$, for each trajectory $y_x$ in $Y_x$, there exists a measurable function $\Phi:[0,+\infty)\to M$, $\Phi(t)=(\ph_1(t),\ph_2(t))$ with
    \begin{displaymath}
      (y_x(t),\dot y_x(t))= ( \ph_1(t), f (\ph_1(t),\ph_2(t))) ,\quad \hbox{for a.a. }t,
    \end{displaymath}
which means in particular that $y_x$ is a continuous representation of $\ph_1$.
\item  Almost everywhere on $\{t: y_x(t)\in \Ga\}$, $f(y_x(t), \ph_2(t))\in \R e_0$.
\end{enumerate}
\end{theorem}


We introduce the set of admissible controlled trajectories starting from the initial \mbox{datum $x$ :}
\begin{equation}
  \label{eq:3}
\cT_x=\left\{
 ( y_x, \alpha)  \ds \in L_{\rm{Loc}}^\infty( \R^+; M): \left|
  \begin{array}[c]{l}
\ds  y_x\in \rm{Lip}(\R^+; \cS),  \\  \ds  y_x(t)=x+\int_0^t f( y_x(s), \alpha(s)) ds \quad  \hbox{ in } \R^+ 
  \end{array} \right. \right\}.
\end{equation}

\begin{remark}
  \label{sec:optim-contr-probl-3}
If  two different half-planes are  parallel to each other, say the half-planes $\cP_1$ and $\cP_2$, many other assumptions can be made on the dynamics and costs:
 \begin{itemize}
 \item a trivial case in which the assumptions [H1]-[H3] are satisfied is when the dynamics and costs
are continuous at the origin, i.e.  $A_1 =A_2$;  $f_1$ and $f_2$ are respectively
   the restrictions to $\cP_1\times A_1$ and
 $\cP_2\times A_2$ of a continuous and bounded function $f_{1,2}$ defined in $\R e_0 \times\R e_1 \times A_1$, 
which is Lipschitz continuous with respect to the first variable;  $\ell_1$ and $\ell_2$ 
are respectively the restrictions to $\cP_1\times A_1$ and $\cP_2\times A_2$ of a continuous and bounded 
function $\ell_{1,2}$ defined in $\R e_0 \times\R e_1 \times A_1$.
\item In this particular geometrical setting,  one can allow some mixing (relaxation) at the vertex with several possible rules:  
More precisely, in \cite{barles2011bellman, barles2013bellman}, Barles et al  introduce several kinds of trajectories
which stay at the interface:  the {\sl regular} trajectories are obtained by mixing outgoing dynamics from $\cP_1$ and $\cP_2$, whereas 
{\sl singular} trajectories are obtained by mixing strictly ingoing   dynamics from $\cP_1$ and $\cP_2$.
Two different value functions are obtained whether singular mixing is permitted or not. 
 \end{itemize}
\end{remark}

\paragraph{The cost functional}
The cost associated to the trajectory $ ( y_x, \alpha)\in \cT_x$ is 
\begin{equation}\label{eq: cost functional}
  J(x;( y_x, \alpha) )=\int_0^\infty \ell(y_x(t),\alpha(t)) e^{-\lambda t} dt,
\end{equation}
where $\lambda>0$ is a  real number  and the Lagrangian $\ell$ is defined on $M$ by
\begin{equation}\label{def_l}
\forall (x,a)\in M,\quad\quad   \ell(x, a)=\left\{
    \begin{array}[c]{ll}
      \ell_i(x,a) \quad &\hbox{ if } x\in \cP_i\backslash \Ga, \\
      \ell_i(x,a) \quad &\hbox{ if } x\in \Ga \hbox{ and }     a\in A_i.
    \end{array}\right.
\end{equation}
\paragraph{The value function}
 The value function of the infinite horizon optimal control problem is 
\begin{equation}
  \label{eq:4}
v(x)= \inf_{( y_x, \alpha)\in \cT_x}  J(x;( y_x, \alpha) ).
\end{equation}

\begin{proposition}\label{prop:dpp}
 Assume  $[\rm{H}0]$ and $[\rm{H}1]$. We have the dynamic programming principle:
\begin{equation}\label{DPP}
 \forall t\ge 0, \quad  v(x)=\inf_{( y_x, \alpha)\in \cT_x}\left\{\int_0^t \ell(y_x(s), \a(s))e^{-\lambda s}ds+ e^{-\lambda t}v(y_x(t))\right\}.
\end{equation}
\end{proposition}

\begin{proposition}\label{sec:assumption}
 Assume  $[\rm{H}0]$, $[\rm{H}1]$, $[\rm{H}2]$ and $[\rm{H}3]$. Then the value function $v$ is bounded and continuous on $ \cS$.
\end{proposition}

\mbox{Both propositions above are classical and can be proved with the same arguments as in \cite{MR1484411}.}

\subsection{The Hamilton-Jacobi equation}
\label{sec:hamilt-jacobi-equat}

\subsubsection{Test-functions}\label{sec:test-functions}
To define viscosity solutions on the irregular set $\cS$, it is necessary to first define a class of admissible  test-functions 
\begin{definition}\label{adtest}
A function $\ph: \cS\to \R$ is an admissible test-function if
\begin{itemize}
\item $\ph$ is continuous in $\cS$
\item for any $j$, $j=1,\dots, N$, $\ph|_{\cP_j} \in \cC^1(\cP_j)$.
\end{itemize}
The set of admissible test-functions is denoted by $\cR(\cS)$.
 If $\ph \in \cR(\cS)$, $x\in \cS$ and $\zeta\in T_x(\cS)$,  let $D\ph(x, \zeta)$ be defined by
 
 \begin{displaymath}
D\varphi(x, \zeta) = \left\lbrace
\begin{array}{rll}
D(\varphi|_{\cP_i})(x).\zeta  & \mbox{if} & x\in \cP_i\backslash \Ga,\\
D(\varphi|_{\cP_i})(x).\zeta & \mbox{if} & x\in \Ga \mbox{ and } \zeta \in \R e_0 \times \R e_i,
\end{array}
\right.
\end{displaymath}
 
where for $u$, $v\in \R e_0\times\R e_i$, $u.v$ denotes the usual Euclidean scalar product in $\R e_0 \times \R e_i$ and for $x\in \Ga$, $D\left( \ph|_{\cP_i}\right) (x).\xi$ is defined by 
\begin{equation}\label{eq_def: D_phi}
D\left( \ph|_{\cP_i}\right) (x).\xi =\lim_{y\to x, y\in \cP_i\backslash \Ga} D\left( \ph|_{\cP_i}\right) (y).\xi.
\end{equation}
If $\ph \in \cR(\cS)$, $x\in \Ga$ and $\xi \in \R e_0$, we will use also the notations $D(\ph|_{\Ga})(x).\xi$ for the differential of $\ph|_{\Ga}$ at the point $x$ evaluated in $\xi$.\\
Other notations which will be useful are the following: for $j\in \{1, \ldots , N \}$, $x\in \cP_j$ and $\phi\in \cR(\cS)$
\begin{equation}\label{eq_def: D_phi2}
\partial_{x_j}\left(\phi|_{\cP_j}\right)(x)=
\left\lbrace\begin{array}{ll}
\lim_{h\to 0}\frac{\phi(x+h e_j)-\phi(x)}{h} & \mbox{ if } x\in \cP_i\setminus\Ga,\\
\lim_{\overset{h>0}{h\to 0}}\frac{\phi(x+h e_j)-\phi(x)}{h} & \mbox{ if } x\in \Ga,
\end{array}
\right.
\end{equation}
 and for $x\in \Ga$
 \begin{equation}\label{eq_def: D_phi3}
 \partial_{x_0}\phi(x)=\lim_{h\to 0}\frac{\phi(x+h e_0)-\phi(x)}{h}.
 \end{equation}

 \end{definition}
 \begin{remark}\label{rmk: D_phi_Ga}
 If $\ph \in \cR(\cS)$, $x\in \Ga$ and $\xi \in \R $, then for all $i\in \{1,\ldots,N\}$,
\begin{displaymath}
D(\ph|_{\cP_i})(x).\xi e_0=D(\ph|_{\Ga})(x).\xi e_0= \partial_{x_0}\phi(x) \xi.
\end{displaymath}
Particularly, the tangential component of $D(\ph|_{\cP_i})(x)$ is independent of $i\in \{1,\ldots,N\}$.
 \end{remark}
 \begin{property}\label{composition}
If $\ph=g\circ\psi$ with $g\in \cC^1(\R)$ and $\psi\in\cR (\cS)$, then $\ph\in\cR (\cS)$ and for any $x\in \cS$, $\zeta\in T_x(\cS)$
\begin{equation*}
D\ph(x,\zeta)=g^\prime(\psi(x))D\psi(x,\zeta).
\end{equation*}
\end{property}

\subsubsection{Vector fields}\label{sec:prel-defin}
 For $i=1,\dots,N$ and $x$ in $\Ga$, we denote by $F_i^+(x)$ and $\FL_i^+(x)$ the sets
\begin{displaymath}
  F_i^+(x)=F_i(x)\cap \left( \R e_0\times \R^+e_i \right) ,\quad\quad   \FL_i^+(x)=\FL_i(x)\cap \left( (\R e_0\times \R^+e_i )\times \R\right),
\end{displaymath}
which are non empty thanks to assumption $[\rm{H}3]$. Note that $0\in \cap _{i=1}^N F_i(x)$. From assumption $[\rm{H}2]$, these sets are compact and convex. 
For $x\in \cS$, the sets $F(x)$ and $\FL(x)$ are defined by
\begin{displaymath}
  F(x)=\left\{ 
\begin{array}{ll}
F_i(x)\quad  &\hbox{if } x \hbox{ belongs to } \cP_i\backslash\Ga \\
\cup_{i=1,\dots,N} F_i^+(x) \quad  &\hbox{if }x \in  \Ga
\end{array} 
\right.
\end{displaymath}
and 
\begin{displaymath}
     \FL(x)=\left\{ 
\begin{array}{ll}
\FL_i(x)\quad  &\hbox{if } x \hbox{ belongs to } \cP_i\backslash\Ga \\
\cup_{i=1,\dots,N} \FL_i^+(x) \quad  &\hbox{if }x \in \Ga.
\end{array}\right.
\end{displaymath}

\subsubsection{Definition of viscosity solutions}\label{sec:definitiona}
We now introduce the definition of  a viscosity solution of
\begin{equation}\label{HJa}
    \lambda u(x)+\sup_{(\zeta,\xi)\in \FL(x)}\{- Du(x, \zeta) -\xi\}=0 \quad \hbox{in } \cS.
\end{equation}

\begin{definition}\label{netviscoa}
\begin{itemize}
\item An upper semi-continuous function $u:\cS\to\R$ is a subsolution of \eqref{HJa} in $\cS$
 if for any $x\in\cS$, any $\ph\in\cR(\cS)$ s.t. $u-\ph$ has a local maximum point at $x$, then
 \begin{equation}
\label{eq:5}
\lambda u(x)+\sup_{ (\zeta,\xi) \in \FL(x)}\{-D\ph(x,\zeta) -\xi\}\le 0.
 \end{equation}
\item  A lower semi-continuous function $u:\cS\to\R$ is a  supersolution of \eqref{HJa} if for any $x\in\cS$,
 any $\ph\in\cR(\cS)$ s.t. $u-\ph$ has a local minimum point at $x$, then
\begin{equation}
  \label{eq:6}
\lambda u(x)+\sup_{ (\zeta,\xi) \in \FL(x)}\{-D\ph(x,\zeta) -\xi\}\ge 0.
\end{equation}
  \item A continuous function $u:\cS\to\R$ is a   viscosity solution of \eqref{HJa} 
in  $\cS$
 if it is both a viscosity subsolution  and a viscosity supersolution of \eqref{HJa} in $\cS$.
\end{itemize}
\end{definition}

\begin{remark}\label{IKsola}
At $x\in \cP_i\backslash\Ga$, the notion of sub, respectively  super-solution
  in Definition \ref{netviscoa} 
 is equivalent to the standard definition of viscosity   sub, respectively  super-solution of 
\[\lambda u(x)+\sup_{a\in A_i}\{-f_i(x,a)\cdot Du(x)-\ell_i(x,a)\}= 0.\]
\end{remark}

\subsubsection{Hamiltonians}
We define the Hamiltonian $H_i: \cP_i\times(\R e_0 \times \R e_i) \to \R$  by
\begin{equation}
  \label{eq:7}
H_i(x,p)= \max_{a\in A_i} (-p. f_i(x,a) -\ell_i(x,a)),
\end{equation} 
and the Hamiltonian $H_{\Ga}: \Ga \times \left(\prod_{i=1, \dots , N}(\R e_0 \times \R e_i) \right) \to \R$ by 
\begin{equation}
  \label{eq:8}
H_{\Ga}(x, p_1,\dots,p_N)=  \max_{i=1,\dots,N} \;H_i^+(x,p_i),
\end{equation} 
where the Hamiltonian $H^+_i: \cP_i\times(\R e_0 \times \R e_i) \to \R$ is defined by
\begin{equation}
  \label{eq:7bis}
H^+_i(x,p)= \max_{a\in A_i \hbox{ s.t. }  f_i(x,a).e_i \geq 0} (-p. f_i(x,a) -\ell_i(x,a)).
\end{equation} 
We also define what may be called the tangential Hamiltonian at $\Ga$,  \mbox{$H_{\Ga}^T: \Ga \times \R e_0 \to \R$,} by 
\begin{equation}
  \label{eq:9}
H_{\Ga}^T(x,p)=   \max_{i=1,\dots,N} \;H_{\Ga,i}^T(x,p),
\end{equation}
where the Hamiltonian $H_{\Ga,i}^T: \Ga \times \R e_0 \to \R$ is defined by
\begin{equation}
  \label{eq:9bis}
H_{\Ga,i}^T(x,p)=  \max_{a\in A_i \hbox{ s.t. }  f_i(x,a).e_i=0} ( -f_i(x,a).p-\ell_i(x,a)).
\end{equation}

The following  definitions are equivalent to Definition \ref{netviscoa}:
\begin{definition}\label{netviscoa2}
\begin{itemize}
\item An upper semi-continuous function $u:\cS\to\R$ is a subsolution of \eqref{HJa} in $\cS$
 if for any $x\in\cS$, any $\ph\in\cR(\cS)$ s.t. $u-\ph$ has a local maximum point at $x$, then
 \begin{equation}
\label{eq:10}
\begin{array}[c]{ll}
\lambda u(x)+ H_i(x, D\left( \ph|_{\cP_i}\right) (x))  \le 0  \quad &\hbox{if } x\in \cP_i\backslash \Ga,\\
\lambda u(x)+ H_\Ga\left( x, D\left( \ph|_{\cP_1}\right) (x), \dots, D\left( \ph|_{\cP_N}\right) (x) \right)  \le 0  \quad &\hbox{if } x\in \Ga.
\end{array}
 \end{equation}
\item  A  lower semi-continuous function $u:\cS\to\R$ is a  supersolution of \eqref{HJa} if for any $x\in\cS$,
 any $\ph\in\cR(\cS)$ s.t. $u-\ph$ has a local minimum point at $x$, then
\begin{equation}\label{eq:11}
\begin{array}[c]{ll}
\lambda u(x)+ H_i(x, D\left( \ph|_{\cP_i}\right) (x))   \ge 0  \quad &\hbox{if } x\in \cP_i\backslash \Ga,\\
\lambda u(x)+ H_\Ga\left( x, D\left( \ph|_{\cP_1}\right) (x), \dots, D\left( \ph|_{\cP_N}\right) (x) \right)   \ge 0  \quad &\hbox{if } x\in \Ga.
\end{array}
\end{equation}
\end{itemize}
\end{definition}
The Hamiltonian $H_i$ are continuous with respect to $x\in \cP_i$, convex with respect to $p$.
Moreover, if $x$ belongs to $\Ga$, the function $p\mapsto H_i(x,p)$ is coercive,
 i.e. $\lim_{|p|\to +\infty } H_i(x,p)=+\infty$ from the controllability assumption [H3].
 \\
 The following lemma is the counterpart of Lemma 2.1 in \cite{uniqueness2013}.
 \begin{lemma} \label{sec:hamiltonians}
 Assume  $[\rm{H}0]$, $[\rm{H}1]$, $[\rm{H}2]$ and $[\rm{H}3]$. Take $i \in \{1, \dots, N \}$, $x \in\Ga$ and $p\in \R e_0 \times \R e_i$. Let $\ph_{i,x,p}: \R \to \R$ be the function defined by  $\ph_{i,x,p}(\delta)=H_i(x,p+\delta e_i)$. We denote by $\Delta_{i,x,p}$ the set 
    \begin{equation}
   \label{eq:12}
   \Delta_{i,x,p}= \{ \delta \in \R \hbox{  s.t.  } \ph_{i,x,p}(\delta)= \min_{d\in \R}  (\ph_{i,x,p}(d))  \}.
 \end{equation}
   \begin{enumerate}
   \item The set $\Delta_{i,x,p}$ is not empty.
   \item $\delta\in  \Delta_{i,x,p}$ if and only if  there exists $a^*\in A_i$ such that $f_i(x,a^*).e_i=0$ and\\
   $\ph_{i,x,p}(\delta) = - f_i(x,a^*).p-\ell_i(x,a^*)$.
   \item For any $x\in \Ga$, $p=p_0e_0+p_ie_i$, with $p_0,p_i\in \R$ and $\delta\in  \Delta_{i,x,p}$ we have
     \begin{equation}
 H_i(x,p+\delta e_i)= H_i^+(x,p+\delta e_i)=  H_{\Ga,i}^T(x,p_0e_0).
     \end{equation}
     \item For any $x\in \Ga$, $p=p_0e_0+p_ie_i$, with $p_0,p_i\in \R$ and $\delta\ge \min\{ d : d\in \Delta_{i,x,p}\}$ we have
     \begin{equation}
 H_i^+(x,p+\delta e_i)=  H_{\Ga,i}^T(x,p_0e_0).
     \end{equation}
   \end{enumerate}
 \end{lemma}
  \begin{proof}
Point 1 is easy, because the Hamiltonian $H_i$ is continuous and coercive with respect to $p$. The function $\ph_{i,x,p}$ reaches its minimum at $\delta$ if and only if $0\in \partial\ph_{i,x,p}(\delta)$. The subdifferential of $\ph_{i,x,p}$ at $\delta$ is characterized by
   \begin{displaymath}
     \partial \ph_{i,x,p}(\delta)=\overline{\rm co}  \{ -f_i(x,a).e_i ; a\in A_i \hbox { s.t. }  \ph_{i,x,p}(\delta)= -f_i(x,a).(p+\delta e_i) -\ell_i(x,a)    \},
   \end{displaymath}
see \cite{MR0241975}.
But from [H2], 
\begin{displaymath}
   \{ (f_i(x,a),\ell_i(x,a)) ; a\in A_i \hbox { s.t. }  \ph_{i,x,p}(\delta)= -f_i(x,a).(p+\delta e_i) -\ell_i(x,a)    \}
\end{displaymath}
is compact and convex. Hence, 
 \begin{displaymath}
   \partial \ph_{i,x,p}(\delta)= \{ -f_i(x,a).e_i ; a\in A_i \hbox { s.t. }  \ph_{i,x,p}(\delta)= -f_i(x,a).(p+\delta e_i) -\ell_i(x,a)    \}.
   \end{displaymath}
Therefore, $0\in \partial \ph_{i,x,p}(\delta)$ if and only if there exists $a^*\in A_i$ such that $f_i(x,a^*).e_i=0$ and $\ph_{i,x,p}(\delta)=   -f_i(x,a^*).(p+\delta e_i) -\ell_i(x,a^*)$. We have proved point 2.
\\
Points 3 is a direct consequence of point 2. Point 4 is a consequence of point 3 and of the decreasing character of the function $d\longmapsto H_i^+(x,p+d e_i)$.
 \end{proof}
 
 \begin{remark}
 The conclusions of Lemma \ref{sec:hamiltonians} hold if we replace $[\rm{H}3]$ with $[\tH3]$. Indeed, we actually just need that the Hamiltonian $H_i$ be coercive with respect to $p_i$, where $p_i=p.e_i$.
 \end{remark}
 
 \subsubsection{Existence}\label{sec:existence-uniqueness-1}
\begin{theorem}
  \label{sec:existence-uniqueness-3}
 Assume  $[\rm{H}0]$, $[\rm{H}1]$, $[\rm{H}2]$ and $[\rm{H}3]$.
The value function $v$ defined in (\ref{eq:4}) is a  bounded  viscosity solution of \eqref{HJa} in $ \cS$.
\end{theorem}
The proof of Theorem \ref{sec:existence-uniqueness-3} is made in several steps: the first step consists of proving that the value function is a viscosity solution of a Hamilton-Jacobi 
equation with a more general definition of the Hamiltonian: for that, we introduce larger relaxed vector fields:
for $x\in  \cS$,
\begin{displaymath}
  \begin{split}
    \tilde f(x)=\left\{ \eta\in \R^d :
      \begin{array}[c]{l}
        \exists ( y_{x,n},\alpha_n)_{n\in \N}, \\( y_{x,n},\alpha_n)\in \cT_x, \\
\exists (t_n)_{n\in \N}
      \end{array} \;{\rm s.t.}\;
      \left|
      \begin{array}[c]{l}
       \ds  t_n\to 0^+
 \;{\rm and}\;\\ \ds \lim_{n\to \infty} \frac{1}{t_n}\int_0^{t_n}f(y_{x,n} (t),\a_n(t))dt=\eta
      \end{array} \right. \right \}
  \end{split}
\end{displaymath}
and
\begin{displaymath}
  \begin{split}
  &\fl (x)=
\\ &\left\{ (\eta,\mu)\in \R^d \times \R:
      \begin{array}[c]{l}
        \exists ( y_{x,n},\alpha_n)_{n\in \N}, \\ ( y_{x,n},\alpha_n)\in \cT_x, \\
\exists (t_n)_{n\in \N}
      \end{array} \;{\rm s.t.}\;\left|
      \begin{array}[c]{l}
       \ds  t_n\to 0^+,\\
 \ds \lim_{n\to \infty} \frac{1}{t_n}\int_0^{t_n}f(y_{x,n}(t),\a_n(t))dt=\eta,
 \\ \ds \lim_{n\to \infty} \frac{1}{t_n}\int_0^{t_n}\ell(y_{x,n}(t),\a_n(t))dt=\mu
      \end{array} \right. \right \};
  \end{split}
\end{displaymath}
where $\cT_x$ is the set of admissible controlled trajectories starting from the point $x$ which was introduced in (\ref{eq:3}).

\begin{proposition}
  \label{sec:existence}
Assume  $[\rm{H}0]$, $[\rm{H}1]$, $[\rm{H}2]$ and $[\rm{H}3]$. The value function $v$ defined in (\ref{eq:4}) is a   viscosity solution of 
\begin{equation}\label{HJa2}
    \lambda u(x)+\sup_{(\zeta,\xi)\in \fl(x)}\{- Du(x, \zeta) -\xi\}=0 \quad \hbox{in } \cS,
\end{equation}
where the definition of viscosity solution is exactly the same as in Definition \ref{netviscoa}, replacing $\FL(x)$ with $\fl(x)$.
\end{proposition}
\begin{proof}
See \cite{MR3057137}. 
\end{proof}

The second step consists of proving that for any $\ph\in \cR(\cS)$ and $x\in \cS$, $\sup_{(\zeta,\xi)\in \FL(x)}\{- D\ph(x, \zeta) -\xi\}$ and $\sup_{(\zeta,\xi)\in \fl(x)}\{- D\ph(x, \zeta) -\xi\}$ are equal. This is a consequence of the following lemma.
\begin{lemma}\label{sec:existence-2} For any $x\in \cS$,
  \begin{displaymath}
  \begin{array}[c]{lll}
      \fl(x)= &  \FL(x) & \hbox{if } x\in \cS\backslash \Ga,\\
       \fl(x)= &  \bigcup_{i=1,\dots,N}  \overline{\rm co} \left\{ \FL_i^+(x) \cup \bigcup_{j\not=i} \Bigl(\FL_j(x)\cap (\R e_0\times \R) \Bigr)     \right\} & \hbox{if } x\in \Ga.
  \end{array}
  \end{displaymath}
\end{lemma}
\begin{proof}
  The proof being a bit long, we postpone it to the appendix \ref{sec: appendix1}.
\end{proof}

\begin{lemma}\label{sec:existence-1}
Assume  $[\rm{H}0]$, $[\rm{H}1]$, $[\rm{H}2]$ and $[\rm{H}3]$.
  For any function $\ph\in\cR(\cS)$ and $x\in \cS$,
  \begin{equation}
    \label{eq:14}
    \sup_{(\zeta,\xi)\in \fl(x)}\{- D\ph(x, \zeta) -\xi\}= \max_{(\zeta,\xi)\in \FL(x)}\{- D\ph(x, \zeta) -\xi\}.
  \end{equation}
\end{lemma}

\begin{proof}
For $x\in\cS \backslash \Ga$ there is nothing to prove because $\FL(x)=\fl(x)$. If $x\in\Ga$ we can prove that  $\FL(x)\subset \fl(x)$ for any $x\in\Ga$, in the same way as in \cite{MR3057137}. Hence
\begin{displaymath}
  \max_{(\zeta,\xi)\in \FL(x)}\{- D\ph(x, \zeta) -\xi\} \le \sup_{(\zeta,\xi)\in \fl(x)}\{- D\ph(x, \zeta) -\xi\}.
\end{displaymath}
\\
 From the piecewise linearity of the function $(\zeta, \mu)\mapsto -D\ph(x, \zeta) -\mu$, we infer that
  \begin{displaymath}
    \begin{array}[c]{l}
      \ds   \sup_{(\zeta,\mu) \in    \overline{\rm co} \left\{ \FL_i^+(x) \cup \bigcup_{j\not=i} \Bigl(\FL_j(x)\cap (\R e_0 \times \R) \Bigr) \right\}} (   -D\ph(x, \zeta) -\mu) \\ \ds=
      \max\left( \max_{(\zeta,\mu)\in \FL_i^+(x) }   (   -D\ph(x, \zeta) -\mu), 
        \max_{j\not=i} \left( \max_{(\zeta,\mu)\in \FL_j(x)\cap (\R e_0 \times \R) }   (   -D\ph(x, \zeta) -\mu) \right) \right)
      \\
      \le \max_{j=1,\dots,N} \max_{ (\zeta,\mu)\in \FL^+_j(x)} ( -D\ph(x, \zeta) -\mu) =  \max_{(\zeta,\xi)\in \FL(x)}\{- D\ph(x, \zeta) -\xi\}.
    \end{array}
  \end{displaymath}
  We conclude by using Lemma \ref{sec:existence-2}.
\end{proof}

\subsection{Properties of viscosity sub and supersolutions}
\label{sec:properties-sub-super}
In this part, we study sub and supersolutions of (\ref{HJa}), 
transposing
 ideas coming from Barles-Briani-Chasseigne \cite{barles2011bellman,barles2013bellman} 
to the present context.

\begin{property}
\label{prop: R_controllability}
Assume  $[\rm{H}0]$ and $[\rm{H}3]$. Then, there exists $R>0$ a positive real number such that  for all $i=1,\dots,N$ and  $x\in B(\Ga,R)\cap \cP_i$
\begin{equation}
\label{eq: R_controllability}
    B(x,\frac{\delta}{2})\cap (\R e_0 \times \R e_i) \subset F_i(x).
\end{equation}
\end{property}
\begin{remark}
This property means that the controlabillity assumption $[\rm{H}3]$, which focuses on $\Ga$, holds in a neighborhood of $\Ga$ thanks to the continuity properties of the functions $f_i$, $i\in \{1, \ldots,N\}$.
\end{remark}

\begin{lemma} \label{sec:prop-visc-sub-4}
 Assume  $[\rm{H}0]$, $[\rm{H}1]$, $[\rm{H}2]$ and $[\rm{H}3]$.
Let  $R>0$ be as in \eqref{eq: R_controllability}.
For any bounded viscosity subsolution $u$ of
  (\ref{HJa}), there exists a constant $C^*>0$ such that $u$ is a viscosity subsolution of
 \begin{equation}\label{eq: estimation_L_infty}
     | Du(x)  | \le C^*  \quad \hbox{in } B(\Ga,R) \cap \cS,
\end{equation}
i.e. for any $x\in B(\Ga,R)\cap \cS$ and  $\phi\in \cR(\cS)$ such that $u-\phi$ has a local
maximum point at $x$, 
\begin{eqnarray} \label{eq:36}
  |D(\phi |_{\cP_i}) (x)| \le C^*\quad & \hbox{ if } x\in   (B(\Ga,R)  \cap
  \cP_i) \backslash \Ga,\\
\label{eq:40}
\max_{i=1,\ldots,N}\left\lbrace|\partial_{x_0}\phi (x)|+\left[\partial_{x_i}(\phi |_{\cP_i}) (x)\right]_-\right\rbrace \le  C^* \quad & \hbox{ if } x\in \Ga,
\end{eqnarray}
where $\partial_{x_j}\left(\phi|_{\cP_j}\right)(x)$ and  $\partial_{x_0}\phi(x)$ are defined in \eqref{eq_def: D_phi2} and \eqref{eq_def: D_phi3} and $[ \cdot ]_-$ denote the negative part function, i.e.  for $x\in \R$, $[x]_-=\max\{0,-x\}$.
\end{lemma}
\begin{proof}
Let  $M_u$ (resp $ M_\ell$) be an upper bound on  $|u|$ (resp. $\ell_j$ for all
$j=1,\dots,N$).  The viscosity inequality (\ref{eq:10}) yields that 
\begin{eqnarray}
  \label{eq:37}
H_i(x, D(\phi |_{\cP_i}) (x))  &\le \lambda M_u  \quad     &\hbox{if } x  \in (B(\Ga,R)  \cap
  \cP_i) \backslash \Ga,\\
\label{eq:38}
 H_\Ga( x,D(\phi |_{\cP_1}) (x),\dots,D(\phi |_{\cP_N}) (x) ) & \le\lambda M_u  &\hbox{if } x\in \Ga.
\end{eqnarray}
From the controllability in $B(\Ga,R)\cap \cP_i$, we see that $H_i$ is coercive with
respect to its second argument uniformly in $x\in B(\Ga,R) \cap \cP_i$. More
precisely we have that
$H_i(x,p)\ge \frac \delta 2 |p| -M_\ell$.
\\
 Thus, if  $x \in
(B(\Ga,R) \cap \cP_i) \backslash
\Ga $, from (\ref{eq:37}), we have $ |D(\phi |_{\cP_i}) (x)| \le C^*$ with $C^*= 2 \frac {\lambda M_u +M_\ell} \delta$. \\
 If $x\in \Ga$, considering the controls $a^+_0,a^-_0, a_i^+$ and $a^0_i\in A_i$ such that $f_i(x,a^{\pm}_0)= \pm\delta e_0$, $f_i(x,a^+_i)= \delta e_i$ and $f_i(x,a^0_i)=0$ we can prove that $ H_i^+(x,p_0 e_0+p_ie_i) \ge  \frac \delta 2 (|p_0|+[p_i]_-)
 -M_\ell$.
 Then, the viscosity inequality (\ref{eq:38})  yields \eqref{eq:40} with the same constant $C^\star= 2 \frac {\lambda M_u +M_\ell} \delta$. 
\end{proof}

The following lemma gives us an explicit expression for the geodesic distance which will be convenient in future calculations.

\begin{lemma}\label{lem: geodesic_distance}
Let $x=x_0 e_0 +x_i e_i\in \cP_i$ and  $y=y_0 e_0 +y_j e_j\in \cP_j$. Then,
\begin{equation}\label{eq:  lem_geodesic_distance}
 d(x,y)=\left\{  \begin{array}[c]{ll}
\left[ (x_0-y_0)^2+(x_i-y_i)^2\right]^{\frac{1}{2}}\quad \hbox{if } i=j, \\
\left[ (x_0-y_0)^2+(x_i+y_j)^2\right]^{\frac{1}{2}}  \quad \hbox{if }  i\neq j.
  \end{array}\right.
\end{equation}
\end{lemma}

\begin{lemma}
   \label{sec:prop-visc-sub-3}
  Assume  $[\rm{H}0]$, $[\rm{H}1]$, $[\rm{H}2]$ and $[\rm{H}3]$. Any bounded  viscosity subsolution $u$ of (\ref{HJa})  is Lipschitz continuous in a neighborhood of $\Ga$, i.e there exists some strictly positive number $r$ such that $u$ is Lipschiz continuous on $B(\Ga, r)\cap \cS$, where $B(\Ga, r)$ denotes the set $\{y\in \R^d: dist(y, \Ga)<r\}$.
 \end{lemma}
\begin{proof}
We adapt the proof of H.Ishii, see \cite{ishii2013short}.\\
Take $R$ as in \eqref{eq: R_controllability},  fix $z\in B(\Ga,R)\cap \cS $
and set $r= (R- dist(z,\Ga))/4$. Fix any $y\in \cS$ such that $d(y,z)<r$.  It can be
checked that for any $x\in \cS$, if $d(x,y)<3r$ then $dist(x,\Ga)< R$. Choose a
function $f\in \cC^1([0,3r))$ such  that $f(t)=t$ in $[0,2r]$, $f'(t)\ge 1$
for all $t\in [0,3r)$ and $\lim_{t\to 3r} f(t)=+\infty$. Fix any
$\epsilon>0$. We are going to show that
\begin{equation}
  \label{eq:39}
u(x)\le u(y)+(C^*+\epsilon) f(d(x,y)),\quad \forall x\in \cS \hbox{ such that } d(x,y)<3r,
\end{equation}
where $C^*$ is the constant in Lemma~\ref{sec:prop-visc-sub-4}. Let us proceed by contradiction. Assume that \eqref{eq:39} is not true. According to the properties of $f$, the function $x\mapsto u(x)- u(y)-(C^*+\epsilon) f(d(x,y))$ admits a maximum $\xi \in B(y,3r)\cap \cS$. However, since we assumed that \eqref{eq:39} is not true, necessarily $\xi \neq y$. Consequently, the function $\psi : \cS \to \R$, $x\longmapsto (C^\star + \epsilon)f(d(x,y))$ is an appropriate test function in a neighborhood of $\xi$ which can be used as test function in the viscosity inequality \eqref{eq: estimation_L_infty}, satisfied by $u$ from Lemma \ref{sec:prop-visc-sub-4}.
For the calculations, we need to distinguish several cases. Assume that $y=y_0 e_0 +y_i e_i \in \cP_i$.
\begin{enumerate}
\item \textbf{If $\xi=\xi_0 e_0 + \xi_i e_i \in \cP_i \setminus \Ga$ :} i.e. $\xi$ and $y$ belong to the same half-plane $\cP_i$. Then, from \eqref{eq:  lem_geodesic_distance} in Lemma \ref{lem: geodesic_distance}, we have $D(\psi|_{\cP_i})(\xi)=(C^\star + \epsilon)f'(d(\xi,y))\frac{\xi-y}{d(\xi,y)}$ and \eqref{eq:36} in Lemma \ref{sec:prop-visc-sub-4} gives us 
\begin{equation}\label{eq: proof_u_lipschitz1}
(C^*+\epsilon) f'(d(\xi, y))\frac{|\xi-y|}{d(\xi,y)} \le C^*.
\end{equation}
Since $\frac{|\xi-y|}{d(\xi,y)}=1$ and $f'(t)\ge 1$
for all $t\in [0,3r)$, \eqref{eq: proof_u_lipschitz1} leads to a contradiction.
\item \textbf{If $\xi=\xi_0 e_0 + \xi_j e_j \in \cP_j \setminus \Ga$ with $j\neq i$ :} i.e. $\xi$ and $y$ belong to different half-planes, respectively $\cP_j$ and $\cP_i$. Then, from \eqref{eq:  lem_geodesic_distance} in Lemma \ref{lem: geodesic_distance}, we have $D(\psi|_{\cP_j})(\xi)=(C^\star + \epsilon)f'(d(\xi,y))\frac{(\xi_0-y_0)e_0+(\xi_j+y_i)e_j}{d(\xi,y)}$ and \eqref{eq:36} in Lemma \ref{sec:prop-visc-sub-4} gives us 
\begin{equation}\label{eq: proof_u_lipschitz2}
(C^*+\epsilon) f'(d(\xi, y))\frac{|(\xi_0-y_0)e_0+(\xi_j+y_i)e_j|}{d(\xi,y)} \le C^*.
\end{equation}
Since $\frac{|(\xi_0-y_0)e_0+(\xi_j+y_i)e_j|}{d(\xi,y)}=1$ and $f'(t)\ge 1$
for all $t\in [0,3r)$, \eqref{eq: proof_u_lipschitz2} leads to a contradiction.
\item \textbf{If $\xi=\xi_0 e_0  \in  \Ga$ :} In this case, $\xi$ and $y$ belong to the same half-plane, but we have to deal with \eqref{eq:40} in Lemma \ref{sec:prop-visc-sub-4}. For $i\in \{1,\ldots,N\}$ be such that $y\in \cP_i$, from \eqref{eq:  lem_geodesic_distance} in Lemma \ref{lem: geodesic_distance} we have 
\begin{displaymath}
\partial_{x_i}\left( \psi|_{\cP_i}\right)(\xi)=(C^*+\epsilon)f'(d(\xi,y))\frac{-y_i}{d(\xi,y)} \le 0,
\end{displaymath}
and
\begin{displaymath}
\partial_{x_0}\psi(\xi)=(C^*+\epsilon)g'(d(\xi,y))\frac{\xi_0-y_0}{d(\xi,y)}.
\end{displaymath}
Then, \eqref{eq:40} in Lemma \ref{sec:prop-visc-sub-4} gives us
\begin{equation}\label{eq: proof_u_lipschitz3}
(C^*+\epsilon) f'(d(\xi, y))\frac{|\xi_0-y_0|+y_i}{d(\xi,y)} \le C^*.
\end{equation}
Since $\frac{|\xi_0-y_0|+y_i}{d(\xi,y)}\ge 1$ and $f'(t)\ge 1$
for all $t\in [0,3r)$, \eqref{eq: proof_u_lipschitz3} leads to a contradiction.
\end{enumerate}
This conclude the proof of \eqref{eq:39}.
Remark that if  $d(x,z)<r$ then $d(x,y)<2r$ and $f(d(x,y))=d(x,y)$. Then, (\ref{eq:39}) yields that 
\begin{displaymath}
  u(x)\le u(y )+(C^*+\epsilon) d(x,y),\quad \forall x,y \in \cS \hbox{ s.t }
  d(x,z)<r , \; d(y,z)<r.
\end{displaymath}
By symmetry, we get
\begin{displaymath}
  |u(x)-u(y)| \le (C^*+\epsilon) d(x,y),\quad \forall x,y \in \cS \hbox{ s.t }
  d(x,z)<r , \; d(y,z)<r,
\end{displaymath}
and by letting $\epsilon$ tend to zero, we get
\begin{equation}
  \label{eq:41}
|u(x)-u(y)| \le C^* d(x,y),\quad \forall x,y \in \cS \hbox{ s.t }
  d(x,z)<r , \; d(y,z)<r.
\end{equation}
Now, for two arbitrary points $x,y$ in $\cS\cap B(\Ga,R)$, we take $r=\frac 1 4
\min( R-dist(x,\Ga), R-dist(y,\Ga))$ and choose a finite sequence $(z_j)_{j=1,\dots, M}\in
\cG$ 
belonging to the geodesic between $x$ and $y$, such that $z_1=x$, $z_M=y$,
$d(z_i,z_{i+1})<r$ for all $i=1,\dots, M-1$ and $\sum_{i=1}^{M-1}
d(z_i,z_{i+1}) =d(x,y)$. From \eqref{eq:41}, we get that 
\begin{displaymath}
  |u(x)-u(y)| \le C^* d(x,y),\quad \forall x,y \in \cS\cap B(\Ga,R).
\end{displaymath}
\end{proof} 
 
 An important consequence of this lemma is the following result.
 
 \begin{lemma}\label{lem subsol inequality on gamma}
   Assume  $[\rm{H}0]$, $[\rm{H}1]$, $[\rm{H}2]$ and $[\rm{H}3]$. Let $u$ be a bounded viscosity subsolution of (\ref{HJa}) and $\ph: \Ga \to \R$ a $C^1$-function. Then, for any local maximum point $\bar x$ in $\Ga$ of $u-\ph$ on $\Ga$, one has
    \begin{displaymath}
   \lambda u(\bar x) + H_{\Ga}^T(\bar x,D(\ph|_{\Ga})(\bar x)) \le 0.
 \end{displaymath}
\end{lemma}
\begin{proof}
Let $\ph : \Ga\to \R$ be a $C^1$-function  and $\bar x\in \Ga$ be a local maximum point of $(u-\ph)\
|_\Ga$ on $\Ga$. Since $u$ is a subsolution of (\ref{HJa}), according to Lemma \ref{sec:prop-visc-sub-3}, the function $u$ is Lipschitz continuous on $B(\Ga,r)$ for some positive real $r$, with a Lipschitz constant $L_{u,r}$. We introduce the function $\overline{\ph}$ defined on $\cS$ by $\overline{\ph}(x_0 e_0+x_i e_i)=\ph(x_0e_0)+L_{u,r}x_i$.
 By construction, the function $\overline{\ph}$ belongs to $\cR(\cS)$ and $u-\overline{\ph}$ admits a local maximum at the point $\bar x$ on $\cS$. Then, since $u$ is a viscosity subsolution of (\ref{HJa}), we see that 
\begin{displaymath}
 \lambda u(\bar x) + H_{\Ga}(\bar x,D(\overline{\ph}|_{\cP_1})(\bar x),\ldots,D(\overline{\ph}|_{\cP_N})(\bar x)) \le 0,
 \end{displaymath}
  which implies that $\lambda u(\bar x) + H_{\Ga}^T(\bar x,D(\ph|_{\Ga})(\bar x)) \le 0$.
\end{proof}

 \begin{remark}\label{rmk subsol inequality on gamma}
 The conclusion of Lemma \ref{lem subsol inequality on gamma} holds if we replace $[\rm{H}3]$ with the assumption that the subsolution $u$ is Lipschitz continous.
 \end{remark}
 
 The following lemma can be found in \cite{barles2011bellman,barles2013bellman} in a different context:
\begin{lemma}\label{sec:prop-visc-sub-1}
 Assume  $[\rm{H}0]$, $[\rm{H}1]$, $[\rm{H}2]$ and $[\rm{H}3]$.  Let $v:\cS\to \R$ be a viscosity supersolution of (\ref{HJa}) in $\cS$ and $w:\cS \to \R$ be a continuous viscosity subsolution of (\ref{HJa}) in $\cS$.
Then if $x\in \cP_i\backslash \Ga$, we have for all $t>0$, 
\begin{equation}
  \label{eq:16}
v(x)\ge \inf_{\alpha_i(\cdot), \theta_i}\left( \int_0^{t\wedge \theta_i}   \ell_i(y_x^i(s),  \alpha_i(s)) e^{-\lambda s} ds + v (y^i_x(t\wedge \theta_i)) e^{-\lambda (t\wedge \theta_i)} \right),
\end{equation}
and 
\begin{equation}
\label{eq:17}
w(x)\le \inf_{\alpha_i(\cdot), \theta_i}\left( \int_0^{t\wedge \theta_i}   \ell_i(y_x^i(s),  \alpha_i(s)) e^{-\lambda s} ds + w (y^i_x(t\wedge \theta_i)) e^{-\lambda (t\wedge \theta_i)} \right),
\end{equation}
where $\alpha_i\in L^\infty(0,\infty; A_i)$, $y^i_x$ is the solution of $ y^i_x(t)= x+ \left(\int_0^t f_i( y^i_x(s), \alpha_i(s)) ds\right)$  and 
$\theta_i$ is such that $y^ i _x(\theta_i)\in \Ga$ and $\theta_i$ lies in $[\tau_i, \bar \tau_i]$, where $\tau_i$ is the exit time of $y_x ^ i$ from $\cP_i\backslash \Ga$ and  $\bar \tau_i$ is the exit time of $y_x ^ i$ from $\cP_i$.
\end{lemma}
\begin{proof}
  See  \cite{barles2011bellman}.
\end{proof}

 \begin{remark}
 The conclusions of Lemma \ref{sec:prop-visc-sub-1} hold if we replace $[\rm{H}3]$ with $[\tH3]$.   See  \cite{barles2013bellman}.
 \end{remark}

The following theorem is reminiscent of Theorem 3.3 in \cite{barles2011bellman}:
\begin{theorem}
  \label{sec:prop-visc-sub}
 Assume  $[\rm{H}0]$, $[\rm{H}1]$, $[\rm{H}2]$ and $[\rm{H}3]$. Let $v: \cS\to \R$ be a viscosity supersolution of (\ref{HJa}), bounded from below by $-c|x| -C$ for two positive numbers $c$ and $C$. Let $\phi\in\cR(\cS)$ and $\bar x\in\Ga$ be such that $v-\phi$ has a local minimum point at $\bar x$. Then, either [A] or [B] below is true:
 \begin{description}
 \item{[A]} There exists $\eta>0$, $i\in \{1,\dots, N\}$ and a sequence $x_k\in \cP_i\setminus \Ga$, 
 $\lim_{k\to +\infty} x_k=\bar x$
 such that $\lim_{k\to +\infty} v(x_k)=v(\bar x)$ and for each $k$, there exists a control law $\alpha_i^k$ such that the corresponding trajectory $y_{x_k}(s)\in \cP_i$ for all $s\in [0,\eta]$  and 
   \begin{equation}
     \label{eq:18}
v(x_k)\ge   \int_0^{\eta}   \ell_i(y_{x_k}(s),  \alpha^k_i(s)) e^{-\lambda s} ds + v (y_{x_k}(\eta)) e^{-\lambda \eta}.
   \end{equation}
 \item{[B]}
   \begin{equation}
     \label{eq:19}
 \lambda v(\bar x) + H_{\Ga}^T(\bar x,D(\phi|_\Gamma)(\bar x)) \ge 0.
   \end{equation}
 \end{description}
\end{theorem}
\begin{proof}
For any $i$ in $\{ 1, \dots, N\}$ we consider the function $\ph_i:\R \to \R$ defined as follows
  \begin{displaymath}
    \ph_i(d)= H_i(\bar x,D\left(\phi|_{\cP_i}\right) (\bar x) + d e_i)
  \end{displaymath}
  For $i$ in $\{1,\dots, N\}$, let $q_i$ be a real number such that $\ph_i(q_i)= \min_{d \in \R}\{\ph_i(d)\}$. We can already remark that according to Lemma \ref{sec:hamiltonians} we have,
\begin{equation}
  \label{eq:20}
H_\Ga(\bar x,D(\phi|_{\cP_1})(\bar x)+q_1e_1,\ldots,D(\phi|_{\cP_N})(\bar x)+q_Ne_N)=H_\Ga^T(\bar x,D(\phi|_\Gamma)(\bar x)).
\end{equation}
Consider the function
\begin{displaymath}
\psi _\epsilon(x)= v(x) - \phi(x) - q_i x_i  + \frac {x_i^2}{\epsilon^2} \quad \hbox{if } x\in \cP_i,
\end{displaymath}
defined on $\cS$, where $x_i=x.e_i$. Changing $\phi(x)$ into $\phi(x)-|x-\bar x|^2$ if necessary, we can assume that $\bar x$ is a strict local minimum point of $v-\phi$. Then, standard arguments show that for $\epsilon$ small enough, the function $\psi_\epsilon$ reaches its minimum close to $\bar x$, and that any sequence of such minimum points $x_\epsilon$ converges to $\bar x$ and satisfies $v(x_\epsilon)$ converges to $v(\bar x)$.\\
Up to the extraction of a subsequence, we can make out two cases.
\begin{enumerate}
\item \textbf{If $x_\epsilon \in \Ga$ :} Then, since $v$ is a viscosity supersolution of \eqref{HJa}, we have
\begin{displaymath}
\lambda v(x_\epsilon)+H_\Ga\left(x_\epsilon, D\left(\phi|_{\cP_1} \right)(x_\epsilon)+q_1 e_1, \ldots, D\left(\phi|_{\cP_N}+q_N e_N \right)(x_\epsilon)  \right) \ge 0,
\end{displaymath}
and letting $\epsilon$ tend to $0$, we obtain
\begin{equation}\label{eq: alt_A_B_20bis}
\lambda v(\bar x)+H_\Ga\left(\bar x, D\left(\phi|_{\cP_1} \right)(\bar x)+q_1 e_1, \ldots, D\left(\phi|_{\cP_N} \right)(\bar x) +q_N e_N \right) \ge 0.
\end{equation}
Then, using \eqref{eq:20} we deduce from \eqref{eq: alt_A_B_20bis} that
\begin{displaymath}
\lambda v(\bar x)+H^T_\Ga\left(\bar x, D\left(\phi|_{\Ga} \right)(\bar x)\right) \ge 0,
\end{displaymath}
hence $[B]$.
\item \textbf{If $x_\epsilon \in \cP_i\setminus\Ga$ for some $i\in \{1,\ldots,N\}$ :} We skip the writing of this case which is treated as in the corresponding case in the proof of \cite[Theorem~4.1]{uniqueness2013}, using the superoptimality \eqref{eq:16} of Lemma \ref{sec:prop-visc-sub-1}. 
\end{enumerate}
\end{proof}

 \begin{remark}
 The conclusions of Theorem \ref{sec:prop-visc-sub} hold if we replace $[\rm{H}3]$ with $[\tH3]$.   Indeed, the proof is based on Lemma \ref{sec:hamiltonians} and Lemma \ref{sec:prop-visc-sub-1} which stay true if we replace $[\rm{H}3]$ with $[\tH3]$.
 \end{remark}

\subsection{Comparison principle and uniqueness}
\label{sec:comparison-principle}

\begin{lemma}
\label{lem regularization}
Assume  $[\rm{H}0]$, $[\rm{H}1]$, $[\rm{H}2]$ and $[\rm{H}3]$.
Let $u:\cS \to \R$ be a bounded continuous viscosity subsolution of (\ref{HJa}). Let $(\rho_\epsilon)_\epsilon$ be a sequence of mollifiers defined on $\R$ as follows
\begin{displaymath}
\rho_\epsilon(x)= \frac{1}{\epsilon}\rho(\frac{x}{\epsilon}),
\end{displaymath}
where
\begin{displaymath}
\rho \in \cC^\infty(\R, \R^+), \quad \int_{\R}\rho(x)dx=1 \quad \hbox{and} \quad \supp(\rho)=[-1,1].
\end{displaymath}
We consider the function $u_\epsilon$ defined on $\cS$ by
\begin{displaymath}
u_\epsilon(x)= u * \rho_\epsilon(x)= \int_{\R}u((x_0-\tau)e_0+x')\rho_\epsilon(\tau) d\tau, \quad \quad \mbox{if } x=x_0e_0+x',
\end{displaymath}
where the decomposition of $x\in \cS$, $x=x_0e_0+x'$ is explained in \eqref{eq: decomposition_of_x2}.\\
Then, $u_\epsilon$ converges uniformly  to $u$ in $L^\infty(\cS, \R)$ and there exists a function $m:(0,+\infty) \to (0,+ \infty)$,  such that $\lim_{\epsilon \to 0}m(\epsilon)=0$ and the
function $u_\epsilon - m(\epsilon)$ is a viscosity subsolution of (\ref{HJa}) on a neighborhood of $\Ga$.
\end{lemma}
\begin{proof}
The uniform convergence of $u_\epsilon$  to $u$ in $L^\infty(\cS, \R)$ is classical because $u$ is bounded and continuous in $\cS$. The existence of a function $m: (0, +\infty) \to (0, +\infty)$ with $m(0^+)=0$ such that $u_\epsilon-m(\epsilon)$ is a subsolution of (\ref{HJa}) on a neighborhood of $\Ga$ was proved by P.L. Lions \cite{PLL} or Barles $\&$  Jakobsen \cite{BJconvergence}. A crucial information for obtaining this result is the fact that, according to Lemma \ref{sec:prop-visc-sub-3}, $u$ is Lipschiz continuous on $B(\Ga,r)\cap \cS$, for some \mbox{positive number $r$.}
\end{proof}


\begin{theorem}\label{sec:comparison-principle-1}
  Assume  $[\rm{H}0]$, $[\rm{H}1]$, $[\rm{H}2]$ and $[\rm{H}3]$. Let $u:  \cS\to \R$ be a bounded continuous viscosity subsolution of (\ref{HJa}),
 and $v: \cS\to \R$ be a  bounded viscosity supersolution of (\ref{HJa}). 
Then $u\le v$ in $\cS$.
\end{theorem}
\begin{proof}
The strategy of proof adopted here is the one of Barles-Briani-Chasseigne \cite{barles2011bellman} in the proof of their theorem 4.1. It is adapted to the present context.
\\
Let $u_\epsilon$ be the approximation of $u$ given by Lemma \ref{lem regularization}.
It is a simple matter to check that there exists a positive real number $M$ such that the function $\psi(x)=-|x|^2 -M$ is a viscosity subsolution of  (\ref{HJa}). For $0<\mu<1$, $\mu$ close to $1$, 
the function $u_{\epsilon,\mu}=\mu (u_\epsilon-m(\epsilon))+(1-\mu) \psi$ is a viscosity subsolution of  (\ref{HJa}), which tends to $-\infty$ as $|x|$ tends to $+\infty$. Let $M_{\epsilon,\mu}$ be the maximal value of $u_{\epsilon,\mu} -v $ which is reached at
some point $\bar x_{\epsilon,\mu}$. We argue by contradiction assuming that $M_{\epsilon,\mu}>0$.
\begin{enumerate}
\item If $\bar x_{\epsilon,\mu}\notin \Ga$, then we introduce the function $u_{\epsilon,\mu}(x)-v(x)-d^2(x,\bar x_{\epsilon,\mu})$, where $d(.,.)$ is the geodesic distance defined by (\ref{geodesic_distance}), which has a strict maximum at $\bar x_{\epsilon,\mu}$, and we double the variables, i.e. for $0<\beta\ll 1$, we consider
\begin{displaymath}
 (x,y)\longmapsto u_{\epsilon,\mu}(x) -v(y)- d^2(x,\bar x_{\epsilon,\mu})- \frac {d^2(x,y)}{\beta^2}.
\end{displaymath}
 Classical arguments then lead to the conclusion that $ u_{\epsilon,\mu}(\bar x_{\epsilon,\mu}) -v(\bar x _{\epsilon,\mu})\le 0$, thus $M_{\epsilon,\mu}\le 0$, which is a contradiction.
\item If  $\bar x_{\epsilon,\mu}\in \Ga$, according to Lemma \ref{sec:prop-visc-sub-3}, $u_{\epsilon, \mu}$ is Lipschitz continuous in a neighborhood of $\Ga$. Then, with a similar argument as in the proof of Lemma \ref{lem subsol inequality on gamma}, we can construct a test function $\ph \in \cR(\cS)$ such that $\ph |_\Ga=u_{\epsilon, \mu}|_\Ga$ and $\ph$ remains below $u_{\epsilon, \mu}$ on a neighborhood of $\Ga$ (take for exemple $\ph(x)= u_{\epsilon, \mu}(x_0,0)-Cx_i$ if $x= x_0e_0+x_ie_i\in \cP_i$, with $C$  great enough). It is easy to check that $v-\ph$ has a local minimum at $\bar x_{\epsilon,\mu}$, from the assumption that $M_{\epsilon, \mu}>0$. Then, we can use Theorem \ref{sec:prop-visc-sub} and we have two possible cases:
\begin{description}
\item{[B]}  $ \lambda v(\bar x_{\epsilon,\mu}) + H_{\Ga}^T(\bar x_{\epsilon,\mu},D( u_{\epsilon,\mu}|_\Ga)(\bar x_{\epsilon,\mu})) \ge 0$. \\
Moreover, $u_{\epsilon,\mu}$ is a subsolution of (\ref{HJa}), $C^1$ on $\Ga$. Then, according to Lemma \ref{lem subsol inequality on gamma}, we have the inequality $ \lambda u_{\epsilon,\mu}(\bar x_{\epsilon,\mu}) + H_{\Ga}^T(\bar x_{\epsilon,\mu},D( u_{\epsilon,\mu}|_\Ga)(\bar x_{\epsilon,\mu})) \le 0$. Therefore, we obtain that $u_{\epsilon,\mu}(\bar x_{\epsilon,\mu}) \le v(\bar x _{\epsilon,\mu})$, thus   $M_{\epsilon,\mu}\le 0$, which is a contradiction.
\item {[A]} With the notations of Theorem \ref{sec:prop-visc-sub},  we have that 
  \begin{displaymath}
    v(x_k)\ge   \int_0^{\eta}   \ell_i(y_{x_k}(s),  \alpha^k_i(s)) e^{-\lambda s} ds + v (y_{x_k}(\eta)) e^{-\lambda \eta}.
  \end{displaymath}
Moreover, from    Lemma \ref{sec:prop-visc-sub-1},
  \begin{displaymath}
    u_{\epsilon,\mu}(x_k)\le   \int_0^{\eta}   \ell_i(y_{x_k}(s),  \alpha^k_i(s)) e^{-\lambda s} ds + u_{\epsilon,\mu} (y_{x_k}(\eta)) e^{-\lambda \eta}.
  \end{displaymath}
Therefore 
\begin{displaymath}
   u_{\epsilon,\mu}(x_k)- v(x_k) \le (u_{\epsilon,\mu} (y_{x_k}(\eta)) - v (y_{x_k}(\eta)) ) e^{-\lambda \eta}.
\end{displaymath}
Letting $k$ tend to $+\infty$, we find that $M_{\epsilon,\mu}\le M_{\epsilon,\mu} e^{-\lambda \eta}$, which implies that $ M_{\epsilon,\mu} \le 0$, which is a contradiction.
\end{description}
\end{enumerate}
We conclude by letting $\epsilon$ tend to $0$ and $\mu$ tend to $1$.
\end{proof}

\begin{theorem}
 Assume  $[\rm{H}0]$, $[\rm{H}1]$, $[\rm{H}2]$ and $[\rm{H}3]$. Then, the value function $v$ is the unique viscosity solution of (\ref{HJa}) in $\cS$.
\end{theorem}

\begin{remark}
Under the assumptions $[\rm{H}0]$, $[\rm{H}1]$, $[\rm{H}2]$ and $[\rm{H}3]$ it is possible to prove a more general Comparison principle, where we do not assume the continuity of the subsolutions. The prove of a such Comparison principle is slightly more technical. Using that any subsolution of \eqref{HJa} is Lipschitz continuous in a neighborhood of $\Ga$, from Lemma \ref{sec:prop-visc-sub-3}, we are first able to prove a local Comparison principle similar to Theorem \ref{th: local comparison} and then to deduce a global Comparison principle similar to Theorem \ref{th: global comparison normal}. Essentially, we have to follow the proofs of Theorem \ref{th: local comparison} and Theorem \ref{th: global comparison normal}, with removing of the step of regularisation by sup-convolution.
\end{remark}

\section{Second case :  normal controllability near interface}\label{section_normal_controllability}

\subsection{The new framework}

We keep assumptions [H0], [H1], [H2] and we weaken the controllability assumption [H3] by only supposing normal controllability,
\begin{description}
\item{[$\tH$3] } There is a real number $\delta>0$ such that for any $i=1,\dots,N$ and for all $x\in\Ga$, 
  \begin{displaymath}
    [-\delta,\delta] \subset \left\lbrace f_i(x,a).e_i : a\in A_i \right\rbrace.
  \end{displaymath}
\end{description}

The following property is the counterpart of property \ref{prop: R_controllability} in the current framework.

\begin{property}
\label{prop: R_normal_controllability}
Assume  $[\rm{H}0]$ and $[\tH 3]$. Then, there exists $R>0$ a positive real number such that  for all $i=1,\dots,N$ and  $x\in B(\Ga,R)\cap \cP_i$
\begin{equation}
\label{eq: R_normal_controllability}
   [-\frac{\delta}{2},\frac{\delta}{2}] \subset \left\lbrace f_i(x,a).e_i : a\in A_i \right\rbrace.
\end{equation}
\end{property}

The dynamics $f$ and the cost function $\ell$ are defined on 
\begin{displaymath}
 M=\left\{(x,a);\; x\in \cS,\quad  a\in A_i  \hbox{ if } x\in \cP_i\backslash \Ga,  \hbox{ and } a\in \cup_{i=1}^N A_i   \hbox{ if } x \in \Ga \right\},
 \end{displaymath}
 as in (\ref{def_f}) and (\ref{def_l}). As above, we need to introduce the set of the admissible controlled trajectories. For this purpose, we recall that for $x\in \cS$,
\begin{displaymath}
  \tilde F(x)=\left\{ 
\begin{array}{ll}
F_i(x)\quad  &\hbox{if } x \hbox{ belongs to the open half-plane } \cP_i\backslash\Ga \\
\cup_{i=1}^N F_i(x) \quad  &\hbox{if } x\in \Ga,
\end{array}
\right. 
\end{displaymath}
\begin{equation}
  \label{eq:2}
Y_x=\left\{  y_x\in \rm{Lip}(\R^+; \cS)\;:\left|
    \begin{array}[c]{ll}
       \dot y_x(t)   \in \tilde F(y_x(t))  ,\quad& \hbox{for a.a. } t>0\\
       y_x(0)=x,
    \end{array}\right.  \right\}.
\end{equation}
 
\begin{theorem}\label{sec:optim-contr-probl_normal}
  Assume  $[\rm{H}0]$, $[\rm{H}2]$ and $[\tH3]$. Then 
  \begin{enumerate}
  \item For any $x\in \cS$, $Y_x$ is not empty.
   \item  For any $x\in \cS$, for each trajectory $y_x\in Y_x$, there exists a measurable function $\Phi:[0,+\infty)\to M$, $\Phi(t)=(\ph_1(t),\ph_2(t))$ with
    \begin{displaymath}
      (y_x(t),\dot y_x(t))= ( \ph_1(t), f (\ph_1(t),\ph_2(t))) ,\quad \hbox{for a.a. }t.
    \end{displaymath}
\item  Almost everywhere on $\{t: y_x(t)\in \Ga\}$, $f(y_x(t), \ph_2(t))\in \R e_0$.
\end{enumerate}
\end{theorem}

Points 2 and 3 are proved exactly as in the Theorem \ref{sec:optim-contr-probl}. Point 1 is little bit more difficult to prove because with [$\tH$3], it may happen that $0\not\in \tilde F(x)$ when $x\in \Ga$. Here, point 1 is essentially a consequence of the following lemma.

\begin{lemma}\label{lem: normal_controllability16}
Assume  $[\rm{H}0]$, $[\rm{H}2]$ and $[\tH3]$.
There exists $T>0$ such that $\forall i\in \{1, \ldots, N\}$ and $\forall x\in \cP_i$, there exists $y_x\in \rm{Lip}([0,T]; \cP_i)$ a solution of the differential inclusion 
\begin{displaymath}
\left\lbrace \begin{array}{l}
 \dot y_x(t)   \in \tilde F(y_x(t))\\
 y_x(0)=x.
 \end{array}\right.
 \end{displaymath}
\end{lemma}
\begin{proof}
Take $x\in \cP_i$.
\begin{itemize}
\item \textbf{If $x\in \Ga$ :}
according to [$\tH$3], there exists $a_i\in A_i$ such that $f_i(x,a_i).e_i=\delta$. Then, from the Lipschitz property in [H0], we have
\begin{equation}\label{eq: normal1}
f_i(z,a_i).e_i\ge \frac{\delta}{2}, \quad \quad \forall z\in B(x,\frac{\delta}{2L_f})\cap \cP_i.
 \end{equation}
 As a consequence, if we set $T=\frac{\delta}{2L_fM_f}$, there exists a unique function $y_x:[0,T)\to \cP_i$ such that $y_x(t)=x+\int_0^tf_i(y_x(s),a_i)ds$ and $y_x(t)\in B(x,\frac{\delta}{2L_f})\cap \cP_i$ $\forall t\in (0,T]$.
 \item \textbf{If $x\in \cP_i\setminus\Ga$ :} then, we chose arbitrarily $a_i \in A_i$ and we consider $\bar y: [0,\bar T) \to \cP_i$ the maximal solution of the integral equation $y(t)=x+\int_0^tf_i(y(s),a_i)ds$. If $\bar T> T=\frac{\delta}{2L_fM_f}$ then we take $y_x=\bar y$. Otherwise $\bar y(\bar T)$ is well defined and belongs to $\Ga$, then we are reduced to the case where $x\in \Ga$.
\end{itemize}
\end{proof}

Now, we are able to prove Point 1 of Theorem \ref{sec:optim-contr-probl_normal}.

\begin{proof}
Let $x$ be in $\cP_i$, for some $i\in \{1, \ldots, N \}$. According to Lemma \ref{lem: normal_controllability16} we are able to build a sequence of positive reals $(T_n)_{n\in \N}$ and maps $y_n: [0,T_n]\to \cP_i$ such that, for any $n\in \N$
\begin{itemize}
\item[(i)] $y_n(0)=x$ and $\dot{y}_n(t)\in F_i(y_n(t))$, for a.a. $t\in (0,T_n]$
\item[(ii)] $y_n(t)\in \cP_i$, $\forall t\in (0,T_n]$
\item[(iii)] $T_{n+1}-T_n\ge T$, with $T$ given by Lemma  \ref{lem: normal_controllability16}
\item[(iv)] $\forall 0\le q \le p$, $y_q|_{[0,T_p]}\equiv y_p$
\end{itemize}
So, the map $y_x : \R_+ \to \cP_i$ defined by $y_x(t)=y_n(t)$ if $t\in [0,T_n]$ belongs to $Y_x$ and particularly Point 1 of Theorem \ref{sec:optim-contr-probl_normal} is true.
\end{proof}

So, as above, we can take the set of admissible controlled trajectories starting from the initial datum $x$:
\begin{displaymath}
\cT_x=\left\{
 ( y_x, \alpha)  \ds \in L_{\rm{Loc}}^\infty( \R^+; M): \left|
  \begin{array}[c]{l}
\ds  y_x\in \rm{Lip}(\R^+; \cS),  \\  \ds  y_x(t)=x+\int_0^t f( y_x(s), \alpha(s)) ds \quad  \hbox{ in } \R^+ 
  \end{array} \right. \right\}.
\end{displaymath}

Then, the cost functional $J$ and the value function $v$ are defined by
\begin{displaymath}
  J(x;( y_x, \alpha) )=\int_0^\infty \ell(y_x(t),\alpha(t)) e^{-\lambda t} dt,
\end{displaymath}
where $\lambda>0$, and
\begin{displaymath}
v(x)= \inf_{( y_x, \alpha)\in \cT_x}  J(x;( y_x, \alpha) ).
\end{displaymath}
Unlike in \S~\ref{section: classical_controllability}, we cannot use classical arguments to prove that $v$ is continuous, because we do not suppose [H3] any longer. The main problem is that with $[\tH3]$ we are no longer sure that for each $x,z$ close to $\Ga$, there exists an admissible trajectory $y_{x,z}\in \cT_x$ from $x$ to $z$. We will prove later that $v$ is continuous, but for the moment $v$ is a priori a discontinuous function. In order to deal with this a priori discontinuity, we use the following notions :

\begin{definition}
Let $u : \cS\to \R$ be a bounded function.
\begin{itemize}
\item The lower semi-continuous envelope of $u$ is defined by
\begin{displaymath}
u_\star(x)=\liminf_{z\to x}u(z).
\end{displaymath}
\item The upper semi-continuous envelope of $u$ is defined by
\begin{displaymath}
u^\star(x)=\limsup_{z\to x}u(z).
\end{displaymath}
\end{itemize}
\end{definition}

\subsection{Hamilton-Jacobi equation}

\begin{definition}
A bounded function $u:\cS\to\R$ is a discontinuous viscosity solution of \eqref{HJa} in $\cS$ if $u^\star$ is a subsolution  and $u_\star$ is a  supersolution of \eqref{HJa} in $\cS$.
\end{definition}

The next lemmas will be used to prove the existence result.

\begin{lemma}\label{lem: discontinuous_existence}
 Assume  $[\rm{H}0]$, $[\rm{H}1]$, $[\rm{H}2]$ and $[\tH3]$.There exists some constants $r>0$ and $C>0$ such that for all $x=x_0e_0+x_ie_i\in B(\Ga,r)\cap \left( \cP_i\setminus\Ga \right)$, there exists an admissible controlled trajectory $(y_x,\alpha_x)\in \cT_x$ such that $\tau_x\le C x_i$, where $\tau_x$ is the exit time  of $y_x$
from $\cP_i\setminus \Ga$.
\end{lemma}
\begin{lemma}\label{lem: discontinuous_existence_2}
  Assume  $[\rm{H}0]$, $[\rm{H}1]$, $[\rm{H}2]$ and $[\tH3]$. For all $x\in \Ga$, $i\in \{1,\ldots,N\}$ and $a\in A_i$ such that $f_i(x,a).e_i\ge 0$, there exists a sequence $a_\epsilon\in A_i$ such that
 \begin{eqnarray}
  \label{maxsup:1_bis}
f_i(x,a_\epsilon).e_i&\geq& \delta \epsilon >0,\\
 \label{maxsup:f_bis}
\vert f_i(x,a_\epsilon)-f_i(x,a)\vert &\leq& 2M_f \epsilon,
\\
 \label{maxsup:ell_bis}
\vert \ell_i(x,a_\epsilon)- \ell_i(x,a)\vert& \leq & 2M_\ell \epsilon.
\end{eqnarray}
\end{lemma}
\begin{proof}
From [$\tH$3] there exists $a_\delta\in A_i$ such that
$f_i(x,a_\delta).e_i=\delta$. From [H2],
\begin{displaymath}
\epsilon(f_i(x,a_\delta),\ell_i(x,a_\delta))+(1-\epsilon)(f_i(x,a),\ell_i(x,a))\in \FL_i(x)  
\end{displaymath}
for any $\epsilon\in[0,1]$. So, there exists $a_\epsilon\in A_i$ such that
\begin{displaymath}
\epsilon(f_i(x,a_\delta),\ell_i(x,a_\delta))+(1-\epsilon)(f_i(x,a),\ell_i(x,a))=
(f_i(x,a_\epsilon),\ell_i(x,a_\epsilon))  
\end{displaymath}
which has all the desired properties.
\end{proof}

\begin{corollary}
  \label{cor: discontinuous_existence_1}
  For any  $i\in\{1,\dots,N\}$, $x\in \Ga$ and  $p_i\in \R e_0 \times \R e_i$,
\begin{equation}
   \label{maxsup:2}
  \max_{a\in A_i \hbox{ s.t. }  f_i(x,a).e_i\ge 0} (-p_i f_i(x,a) -\ell_i(x,a)) =  \sup_{a\in A_i \hbox{ s.t. }  f_i(x,a).e_i> 0} (-p_i f_i(x,a) -\ell_i(x,a)) .
   \end{equation}
\end{corollary}

\begin{theorem}\label{th: v_discontinuous_solution}
 Assume  $[\rm{H}0]$, $[\rm{H}1]$, $[\rm{H}2]$ and $[\tH3]$.
The value function $v$ defined in (\ref{eq:4}) is a  bounded  discontinuous viscosity solution of \eqref{HJa} in $ \cS$.
\end{theorem}
\begin{proof}
  The proof being a bit long, we postpone it to the appendix \ref{appendix: v_discontinuous_solution}.
\end{proof}

We give now some results on the behavior of the Hamiltonians in the new framework.

\begin{lemma}\label{lem: normal_controllability3}
 Assume  $[\rm{H}0]$, $[\rm{H}1]$, $[\rm{H}2]$ and $[\tH3]$.
  Then, for $i= 1, \ldots, N$, we have
 \begin{equation}
 \label{eq: regularity of hi with respect to x}
 |H_i(x,p) -H_i(y,p) | \le L_f |x-y||p|+ \omega_\ell(|x-y|), \quad \hbox{for any } x, y\in \cP_i \hbox{ and }p \in \R e_0 \times \R e_i,
 \end{equation}
 and
  \begin{equation}
 \label{eq: regularity of hi with respect to p}
 |H_i(x,p) -H_i(x,q) | \le M_f |p-q|, \quad \hbox{for any } x\in \cP_i \hbox{ and }p, q \in \R e_0 \times \R e_i,
 \end{equation}
where $L_f, M_f$ and $\omega_\ell$ are defined in $[\rm{H}0]$ and $[\rm{H}1]$.\\
At last, if $C_M=\max \{M_f,M_l \}$,
\begin{equation}
 \label{eq coercivity with respect to p'}
 H_i(x,p_0e_0+p_ie_i) \ge \frac{\delta}{2} |p_i|-C_M(1+|p_0|),  \quad \hbox{for any } x\in B(\Ga,R)\cap \cP_i \hbox{ and } p_0,p_i\in\R ,
\end{equation}
 where $R>0$ is a positive number as in \eqref{eq: R_normal_controllability}.
\end{lemma}
\begin{proof}
The proof of the Lemma is given in \cite{barles2013bellman} Lemma 7.1 We supply it for completeness.
Assumptions $[\rm{H}0]$ and $[\rm{H}1]$ yield \eqref{eq: regularity of hi with respect to x} and \eqref{eq: regularity of hi with respect to p}. For (\ref{eq coercivity with respect to p'}) we also have to use the  partial controllability assumption $[\tH3]$. Indeed, according to Property \ref{prop: R_normal_controllability} there exist some controls $a_1, a_2 \in A_i$ such that
\begin{displaymath}
-f_i(x,a_1).e_i=\frac{\delta}{2} >0, \quad -f_i(x,a_2).e_i=-\frac{\delta}{2}.
\end{displaymath}
Now we compute $H_i(x,p_0e_0+p_ie_i)$ assuming that $p_i>0$ (the other case is treated similarly).
\begin{displaymath}
\begin{array}{lll}
H_i(x,p_0e_0+p_ie_i) & \ge & -f_i(x,a_1).(p_0e_0+p_ie_i) -\ell_i(x,a_1) \\
& \ge & \frac{\delta}{2} |p_i| - f_i(x,a_1).p_0e_0  -\ell_i(x,a_1) \\
& \ge & \frac{\delta}{2} |p_i| - C_M|p_0|  -C_M, \\
\end{array}
\end{displaymath}
the last line coming from the boundedness of $f_i$ and $l_i$. This concludes the proof.
\end{proof}

\begin{lemma}
\label{lem: normal_controllability4}
 Assume  $[\rm{H}0]$, $[\rm{H}1]$, $[\rm{H}2]$ and $[\tH3]$. Then, for any $x,x' \in \Ga$, $i\in \{1, \ldots, N \}$ and $a\in A_i$ such that $f_i(x,a).e_i\ge 0$, there exists $a' \in A_i$ such that $f_i(x',a').e_i\ge 0$ and

\begin{equation}
\label{eq lem of lem 1}
|f_i(x,a)-f_i(x',a')|\le M|x-x'|,
\end{equation}
\begin{equation}
\label{eq lem of lem 2}
|\ell_i(x,a)-\ell_i(x',a')|\le \omega(|x-x'|),
\end{equation}
 where, if $M_f$, $L_f$, $M_\ell$, $\delta$ and $\omega_\ell$ are given by assumptions $[\rm{H}0]$, $[\rm{H}1]$ and $[\tH3]$.
 \begin{displaymath}
 M=L_f\left(1+2M_f\delta^{-1} \right) \quad \hbox{and } \quad \omega(t)=\omega_\ell(t)+2M_\ell L_f\delta^{-1}t \quad \hbox{for } t\ge 0.
 \end{displaymath}
\end{lemma}
\begin{proof}
The proof follows the lines of that of \cite{barles2013bellman}, Lemma 7.4 : let $a\in A_i$ be such that $f_i(x,a).e_i\ge 0$. Fix $x' \in \Ga$, we have two possibilities. If $f_i(x',a).e_i\ge 0$ the conclusion follows easily because according to [H0] and [H1] we have respectively (\ref{eq lem of lem 1}) and (\ref{eq lem of lem 2}) with $a'=a$. Otherwise $f_i(x',a).e_i< 0$. By the partial controllability assumption in [$\tH$3] there exists a control $a_1 \in A_i$ such that $f_i(x',a_1).e_i= \delta$. We then set
\begin{displaymath}
\overline{\mu}:= \frac{\delta}{\delta-f_i(x',a).e_i}.
\end{displaymath}
Since $\overline{\mu}\in (0,1)$, by the convexity assumption in [H2], there exists a control $a'\in A_i$ such that
\begin{displaymath}
\overline{\mu}\left(f_i(x',a),\ell_i(x',a) \right)+ (1-\overline{\mu})\left(f_i(x',a_1),\ell_i(x',a_1) \right) = \left(f_i(x',a'),\ell_i(x',a') \right).
\end{displaymath}
By construction $f_i(x',a').e_i=0$ and then, in particular, $f_i(x',a').e_i \ge 0$. Moreover, since
\begin{displaymath}
(1-\overline{\mu})= \frac{-f_i(x',a).e_i}{\delta-f_i(x',a).e_i} \le \frac{f_i(x,a).e_i-f_i(x',a).e_i}{\delta-f_i(x',a).e_i}  \le \frac{L_f|x-x'|}{\delta},
\end{displaymath}
we have
\begin{displaymath}
\begin{array}{lll}
|f_i(x,a)-f_i(x',a')| & \le & |f_i(x,a)-f_i(x',a)| + |f_i(x',a)-f_i(x',a')| \\
& \le & L_f|x-x'| + (1-\bar \mu)|f_i(x',a)-f_i(x',a_1)|\\
& \le & L_f\left(1+2M_f\delta^{-1} \right)|x-x'|.
\end{array}
\end{displaymath}
This proves (\ref{eq lem of lem 1}). The same calculation with $\ell_i$ gives us (\ref{eq lem of lem 2}).
\end{proof}

\begin{remark}
\label{rmk: normal_controllability4}
In Lemma \ref{lem: normal_controllability4}, if $x\in \Ga$ and $a\in A_i$ are such that $f_i(x,a).e_i=0$, then we have the stronger result that $\forall x'\in \Ga$, $\exists a'\in A_i$ such that $f_i(x',a').e_i=0$ and the inequalities \eqref{eq lem of lem 1} and \eqref{eq lem of lem 2} are true. Indeed, if $f_i(x',a).e_i \le 0$ the proof of Lemma \ref{lem: normal_controllability4} directly provides a suitable $a'\in A_i$. We only need to specify the strategy in the case when $f_i(x',a).e_i > 0$. We consider $a_2\in A_i$, given by $[\tH3]$ such that $f_i(x',a_2).e_i=-\delta$. Then, we set $\tilde \mu=\frac{\delta}{f_i(x',a).e_i+\delta}\in (0,1)$ and we choose $a'\in A_i$, given by $[\rm{H}2]$, such that 
$\tilde \mu \left(f_i(x',a),\ell_i(x',a) \right)+ (1-\tilde \mu)\left(f_i(x',a_2),\ell_i(x',a_2) \right) = \left(f_i(x',a'),\ell_i(x',a') \right)$. With this choice for $a'\in A_i$, the same calculations as in the proof above, using $1-\tilde \mu= \frac{f_i(x',a).e_i}{f_i(x',a).e_i+\delta}=\frac{(f_i(x',a)-f_i(x,a)).e_i}{f_i(x',a).e_i+\delta}$, allow us to conclude.
\end{remark}

\begin{lemma}
\label{lem inequality on Hgamma+}
 Assume  $[\rm{H}0]$, $[\rm{H}1]$, $[\rm{H}2]$ and $[\tH3]$. The Hamiltonian $H_i^+$ defined in (\ref{eq:7bis}) has the following properties
 \begin{equation}
 \label{eq:normal_controllability01}
 H_i^+(x,p_0e_0+p_ie_i) \ge H_i^+(x,p_0e_0+q_ie_i) \quad  \forall x\in \Ga \hbox{ and }  \forall p_0,p_i,q_i\in \R  \hbox{ s.t } \; p_i \le q_i,
 \end{equation}
  \begin{equation}
 \label{eq pseudo coercivity with respect to p' of hi+}
 H^+_i(x,p_0e_0+p_ie_i) \ge -\delta p_i-C_M(1+|p_0|),  \quad \forall x\in \Ga  \;\; \hbox{ and }\;\; p_0,p_i\in \R,
 \end{equation}
 \begin{equation}
  \label{eq:tech02}
  |H_i^+(x,p)-H_i^+(y,p)| \le M|x-y||p|+\omega(|x-y|) \quad \forall x,y\in \Ga \hbox{ and }\;\; p \in \R e_0 \times \R e_i,
 \end{equation}
where $M$ and $\omega$ are defined in Lemma \ref{lem: normal_controllability4} and $C_M=\max\{M_f,M_\ell\}$.
\end{lemma}
\begin{proof}
We skip the easy proof of (\ref{eq:normal_controllability01}). The proof of (\ref{eq pseudo coercivity with respect to p' of hi+}) is the same as that of (\ref{eq coercivity with respect to p'}) with the difference that here the dynamics leading out of $\cP_i$ are not allowed. Let us prove (\ref{eq:tech02}). Let $x,y$ be in $\Ga$ and $p$ be in $\R e_0 \times \R e_i$. First, there exists $a\in A_i$ such that $f_i(x,a).e_i\ge 0$ and $H_i^+(x,p)=-f_i(x,a).p-\ell_i(x,a)$. Then, we consider $a'\in A_i$ given by Lemma \ref{lem: normal_controllability4} such that $f_i(y,a').e_i\ge 0$ and the inequality (\ref{eq lem of lem 1}) and (\ref{eq lem of lem 2}) are satisfied. Therefore,
\begin{displaymath}
\begin{array}{lll}
 H_i^+(x,p)-H_i^+(y,p) & \le & \left(-f_i(x,a).p-\ell_i(x,a) \right)-\left(-f_i(y,a').p-\ell_i(y,a') \right)\\
 & = & \left(f_i(y,a')-f_i(x,a) \right). p+\left(\ell_i(y,a')-\ell_i(x,a) \right)\\
 & \le &  M|x-y||p|+\omega(|x-y|).
\end{array}
\end{displaymath}
We conclude the proof by exchanging the roles of $x$ and $y$.
\end{proof}
\begin{remark}
\label{rmk inequality on Hgamma+}
In view of the calculations above, if the maximum which defines $H_i^+(x,p_0 e_0+p_i e_i)$ is reached for some $a\in A_i$ such that $f_i(x,a).e_i=0$, then from the Remark \ref{rmk: normal_controllability4} we have the stronger inequality,
\begin{equation}
\label{eq: inequality on Hgamma+_bis}
H_i^+(x,p_0 e_0+p_i e_i)-H_i^+(y,p_0 e_0+p_i e_i) \le  M|x-y||p_0|+\omega(|x-y|),
\end{equation}
but without the modulus on the left side of the inequality.
\end{remark}

\subsection{Comparison principle and Uniqueness}

A key argument in the proof of the Comparison Principle in Theorem \ref{sec:comparison-principle-1} is the fact that the subsolutions of (\ref{HJa}) are Lipschitz continuous in a neighborhood of $\Ga$. We have not this property in the current framework the above method cannot be applied directly. We will first prove a local comparison principle by reducing ourselves to the case when a subsolution is Lipschitz continuous, and then we will deduce a global comparison principle.
For this purpose, we start by stating some useful lemmas.

\vspace{1em}

The following transformation will allow us to focus on the case when the subsolutions are locally Lipschitz continuous in a neighborhood of $\Ga$.

\begin{definition}
Let $u: \cS \to \R$ be a bounded, usc  function and $\alpha$ be a positive number. We define the sup-convolution of $u$ with respect to the $x_0$-variable by
\begin{equation}\label{eq: normal_controllability2}
 u_\alpha(x):=\max_{z_0\in \R}\left\lbrace u(z_0e_0 +x_ie_i) - \left(\frac{|z_0-x_0|^2}{\alpha^2}+\alpha\right)^{\frac{p}{2}} \right\rbrace, \quad \quad \mbox{if } x=x_0e_0+x_ie_i\in \cP_i,
\end{equation}
where $\alpha, p>0$ are positive numbers.
\end{definition}

We recall well known results on sup-convolution.

\begin{lemma}\label{lem: sup_convolution_supremum}
Let $u: \cS \to \R$ be a bounded  function and $\alpha, p$ be positive numbers. Then, for any $x \in \cS$ the supremum which defines $u_\alpha(x)$ is achieved at a point $z_0\in \R$ such that
\begin{equation}
\label{eq: control_on_z0}
(\frac{|z_0-x_0|^2}{\alpha^2}+\alpha)^{\frac{p}{2}} \le 2\ \parallel u \parallel_\infty + \alpha^{\frac{p}{2}}.
\end{equation}
\end{lemma}

\begin{lemma}\label{lem: normal_controllability2}
Let $u: \cS \to \R$ be a bounded, usc  function. Then, for all $\alpha, p>0$, the sup-convolution $u_\alpha$ is locally Lipschitz continuous with respect to $x_0$ ; i.e. for any compact subset $K\subset \R^3$, there exists $C_K>0$ such that for all $ x=x_0 e_0+x_i e_i, y=y_0 e_0+x_i e_i\in K\cap \cS$, $|u_\alpha(x)-u_\alpha(y)|\le C_K|x_0-y_0|$.
\end{lemma}

\begin{lemma}\label{lem: u_alpha_lipschitz}
Assume  $[\rm{H}0]$, $[\rm{H}1]$, $[\rm{H}2]$ and $[\tH3]$.
Let $R>0$ be a positive number as in \eqref{eq: R_normal_controllability}.
Let $u : \cS \to \R$ be a bounded, usc subsolution of (\ref{HJa}) in $\cS$ and $\alpha, p$ be some positive numbers. Then, for all $M>0$, $u_\alpha$ is Lipschiz continuous in $B_M(\Ga,R)\cap \cS$, where $B_M(\Ga,R):=\{ x=x_0e_0+x'\in B(\Ga,R): |x_0|\le M\}$.
\end{lemma}
\begin{proof}
Fix $M>0$. To get that $u_\alpha$ is Lipschitz continuous in $B_M(\Ga,R)\cap \cS$, it is enough to show that there exists a positive number $C^\star(M,\alpha,p)$, which can depend to $M, \alpha$ and $p$, such that $u_\alpha$ is a subsolution of 
\begin{equation}\label{eq: of_subsolution_on_u_alpha}
|Du(x)|\le C^\star(M,\alpha,p) \; \mbox{ in } \; B_M(\Gamma,R)\cap \cS, 
\end{equation}
i.e. for any $x\in B_M(\Ga,R)$ and $\varphi \in \cR(\cS)$, such that $u_\alpha-\varphi$ has a local maximum point at $x$,
\begin{eqnarray}
\label{eq: of_subsolution_on_u_alpha2}
|D(\varphi|_{\cP_i})(x)| \le C^\star(M,\alpha,p) & \mbox{ if }  x\in (B_M(\Ga,R)\cap \cP_i)\setminus \Ga,\\
\label{eq: of_subsolution_on_u_alpha3}
\max_{i=1,\ldots,N}\left\lbrace|\partial_{x_0}\varphi (x)|+\left[\partial_{x_i}(\varphi |_{\cP_i}) (x)\right]_-\right\rbrace\le C^\star(M,\alpha,p)& \mbox{ if }  x\in B_M(\Ga,R)\cap\Ga,
\end{eqnarray}
where  for $x\in \R$, $[x]_-=\max\{0,-x\}$.\\
Indeed, if \eqref{eq: of_subsolution_on_u_alpha} is proved, then  the method used in the proof of Lemma \ref{sec:prop-visc-sub-3} allows us to get that $u_\alpha$ is Lipschitz continuous with Lipschitz constant $C^\star(M,\alpha,p)$.\\
Let us prove the existence of the constant $C^\star(M,\alpha,p)$.\\
According to Lemma \ref{lem: normal_controllability2}, there exists a constant $C(M,\alpha,p)$ such that $u_\alpha$ is $C(M,\alpha,p)$-Lipschitz continuous with respect to the $x_0$-variable. A direct consequence of this is that $u_\alpha$ is a subsolution of 
\begin{equation}\label{eq: lip_ineq_h3tilde1}
|\partial_{x_0}u(x)|\le C(M,\alpha,p) \;\; \mbox{ in } \; B_M(\Ga,R)\cap \cS.
\end{equation}
i.e. for any $x\in B_M(\Ga,R)\cap \cS$ and $\varphi \in \cR(\cS)$, such that $u_\alpha-\varphi$ has a local maximum \mbox{point at $x$}
\begin{displaymath}
|\partial_{x_0}\varphi(x)| \le C(M,\alpha,p).
\end{displaymath}
It remains to get information on $|\partial_{x_i}(u_\alpha|_{\cP_i})|$ in a viscosity sense. Let $\bar x= \bar x_0 e_0+\bar x_i e_i\in B_M(\Ga,R)\cap \cP_i$ and $\varphi \in \cR(\cS)$ be such that $u_\alpha-\varphi$ has a local maximum at $\bar x$. According to Lemma \ref{lem: sup_convolution_supremum}, there exists $\bar z_0\in \R$ such that 
\begin{equation}\label{eq: normal_controllability6}
u_\alpha(\bar x):= u(\bar z_0e_0+\bar x_ie_i)- \left(\frac{|\bar z_0-\bar x_0|^2}{\alpha^2}+\alpha\right)^{\frac{p}{2}}
\end{equation}
and
\begin{equation}\label{eq: proof_u_alpha_lip1bis}
\left(\frac{|\bar z_0-\bar x_0|^2}{\alpha^2}+\alpha\right)^{\frac{p}{2}} \le  2\parallel u \parallel_\infty +\alpha^{\frac{p}{2}}.
\end{equation}
Then, if we set $\varphi_\alpha : z_0 e_0+x_i e_i \longmapsto (\frac{| z_0-\bar x_0|^2}{\alpha^2}+\alpha)^{\frac{p}{2}}+\varphi(\bar x_0e_0+x_i e_i)$, we can check that $\varphi_\alpha$ belongs to $\cR(\cS)$ and that $u-\varphi_\alpha$ has a local maximum at $\bar z= \bar z_0 e_0+\bar x_i e_i$. Since $u$ is a subsolution of \eqref{HJa} in $\cS$, we deduce that
\begin{equation}\label{eq: proof_u_alpha_lip1}
\lambda u(\bar z)+\sup_{ (\zeta,\xi) \in \FL(\bar z)}\{-D\ph_\alpha(\bar z,\zeta) -\xi\}\le 0.
\end{equation}
\textbf{If $\bar x\in (B_M(\Ga,R)\cap \cP_i)\setminus \Ga$:}  The equation \eqref{eq: proof_u_alpha_lip1} can be rewritten as follows
\begin{equation}\label{eq: proof_u_alpha_lip2}
\lambda u(\bar z)+H_i(\bar z_0e_0+\bar x_ie_i, \frac{p(\bar z_0-\bar x_0)}{\alpha^2}\left(\frac{|\bar x_0-\bar z_0|^2}{\alpha^2}+\alpha\right)^{\frac{p}{2}-1}e_0 + \partial_{x_i} \left(\ph|_{\cP_i} \right)(\bar x_0e_0+\bar x_ie_i)e_i ) \le 0.
\end{equation}
Then, according to \eqref{eq coercivity with respect to p'} in Lemma \ref{lem: normal_controllability3}
\begin{equation}\label{eq: proof_u_alpha_lip3}
\lambda u(\bar z)+ \frac{\delta}{2} |\partial_{x_i} \left(\ph|_{\cP_i} \right)(\bar x)|-C_M\left(1+\frac{p|\bar z_0-\bar x_0|}{\alpha^2}\left(\frac{|\bar x_0-\bar z_0|^2}{\alpha^2}+\alpha\right)^{\frac{p}{2}-1}\right) \le 0.
\end{equation}
But from \eqref{eq: proof_u_alpha_lip1bis}, the term $\frac{p|\bar z_0-\bar x_0|}{\alpha^2}\left(\frac{|\bar x_0-\bar z_0|^2}{\alpha^2}+\alpha\right)^{\frac{p}{2}-1}$ is bounded by a constant $K(\alpha,p)=\frac{p\sqrt{\left(2 \parallel u\parallel_\infty + \alpha^{\frac{p}{2}} \right)^{\frac{2}{p}}-\alpha}}{\alpha}\left(2 \parallel u\parallel_\infty + \alpha^{\frac{p}{2}} \right)^{1-\frac{2}{p}}$, so that we get
\begin{equation}\label{eq: proof_u_alpha_lip4}
 |\partial_{x_i} \left(\ph|_{\cP_i} \right)(\bar x)| \le \frac{2}{\delta}\left(\lambda \parallel u \parallel_\infty+C_M\left(1+K(M,\alpha,p)\right) \right).
\end{equation}
\textbf{If $\bar x\in B_M(\Ga,R)\cap\Ga$:} 
 equation \eqref{eq: proof_u_alpha_lip1} can be rewritten as follows
\begin{equation}\label{eq: proof_u_alpha_lip5}
\begin{array}{lll}
\lambda u(\bar z)+H_\Ga\left(\bar z, \frac{p|\bar z_0-\bar x_0|}{\alpha^2}\left(\frac{|\bar x_0-\bar z_0|^2}{\alpha^2}+\alpha\right)^{\frac{p}{2}-1}e_0+ \partial_{x_1}\left(\ph|_{\cP_1} \right)(\bar x)e_1,\right. \ldots\\
 \quad \quad \quad \quad \quad \quad \quad \quad \left. \ldots, \frac{p|\bar z_0-\bar x_0|}{\alpha^2}\left(\frac{|\bar x_0-\bar z_0|^2}{\alpha^2}+\alpha\right)^{\frac{p}{2}-1}e_0+ \partial_{x_N}\left(\ph|_{\cP_N} \right)(\bar x)e_N \right) \le 0.
\end{array}
\end{equation}
Particularly, if we fix $i\in \{1,\ldots,N\}$, \eqref{eq: proof_u_alpha_lip5} implies
\begin{equation}\label{eq: proof_u_alpha_lip6}
\lambda u(\bar z)+H_i^+\left(\bar z, \frac{p|\bar z_0-\bar x_0|}{\alpha^2}\left(\frac{|\bar x_0-\bar z_0|^2}{\alpha^2}+\alpha\right)^{\frac{p}{2}-1}e_0 + \partial_{x_i}\left(\ph|_{\cP_i} \right)(\bar x)e_i \right) \le 0,
\end{equation}
and thanks to \eqref{eq pseudo coercivity with respect to p' of hi+} in Lemma \ref{lem inequality on Hgamma+} and  \eqref{eq: proof_u_alpha_lip1bis} we deduce that
\begin{equation}\label{eq: proof_u_alpha_lip7}
 -\partial_{x_i} \left(\ph|_{\cP_i} \right)(\bar x) \le \frac{1}{\delta}\left(\lambda \parallel u \parallel_\infty+C_M\left(1+K(M,\alpha,p)\right) \right),
\end{equation}
with the same constant $K(M,\alpha,p)$ as in the case where $\bar x\in B_M(\Ga,R)\cap\cP_i$.\\
Finally, according to \eqref{eq: lip_ineq_h3tilde1}, \eqref{eq: proof_u_alpha_lip4} and \eqref{eq: proof_u_alpha_lip7} we have that $u_\alpha$ is a subsolution of  \eqref{eq: of_subsolution_on_u_alpha} in  $B_M(\Gamma,R)\cap \cS$ with the constant $C^\star(M,\alpha,p)=\sqrt{ C(M,\alpha,p)^2+\frac{4}{\delta^2}\left(\lambda \parallel u \parallel_\infty+C_M\left(1+K(M,\alpha,p)\right) \right)^2}$. 
This concludes the proof.
\end{proof}

\begin{lemma}\label{lem: sup_convolution}
 Assume  $[\rm{H}0]$, $[\rm{H}1]$, $[\rm{H}2]$, $[\tH3]$ and $[\rm{H}4]$. Let $y_0$ be in $\Ga$ and $R>0$ be as in \eqref{eq: R_normal_controllability}.
We denote by $Q$ the set $\cS \cap B(y_0,R)$.
Let $u : \cS \to \R$ be a bounded, usc subsolution of (\ref{HJa}) in $Q$. Then, for all $\alpha, p>0$ small enough, if we set
\begin{equation}\label{def:Q_alpha_set}
Q_\alpha :=\left\lbrace x\in Q : \rm{dist}(x,\partial Q)> \alpha\sqrt{(2\parallel u \parallel_\infty+\alpha^{\frac{p}{2}})^{2/p}-\alpha} \right\rbrace,
\end{equation}
 the sup-convolution $u_\alpha$ defined in (\ref{eq: normal_controllability2}), is Lipschitz continuous in $Q_\alpha$ and there exists $m:(0,+\infty)\to (0,+\infty)$ such that $\lim_{\alpha \to 0}m(\alpha)=0$ and $u_\alpha-m(\alpha)$ is a subsolution of (\ref{HJa}) in $Q_\alpha$.
\end{lemma}

\begin{proof}
First, according to Lemma \ref{lem: u_alpha_lipschitz} it is clear that $u_\alpha$ is Lipschitz continuous in $Q_\alpha$. It remains to find $m:\R_+ \to \R_+$ such that $u_\alpha-m(\alpha)$ is a subsolution of (\ref{HJa}) in $Q_\alpha$ and to check that $m$ has the desired limit as $\alpha$ tends to $0$. For this purpose, we consider $\varphi \in \cR(\cS)$ such that $u_\alpha-\varphi$ has a local maximum point at $\bar x=\bar x_0 e_0+\bar x_i e_i$. According to Lemma \ref{lem: sup_convolution_supremum}, there exists $\bar z_0\in \R$ such that 
\begin{equation}\label{eq: normal_controllability6bisbis}
u_\alpha(\bar x)= u(\bar z_0e_0+\bar x_ie_i)- \left(\frac{|\bar z_0-\bar x_0|^2}{\alpha^2}+\alpha\right)^{\frac{p}{2}}
\end{equation}
and
\begin{equation}\label{eq: proof_u_alpha_lip1bisbis}
|\bar x_0-\bar z_0|\le \alpha\sqrt{(2\parallel u \parallel_\infty+\alpha^{\frac{p}{2}})^{\frac{2}{p}}-\alpha}.
\end{equation}
From \eqref{eq: normal_controllability6bisbis}  the function  $x_0\mapsto u(\bar z_0e_0+\bar x_ie_i)-\left(\frac{|\bar z_0- x_0|^2}{\alpha^2}+\alpha\right)^{\frac{p}{2}} -\varphi(x_0e_0+ \bar x_ie_i)$ has a local maximum at $\bar x_0$ and we deduce that
\begin{equation}\label{eq: normal_controllability7bis}
\partial_{x_0}\ph(\bar x)=\frac{p(\bar z_0-\bar x_0)}{\alpha^2}\left( \frac{|\bar z_0-\bar x_0|^2}{\alpha^2}+\alpha\right)^{\frac{p}{2}-1}.
\end{equation}
On the other hand, since $\bar x\in Q_\alpha$, the inequality \eqref{eq: proof_u_alpha_lip1bisbis} implies that $\bar z=\bar z_0e_0+\bar x_ie_i\in Q$ and we can use that $u$ is a subsolution of (\ref{HJa}) in $Q$ :
\\
\textbf{1.} If $\bar x\in \cP_i \setminus \Ga$: then, $\bar z =\bar z_0e_0+\bar x_ie_i$ also belongs to $\cP_i \setminus \Ga$ and using the test function $\tilde{\ph}: z=z_0e_0+z' \longmapsto \ph(\bar x_0e_0+z')+\left( \frac{|z_0-\bar x_0|^2}{\alpha^2}+\alpha\right)^{\frac{p}{2}}$, where we recall that $z'$ denotes $z_je_j$ if $z\in \cP_j$, we have the viscosity inequality
\begin{equation}\label{eq: normal_controllability9bis}
\lambda u(\bar z) + H_i\left(\bar z, \frac{p(\bar z_0-\bar x_0)}{\alpha^2}\left( \frac{|\bar z_0-\bar x_0|^2}{\alpha^2}+\alpha\right)^{\frac{p}{2}-1}e_0 + \partial_{x_i} \left(\ph|_{\cP_i} \right)(\bar x)e_i \right) \le 0.
\end{equation}
Combining the previous results, \eqref{eq: normal_controllability6bisbis}, \eqref{eq: normal_controllability7bis} and \eqref{eq: normal_controllability9bis}, we get
\begin{displaymath}
\lambda u_\alpha(\bar x)+ \lambda \left( \frac{|\bar z_0-\bar x_0|^2}{\alpha^2}+\alpha\right)^{\frac{p}{2}} +H_i(\bar z, D \left(\ph|_{\cP_i} \right)(\bar x) ) \le 0.
\end{displaymath}
Then, according to (\ref{eq: regularity of hi with respect to x}) we have
\begin{equation}
\label{eq demo u_alpha lipschitz 1}
\lambda u_\alpha(\bar x)+ H_i(\bar x, D \left(\ph|_{\cP_i} \right)(\bar x) ) \le -\lambda \left( \frac{|\bar z_0-\bar x_0|^2}{\alpha^2}+\alpha\right)^{\frac{p}{2}}+ L_f |\bar x_0-\bar z_0| | D \left(\ph|_{\cP_i} \right)(\bar x)|+ \omega_\ell(|\bar x_0-\bar z_0|).
\end{equation}
But, from \eqref{eq coercivity with respect to p'} and \eqref{eq: normal_controllability9bis}
\begin{displaymath}
| \partial_{x_i} \left(\ph|_{\cP_i} \right)(\bar x)| \le \frac{2}{\delta}\left(\lambda \parallel u \parallel_\infty +C_M\left(1+\frac{p|\bar z_0-\bar x_0|}{\alpha^2}\left( \frac{|\bar z_0-\bar x_0|^2}{\alpha^2}+\alpha\right)^{\frac{p}{2}-1}\right) \right).
\end{displaymath}
Then, from \eqref{eq: normal_controllability7bis} and (\ref{eq demo u_alpha lipschitz 1}) we get
\begin{displaymath}
\begin{array}{l}
\lambda u_\alpha(\bar x)+ H_i(\bar x, D \left(\ph|_{\cP_i} \right)(\bar x) ) \le \frac{|\bar z_0-\bar x_0|^2}{\alpha^2}\left(\frac{|\bar z_0-\bar x_0|^2}{\alpha^2}+\alpha \right)^{\frac{p}{2}-1}\left( pL_f\left(1+\frac{2C_M}{\delta} \right)-\lambda \right)\\
\quad \quad \quad \quad \quad \quad \quad \quad \quad \quad \quad \quad \quad \quad +|\bar x_0-\bar z_0|\frac{2L_f}{\delta}\left(\lambda \parallel u \parallel_\infty +C_M\right)+\omega_\ell(|\bar x_0-\bar z_0|),
\end{array}
\end{displaymath}
and for $p$ small enough, we get
\begin{displaymath}
\begin{array}{l}
\lambda u_\alpha(\bar x)+ H_i(\bar x, D \left(\ph|_{\cP_i} \right)(\bar x) ) \le |\bar x_0-\bar z_0|\frac{2L_f}{\delta}\left(\lambda \parallel u \parallel_\infty +C_M\right)+\omega_\ell(|\bar x_0-\bar z_0|).
\end{array}
\end{displaymath}
Finally, according to \eqref{eq: proof_u_alpha_lip1bisbis} the right hand side of the latter inequality gives us $m(\alpha)$. 
\\
\textbf{2.} If $\bar x\in \Ga$:  $\bar x_i =0$ and $\bar z=\bar z_0e_0+\bar x_ie_i$ also belongs to $\Ga$. Using  $\tilde{\ph}: z=z_0e_0+z'  \longmapsto \ph(\bar x_0e_0+z')+\left(\frac{|z_0-\bar x_0|^2}{\alpha^2}+\alpha\right)^{\frac{p}{2}}$ as a test function, we have
\begin{equation}\label{eq: normal_controllability5bis}
\begin{array}{l}
\lambda u(\bar z) + H_\Ga\left(\bar z, \frac{p(\bar z_0-\bar x_0)}{\alpha^2}\left(\frac{|\bar z_0-\bar x_0|^2}{\alpha^2}+\alpha\right)^{\frac{p}{2}-1}e_0+ \partial_{x_1}\left(\ph|_{\cP_1} \right)(\bar x)e_1, \ldots \right.\\
\left. \quad \quad \quad \quad \quad \quad \ldots, \frac{p(\bar z_0-\bar x_0)}{\alpha^2}\left(\frac{|\bar z_0-\bar x_0|^2}{\alpha^2}+\alpha\right)^{\frac{p}{2}-1}e_0+ \partial_{x_N}\left(\ph|_{\cP_N} \right)(\bar x)e_N \right) \le 0.
\end{array}
\end{equation}
As in the case when $\bar x \in \cP_i\setminus \Ga$, we deduce from \eqref{eq: normal_controllability6bisbis}, \eqref{eq: normal_controllability7bis} and \eqref{eq: normal_controllability5bis} that we have
\begin{equation}\label{eq: normal_controllability6bis}
\lambda u_\alpha(\bar x)+ \lambda \left(\frac{|\bar z_0-\bar x_0|^2}{\alpha^2}+\alpha\right)^{\frac{p}{2}} + H_\Ga(\bar z, D\left( \ph|_{\cP_1}\right) (\bar x), \ldots,  D\left( \ph|_{\cP_N}\right) (\bar x) ) \le 0.
\end{equation}
Let $i\in \{1, \ldots ,N \}$ be such that $ H_\Ga(\bar x, D\left( \ph|_{\cP_1}\right) (\bar x), \ldots,  D\left( \ph|_{\cP_N}\right) (\bar x) ) =  H_i^+(\bar x, D \left(\ph|_{\cP_i}\right)(\bar x) )$. Then,  Lemma \ref{lem inequality on Hgamma+} and \eqref{eq: normal_controllability6bis} give us
\begin{equation}
\label{eq: u_alpha subsolution seconde stape 2}
\begin{array}{lll}
\lambda u_\alpha(\bar x) + H_\Ga(\bar x, D\left( \ph|_{\cP_1}\right) (\bar x), \ldots,  D\left( \ph|_{\cP_N}\right) (\bar x)  ) & \le & -\lambda \left(\frac{|\bar z_0-\bar x_0|^2}{\alpha^2}+\alpha\right)^{\frac{p}{2}}+ M|\bar x_0-\bar z_0|| D \left(\ph|_{\cP_i}\right)(\bar x)  |\\
& & + \omega(|\bar x_0-\bar z_0|),
\end{array}
\end{equation}
where $M$ and $\omega(.)$ are specified in Lemma \ref{lem inequality on Hgamma+}.
On the other hand, from \eqref{eq pseudo coercivity with respect to p' of hi+} in Lemma \ref{lem inequality on Hgamma+} and \eqref{eq: normal_controllability5bis}
\begin{equation}
\label{eq: upper_estimation_on_d_phi1}
-\partial_{x_i} \left(\ph|_{\cP_i} \right)(\bar x) \le \frac{1}{\delta}\left(\lambda \parallel u \parallel_\infty +C_M\left(1+\frac{p|\bar z_0-\bar x_0|}{\alpha^2}\left(\frac{|\bar z_0-\bar x_0|^2}{\alpha^2}+\alpha\right)^{\frac{p}{2}-1}\right) \right).
\end{equation}
This information does not allow us to conclude directly with \eqref{eq: u_alpha subsolution seconde stape 2}.
We need to get a control from above on $\partial_{x_i}\left(\ph|_{\cP_i} \right)(\bar x)$.
For this purpose, we introduce the real number $\tilde p_{i,\bar x,\alpha}$ defined as follow
\begin{equation}
\label{eq: def_upper_bound}
\tilde p_{i,\bar x,\alpha}:= \max\{p_i\in \R : \ph_{i,\bar x,\alpha}(p_i)=\min_{d\in \R}\{\ph_{i,\bar x,\alpha}(d)\}\}
\end{equation}
where
\begin{displaymath}
\ph_{i,\bar x,\alpha}(d)=H_i(\bar x,\frac{p(\bar z_0-\bar x_0)}{\alpha^2}\left(\frac{|\bar z_0-\bar x_0|^2}{\alpha^2}+\alpha\right)^{\frac{p}{2}-1} e_0+d e_i).
\end{displaymath}
We claim that if we note $ h_{i,\bar x}:= \min_{d\in \R}\{H_i(\bar x,d e_i)\}$, we have the following estimation
\begin{equation}
\label{eq: estimation_p_i_x_alpha}
\tilde p_{i,\bar x,\alpha} \le \frac{2}{\delta}\left( h_{i,\bar x}+C_M\left(1+\frac{p|\bar z_0-\bar x_0|}{\alpha^2}\left(\frac{|\bar z_0-\bar x_0|^2}{\alpha^2}+\alpha\right)^{\frac{p}{2}-1}\right) + M_f\frac{p|\bar z_0-\bar x_0|}{\alpha^2}\left(\frac{|\bar z_0-\bar x_0|^2}{\alpha^2}+\alpha\right)^{\frac{p}{2}-1}\right).
\end{equation}
To prove this inequality, it is enough to show that for any real number $q$ larger than the right member of \eqref{eq: estimation_p_i_x_alpha}, $\ph_{i,\bar x,\alpha}(q)> \min_{d\in \R}\left\lbrace\ph_{i,\bar x,\alpha}(d)\right\rbrace$.
\\
Take $q>\frac{2}{\delta}\left( h_{i,\bar x}+C_M\left(1+\frac{p|\bar z_0-\bar x_0|}{\alpha^2}\left(\frac{|\bar z_0-\bar x_0|^2}{\alpha^2}+\alpha\right)^{\frac{p}{2}-1}\right) + M_f\frac{p|\bar z_0-\bar x_0|}{\alpha^2}\left(\frac{|\bar z_0-\bar x_0|^2}{\alpha^2}+\alpha\right)^{\frac{p}{2}-1}\right)$. Then, from the coercivity property of $H_i$ with respect to $p_i$, see \eqref{eq coercivity with respect to p'}, we have
\begin{equation}
\label{eq: estimation_p_i_x_alpha2}
\begin{array}{lll}
H_i(\bar x,\frac{p(\bar z_0-\bar x_0)}{\alpha^2}\left(\frac{|\bar z_0-\bar x_0|^2}{\alpha^2}+\alpha\right)^{\frac{p}{2}-1} e_0+q e_i)
& \ge & \frac{\delta}{2} |q|-C_M\left(1+\frac{p|\bar z_0-\bar x_0|}{\alpha^2}\left(\frac{|\bar z_0-\bar x_0|^2}{\alpha^2}+\alpha\right)^{\frac{p}{2}-1}\right)\\ 
& > &  h_{i,\bar x}+ M_f\frac{p|\bar z_0-\bar x_0|}{\alpha^2}\left(\frac{|\bar z_0-\bar x_0|^2}{\alpha^2}+\alpha\right)^{\frac{p}{2}-1}.
\end{array}
\end{equation}
But, if we consider $d\in \R$ such that $ H_i(\bar x, de_i)=h_{i,\bar x}$, by definition of $p_{i,\bar x,\alpha}$ we have
\begin{equation}
\label{eq: estimation_p_i_x_alpha3bis}
H_i(\bar x,\frac{p(\bar z_0-\bar x_0)}{\alpha^2}\left(\frac{|\bar z_0-\bar x_0|^2}{\alpha^2}+\alpha\right)^{\frac{p}{2}-1} e_0+\tilde p_{i,\bar x,\alpha}  e_i) \le H_i(\bar x,\frac{p(\bar z_0-\bar x_0)}{\alpha^2}\left(\frac{|\bar z_0-\bar x_0|^2}{\alpha^2}+\alpha\right)^{\frac{p}{2}-1} e_0+d e_i),
\end{equation}
and from \eqref{eq: regularity of hi with respect to p}
\begin{equation}
\label{eq: estimation_p_i_x_alpha3}
H_i(\bar x,\frac{p(\bar z_0-\bar x_0)}{\alpha^2}\left(\frac{|\bar z_0-\bar x_0|^2}{\alpha^2}+\alpha\right)^{\frac{p}{2}-1} e_0+d  e_i) \le  h_{i,\bar x} + M_f\frac{p|\bar z_0-\bar x_0|}{\alpha^2}\left(\frac{|\bar z_0-\bar x_0|^2}{\alpha^2}+\alpha\right)^{\frac{p}{2}-1}.
\end{equation}
Finally, \eqref{eq: estimation_p_i_x_alpha3bis} and \eqref{eq: estimation_p_i_x_alpha3} lead to
\begin{equation}
\label{eq: estimation_p_i_x_alpha4}
H_i(\bar x,\frac{p(\bar z_0-\bar x_0)}{\alpha^2}\left(\frac{|\bar z_0-\bar x_0|^2}{\alpha^2}+\alpha\right)^{\frac{p}{2}-1} e_0+ p_{i,\bar x,\alpha}  e_i) \le  h_{i,\bar x} + M_f\frac{p|\bar z_0-\bar x_0|}{\alpha^2}\left(\frac{|\bar z_0-\bar x_0|^2}{\alpha^2}+\alpha\right)^{\frac{p}{2}-1},
\end{equation}
and from \eqref{eq: estimation_p_i_x_alpha2} we infer that
\begin{displaymath}
H_i(\bar x,\frac{p(\bar z_0-\bar x_0)}{\alpha^2}\left(\frac{|\bar z_0-\bar x_0|^2}{\alpha^2}+\alpha\right)^{\frac{p}{2}-1} e_0+ q e_i)  > H_i(\bar x,\frac{p(\bar z_0-\bar x_0)}{\alpha^2}\left(\frac{|\bar z_0-\bar x_0|^2}{\alpha^2}+\alpha\right)^{\frac{p}{2}-1} e_0+\tilde p_{i,\bar x,\alpha}  e_i),
\end{displaymath}
which proves \eqref{eq: estimation_p_i_x_alpha}.\\
\\
We consider three cases :
\\
\textbf{(a)} If $\partial_{x_i} \left(\ph|_{\cP_i} \right)(\bar x)\le 0$: then, \eqref{eq: upper_estimation_on_d_phi1} yields
\begin{displaymath}
|\partial_{x_i} \left(\ph|_{\cP_i} \right)(\bar x)| \le \frac{1}{\delta}\left(\lambda \parallel u \parallel_\infty +C_M\left(1+\frac{p|\bar z_0-\bar x_0|}{\alpha^2}\left(\frac{|\bar z_0-\bar x_0|^2}{\alpha^2}+\alpha\right)^{\frac{p}{2}-1}\right) \right),
\end{displaymath}
and therefore \eqref{eq: normal_controllability7bis} and \eqref{eq: u_alpha subsolution seconde stape 2} imply
\begin{displaymath}
\begin{array}{l}
\lambda u_\alpha(\bar x) + H_\Ga(\bar x, D\left( \ph|_{\cP_1}\right) (\bar x), \ldots,  D\left( \ph|_{\cP_N}\right) (\bar x)  )  \le -\lambda\left(\frac{|\bar z_0-\bar x_0|^2}{\alpha^2} +\alpha\right)^{\frac{p}{2}}+\omega(|\bar x_0-\bar z_0|)
\\
+|\bar x_0-\bar z_0|M\left[\frac{p|\bar z_0-\bar x_0|}{\alpha^2}\left(\frac{|\bar z_0-\bar x_0|^2}{\alpha^2} +\alpha\right)^{\frac{p}{2}-1}+\frac{1}{\delta}\left(\lambda \parallel u \parallel_\infty +C_M\left(1+\frac{p|\bar z_0-\bar x_0|}{\alpha^2}\left(\frac{|\bar z_0-\bar x_0|^2}{\alpha^2}+\alpha\right)^{\frac{p}{2}-1}\right) \right)\right],
\end{array}
\end{displaymath}
and then
\begin{equation}
\label{eq: u_alpha_subsolution_(a)}
\begin{array}{lll}
\lambda u_\alpha(\bar x) + H_\Ga(\bar x, D\left( \ph|_{\cP_1}\right) (\bar x), \ldots,  D\left( \ph|_{\cP_N}\right) (\bar x)  ) & \le & \left(\frac{|\bar z_0-\bar x_0|^2}{\alpha^2} +\alpha\right)^{\frac{p}{2}}\left(-\lambda+ pM\left(1+\frac{C_M}{\delta} \right) \right)\\
& & +\frac{M}{\delta}|\bar z_0-\bar x_0|\left(\lambda \parallel u \parallel_\infty+C_M \right)+\omega(|\bar x_0-\bar z_0|).
\end{array}
\end{equation}
Finally, if $p$ is small enough, \eqref{eq: u_alpha_subsolution_(a)} implies
\begin{equation}
\label{eq: u_alpha_subsolution_(a)2}
\begin{array}{lll}
\lambda u_\alpha(\bar x) + H_\Ga(\bar x, D\left( \ph|_{\cP_1}\right) (\bar x), \ldots,  D\left( \ph|_{\cP_N}\right) (\bar x)  )  & \le & \frac{M}{\delta}|\bar z_0-\bar x_0|\left(\lambda \parallel u \parallel_\infty+C_M \right)+\omega(|\bar x_0-\bar z_0|),
\end{array}
\end{equation}
and from \eqref{eq: proof_u_alpha_lip1bisbis} the right hand side of the latter inequality gives us $m(\alpha)$. 
\\
\textbf{(b)} If $0<\partial_{x_i} \left(\ph|_{\cP_i} \right)(\bar x)\le \tilde p_{i,\bar x,\alpha}$ (case which never occurs if $p_{i,\bar x,\alpha}\le 0$): in this case,  \eqref{eq: normal_controllability7bis} and \eqref{eq: u_alpha subsolution seconde stape 2} give us
\begin{displaymath}
\begin{array}{ll}
\lambda u_\alpha(\bar x) + H_\Ga(\bar x, D\left( \ph|_{\cP_1}\right) (\bar x), &\ldots,  D\left( \ph|_{\cP_N}\right) (\bar x)  )   \le   -\lambda \left(\frac{|\bar z_0-\bar x_0|^2}{\alpha^2}+\alpha \right)^{\frac{p}{2}}\\
&+M\frac{p|\bar z_0-\bar x_0|^2}{\alpha^2}\left(\frac{|\bar z_0-\bar x_0|^2}{\alpha^2}+\alpha \right)^{\frac{p}{2}-1}+M|\bar x_0 -\bar z_0|\tilde p_{i,\bar x,\alpha}+\omega(|\bar x_0-\bar z_0|),
\end{array}
\end{displaymath}
and according to \eqref{eq: estimation_p_i_x_alpha} we get
\begin{equation}
\label{eq: u_alpha_subsolution_(b)}
\begin{array}{l}
\lambda u_\alpha(\bar x) + H_\Ga(\bar x, D\left( \ph|_{\cP_1}\right) (\bar x), \ldots,  D\left( \ph|_{\cP_N}\right) (\bar x)  )\le \frac{2M}{\delta}|\bar x_0 -\bar z_0|( h_{i,\bar x}+C_M)+\omega(|\bar x_0-\bar z_0|)  \\
\quad \quad \quad \quad \quad \quad \quad \quad \quad \quad \quad \quad \quad \quad \quad \quad \quad \quad + \left(\frac{|\bar z_0-\bar x_0|^2}{\alpha^2}+\alpha \right)^{\frac{p}{2}}\left(-\lambda + pM\left(1+\frac{2}{\delta}\left(C_M+M_f\right) \right) \right).
\end{array}
\end{equation}
Finally, for $p$ small enough, \eqref{eq: u_alpha_subsolution_(b)} gives us
\begin{equation}
\label{eq: u_alpha_subsolution_(b)2}
\begin{array}{lll}
\lambda u_\alpha(\bar x) + H_\Ga(\bar x, D\left( \ph|_{\cP_1}\right) (\bar x), \ldots,  D\left( \ph|_{\cP_N}\right) (\bar x)  )  & \le &\frac{2M}{\delta}|\bar x_0 -\bar z_0|( h_{i,\bar x}+C_M)+\omega(|\bar x_0-\bar z_0|),
\end{array}
\end{equation}
and since $h_{i,\bar x}$ is uniformly bounded with respect to $\bar x \in Q$, from \eqref{eq: proof_u_alpha_lip1bisbis} the right hand side of the latter inequality gives us $m(\alpha)$.
\\
\textbf{(c)} If $0 \le \max\{0,\tilde p_{i,\bar x,\alpha}\} < \partial_{x_i} \left(\ph|_{\cP_i} \right)(\bar x)$: in this case, we can not conclude with the inequality \eqref{eq: u_alpha subsolution seconde stape 2}. We need to find a more precise estimate.\\
We recall that $i\in \{1, \ldots, N \}$ is such that $ H_\Ga(\bar x, D\left( \ph|_{\cP_1}\right) (\bar x), \ldots,  D\left( \ph|_{\cP_N}\right) (\bar x) ) =  H_i^+(\bar x, D \left(\ph|_{\cP_i}\right)(\bar x) )$. 
Since $\tilde p_{i,\bar x,\alpha} < \partial_{x_i} \left(\ph|_{\cP_i} \right)(\bar x)$, according to point 4 in Lemma \ref{sec:hamiltonians}, we know that the maximum which defines $H_i^+(\bar x, D(\varphi|_{\cP_i}))(\bar x)$ is reached for one $a\in A_i$ such that $f_i(\bar x,a).e_i=0$. Then, we can apply \eqref{eq: inequality on Hgamma+_bis} in Remark \ref{rmk inequality on Hgamma+} and \eqref{eq: normal_controllability6bis} implies
\begin{displaymath}
\begin{array}{l}
\lambda u_\alpha(\bar x) + H_\Ga(\bar x, D\left( \ph|_{\cP_1}\right) (\bar x), \ldots,  D\left( \ph|_{\cP_N}\right) (\bar x)  )\le -\lambda \left(\frac{|\bar z_0-\bar x_0|^2}{\alpha^2}+\alpha \right)^{\frac{p}{2}}+\omega(|\bar x_0-\bar z_0|)\\
\quad \quad \quad \quad \quad \quad \quad \quad \quad \quad \quad \quad \quad \quad \quad \quad \quad \quad + M\frac{p|\bar z_0-\bar x_0|^2}{\alpha^2}\left(\frac{|\bar z_0-\bar x_0|^2}{\alpha^2}+\alpha \right)^{\frac{p}{2}-1}.
\end{array}
\end{displaymath}
Then, for $p$ small enough, we deduce that
\begin{equation}
\label{eq: proof_u_alpha_subsolution_finer_inequality}
\lambda u_\alpha(\bar x) + H_\Ga(\bar x, D\left( \ph|_{\cP_1}\right) (\bar x), \ldots,  D\left( \ph|_{\cP_N}\right) (\bar x)  )\le \omega(|\bar x_0-\bar z_0|),
\end{equation}
and according to \eqref{eq: proof_u_alpha_lip1bisbis} the right hand side of the latter inequality gives us $m(\alpha)$.
\end{proof}

Finally, let us state the counterpart of Lemma \ref{lem regularization} in the present context.

\begin{lemma}\label{lem: normal_controllability_approximation}
 Assume  $[\rm{H}0]$, $[\rm{H}1]$ and $[\rm{H}2]$. For $y_0\in \Ga$ let $R>0$ be as in \eqref{eq: R_normal_controllability}. We set $Q=B(y_0,R)\cap \cS$.
Let $u:\cS \to \R$ be a bounded, Lipschitz continuous viscosity subsolution of (\ref{HJa}) in $Q$. Let $(\rho_\epsilon)_\epsilon$ be a sequence of mollifiers defined on $\R$.
We consider the function $u_\epsilon$ defined on $Q_\epsilon:=\left\lbrace x\in Q : dist(x,\partial Q)> \epsilon \right\rbrace$ by
\begin{displaymath}
u_\epsilon(x_0e_0+x')= u * \rho_\epsilon(x_0e_0+x')= \int_{\R}u((x_0-\tau)e_0+x')\rho_\epsilon(\tau) d\tau.
\end{displaymath}
We recall that the decomposition of $x\in \cS$, $x=x_0e_0+x'$, is explained in \eqref{eq: decomposition_of_x2}.\\
Then, $\parallel u_\epsilon-u \parallel_{L^\infty(Q_\epsilon)}$ tends to $0$ as $\epsilon$ tends to $0$  and there exists a function $\tilde{m}:(0,+\infty) \to (0,+ \infty)$  such that $\lim_{\epsilon \to 0}\tilde{m}(\epsilon)=0$ and the
function $u_\epsilon - \tilde{m}(\epsilon)$ is a viscosity subsolution of (\ref{HJa}) in $Q_{\epsilon}$.
\end{lemma}
\begin{proof}
 The proof is similar to that of Lemma \ref{lem regularization}. The difference is that we assume that $u$ is Lipschitz continuous in $Q$. This is no longer a consequence of the assumption [H3].
\end{proof}

The following theorem is a local Comparison Principle.
\begin{theorem}
\label{th: local comparison} 
 Assume  $[\rm{H}0]$, $[\rm{H}1]$, $[\rm{H}2]$ and $[\tH3]$. Let $u$ be a bounded, usc subsolution of (\ref{HJa}) in $\cS$ and $v$ be a bounded, lsc supersolution of (\ref{HJa}) in $\cS$. Let $R>0$ be as in \eqref{eq: R_normal_controllability}.
Let $y_0\in \Ga$ be fixed.
Then, if we set $Q=B(y_0,R)\cap \cS$, we have
\begin{equation}
\label{eq: lem local comparison}
\parallel (u-v)_+ \parallel_{L^\infty(Q)} \le \parallel (u-v)_+ \parallel_{L^\infty(\partial Q)}.
\end{equation}
\end{theorem}

\begin{proof}
\textbf{Step 1 :} By assuming $[\tH3]$ instead of $[\rm{H}3]$, we lose the lipschitz continuity of $u$ in a neighborhood of $\Ga$, which was an important property to prove Theorem \ref{sec:comparison-principle-1}. The first step consists therefore of regularizing the subsolution so that it becomes Lipschitz continuous. Take $\alpha,p>0$ two positive numbers and consider $u_\alpha$ ($=u_{\alpha,p}$) the sup-convolution of $u$ with respect to the $x_0$-variable defined in \eqref{eq: normal_controllability2}.
We chose $\alpha,p$ small enough so that Lemma \ref{lem: sup_convolution} can be applied. Thus, from Lemma \ref{lem: sup_convolution}, we know that $u_\alpha$ is Lipschitz continuous in $Q_\alpha$ and that there exists $m:(0,+\infty)\to (0,+\infty)$ such that $\lim_{\alpha \to 0}m(\alpha)=0$ and $u_\alpha-m(\alpha)$ is a subsolution of (\ref{HJa}) in $Q_\alpha$. The definition of the set $Q_\alpha$ ($=Q_{\alpha,p}$) is given in \eqref{def:Q_alpha_set}.\\
\textbf{Step 2 :} We are now able to follow the proof of Theorem \ref{sec:comparison-principle-1}. The next step consists of a second regularization of the subsolution $u$ which this time produces a $\cC^1$ function in $\Ga$.
 Let $Q_{\alpha,\epsilon}$ be the set defined by
\begin{displaymath}
Q_{\alpha,\epsilon}:=\left\lbrace x\in Q : dist(x,\partial Q)>\alpha\sqrt{(2\parallel u \parallel_\infty+\alpha^{\frac{p}{2}})^{2/p}-\alpha}+\epsilon \right\rbrace.
\end{displaymath}
We consider the function $u_{\alpha,\epsilon}$ defined on $Q_{\alpha,\epsilon}$ by
\begin{displaymath}
u_{\alpha,\epsilon}(x_0e_0+x')= u_\alpha * \rho_\epsilon(x_0e_0+x')= \int_{\R}u_\alpha((x_0-\tau)e_0+x')\rho_\epsilon(\tau) d\tau,
\end{displaymath}
where $\rho_\epsilon$ is a sequence of mollifiers defined on $\R$.
It is clear that $u_{\alpha,\epsilon}$ is a $\cC^1$ function in $\Ga\cap Q_{\alpha,\epsilon}$. Besides, from Lemma \ref{lem: normal_controllability_approximation},
\mbox{$\parallel u_{\alpha,\epsilon}-u_\alpha \parallel_{L^\infty(Q_{\alpha,\epsilon})}$} tends to $0$ as $\epsilon$ tends to $0$ and there exists a function $\tilde{m}:(0,+\infty) \to (0,+ \infty)$,  such that $\lim_{\epsilon \to 0}\tilde{m}(\epsilon)=0$ and such that the
function $u_{\alpha,\epsilon}- m(\alpha) - \tilde{m}(\epsilon)$ is a viscosity subsolution of (\ref{HJa}) in $Q_{\alpha,\epsilon}$.\\
\textbf{Step 3 :} Let us prove that
\begin{equation}\label{eq: normal_controllability10}
\parallel (u_{\alpha,\epsilon}-m(\alpha)-\tilde{m}(\epsilon)-v)_+ \parallel_{L^\infty(Q_{\alpha,\epsilon})} \le \parallel (u_{\alpha,\epsilon}-m(\alpha)-\tilde{m}(\epsilon)-v)_+ \parallel_{L^\infty(\partial Q_{\alpha,\epsilon})},
\end{equation}
for a fixed pair $(\alpha,\epsilon)$ of positive numbers.\\
Let $M_{\alpha,\epsilon}$ be the supremum of $u_{\alpha,\epsilon}-m(\alpha)-\tilde{m}(\epsilon)-v$ on $Q_{\alpha,\epsilon}$. The latter is reached for some $\bar x_{\alpha,\epsilon}\in \overline{Q}_{\alpha,\epsilon}$, because the function $u_{\alpha,\epsilon}-m(\alpha)-\tilde{m}(\epsilon)-v$ is usc. If $M_{\alpha,\epsilon}\le 0$, then we clearly have $\parallel (u_{\alpha,\epsilon}-m(\alpha)-\tilde{m}(\epsilon)-v)_+ \parallel_{L^\infty(Q_{\alpha,\epsilon})} \le \parallel (u_{\alpha,\epsilon}-m(\alpha)-\tilde{m}(\epsilon)-v)_+ \parallel_{L^\infty(\partial Q_{\alpha,\epsilon})}$. So, we assume that $M_{\alpha,\epsilon}> 0$ and we want to show that $\bar x_{\alpha,\epsilon}\in \partial Q_{\alpha,\epsilon}$. Assume by contradiction that $\bar x_{\alpha,\epsilon}\not\in \partial Q_{\alpha,\epsilon}$. Then, $\bar x_{\alpha,\epsilon}$ is a local maximum of $u_{\alpha,\epsilon}-m(\alpha)-\tilde{m}(\epsilon)-v$.
\begin{enumerate}
\item If $\bar x_{\alpha,\epsilon}\notin \Ga$ : The usual doubling of variables method, with the auxiliary function $\psi_{\beta}(x,y)= u_{\alpha,\epsilon}(x)-m(\alpha)-\tilde{m}(\epsilon) -v(y)- d^2(x,\bar x_{\alpha,\epsilon})- \frac {d^2(x,y)}{\beta^2}$, leads us to a contradiction.
 \item If  $\bar x_{\alpha,\epsilon}\in \Ga$ : According to Lemma \ref{lem: sup_convolution}, $u_{\alpha,\epsilon}$ is Lipschitz continuous and $\mathcal{C}^1$ with respect to $x_0$ in $\bar Q_{\alpha, \epsilon}$. Then, with a similar argument as in the proof of Lemma \ref{lem subsol inequality on gamma} we can construct a test-function $\ph \in \cR(\cS)$ such that $\ph |_\Ga=u_{\alpha,\epsilon}|_\Ga$ and $\ph$ remains below $u_{\alpha,\epsilon}$ in a neighborhood of $\Ga$ (take for example $\ph(x_0e_0+x_ie_i)= u_{\alpha,\epsilon}(x_0e_0)-Cx_i$ with $C$  great enough). It is easy to check that $v-\ph$ has a local minimum at $\bar x_{\alpha,\epsilon}$. Then, we can use Theorem \ref{sec:prop-visc-sub}, which holds with [$\tH$3], and we have two possible cases:
\begin{description}
\item{[B]}  $ \lambda v(\bar x_{\alpha,\epsilon}) + H_{\Ga}^T(\bar x_{\alpha,\epsilon},D \left(u_{\alpha,\epsilon}|_\Ga\right)(\bar x_{\alpha,\epsilon})) \ge 0$. \\
Moreover, $u_{\alpha,\epsilon}-m(\alpha)-\tilde{m}(\epsilon)$ is a subsolution of (\ref{HJa}) which is $C^1$ on $\Ga$. Then, according to Lemma \ref{lem subsol inequality on gamma}, which can be apply here from Remark \ref{rmk subsol inequality on gamma}, we have the inequality $ \lambda (u_{\alpha,\epsilon}(\bar x_{\alpha,\epsilon})-m(\alpha)-\tilde{m}(\epsilon) )+  H_{\Ga}^T(\bar x_{\alpha,\epsilon},D \left(u_{\alpha,\epsilon}|_\Ga\right)(\bar x_{\alpha,\epsilon})) \le 0$. Therefore, we obtain that $M_{\alpha,\epsilon}=u_{\alpha,\epsilon}(\bar x_{\alpha,\epsilon})-m(\alpha)-\tilde{m}(\epsilon)  - v(\bar x _{\alpha,\epsilon})\le 0$, which is a contradiction.
\item {[A]} With the notations of Theorem \ref{sec:prop-visc-sub},  we have that 
  \begin{displaymath}
    v(x_k)\ge   \int_0^{\eta}   \ell_i(y_{x_k}(s),  \alpha^k_i(s)) e^{-\lambda s} ds + v (y_{x_k}(\eta)) e^{-\lambda \eta}.
  \end{displaymath}
Moreover, from  Lemma \ref{sec:prop-visc-sub-1}, which holds with [$\tH$3],
  \begin{displaymath}
    u_{\alpha,\epsilon}(x_k)- m(\alpha)- \tilde{m}(\epsilon) \le   \int_0^{\eta}   \ell_i(y_{x_k}(s),  \alpha^k_i(s)) e^{-\lambda s} ds + (u_{\alpha,\epsilon} (y_{x_k}(\eta))- m(\alpha)- \tilde{m}(\epsilon) ) e^{-\lambda \eta}.
  \end{displaymath}
Therefore 
\begin{displaymath}
    u_{\alpha,\epsilon}(x_k)- m(\alpha)- \tilde{m}(\epsilon)- v(x_k) \le ( u_{\alpha,\epsilon} (y_{x_k}(\eta))- m(\alpha)- \tilde{m}(\epsilon) - v (y_{x_k}(\eta)) ) e^{-\lambda \eta}.
\end{displaymath}
Letting $k$ tend to $+\infty$, we find that $M_{\alpha,\epsilon}\le M_{\alpha,\epsilon}e^{-\lambda \eta}$, therefore that $ M_{\alpha,\epsilon}\ \le 0$, which is a contradiction.
\end{description}
\end{enumerate}
\textbf{Step 4 :} In order to prove the final result, we have to pass to the limit as $\epsilon$ tends to $0$ and then as $\alpha$ tends to $0$. Let $\alpha>0$ be fixed. Let $\epsilon_0$ be a strictly positive number and $y$ be in $Q_{\alpha,\epsilon_0}$. Then, for all $0< \epsilon < \epsilon_0$ we have that
\begin{equation}\label{eq: normal_controllability14}
(u_{\alpha,\epsilon}(y)-m(\alpha)-\tilde{m}(\epsilon)-v(y))_+ \le \parallel (u_{\alpha}-m(\alpha)-\tilde{m}(\epsilon)-v)_+ \parallel_{L^\infty(\partial Q_{\alpha,\epsilon})}.
\end{equation}
However, $\limsup_{\epsilon\to 0}\parallel(u_{\alpha,\epsilon}-m(\alpha)-\tilde{m}(\epsilon)-v)_+ \parallel_{L^\infty(\partial Q_{\alpha,\epsilon})} \le \parallel (u_\alpha-m(\alpha)-v)_+ \parallel_{L^\infty(\partial Q_\alpha)}$. Indeed, the supremum $\parallel(u_{\alpha,\epsilon}-m(\alpha)-\tilde{m}(\epsilon)-v)_+ \parallel_{L^\infty(\partial Q_{\alpha,\epsilon})}$ is reached for some $x_{\alpha,\epsilon}$ in $\partial Q_{\alpha,\epsilon}$. 
Thus, for any  subsequence such that $ \parallel (u_{\alpha,\epsilon}-m(\alpha)-\tilde{m}(\epsilon)-v)_+ \parallel_{L^\infty(\partial Q_{\alpha,\epsilon})} $ converges to a limit $\bar \ell$ when $\epsilon$ tends to $0$, we can assume that $x_{\alpha,\epsilon}$ converges to some $\bar x_{\alpha}$ that belongs to $\partial Q_\alpha$ when $\epsilon$ tends to $0$. Therefore, since $\parallel u_{\alpha,\epsilon}-u_\alpha \parallel_{L^\infty(Q_{\alpha,\epsilon})}$ tends to $0$ as $\epsilon$ tends to $0$, since $u_\alpha$ is continuous in $Q_\alpha$ and from the lower-semi-continuity of $v$, we have that $\bar \ell\le (u_\alpha(\bar x_\alpha)-m(\alpha)-v(\bar x_\alpha))_+ \le \parallel (u_\alpha-m(\alpha)-v)_+ \parallel_{L^\infty(\partial Q_\alpha)}$. Therefore, by the pointwise convergence of $u_{\alpha,\epsilon}$ to $u_\alpha$, passing to the $\limsup$ as $\epsilon$ tends to $0$ in (\ref{eq: normal_controllability14}) we deduce
\begin{displaymath}
(u_{\alpha}(y)-m(\alpha)-v(y))_+ \le \parallel (u_\alpha-m(\alpha)-v)_+ \parallel_{L^\infty(\partial Q_\alpha)}.
\end{displaymath}
The above inequality is true for all $y\in Q_{\alpha,\epsilon_0}$, with $\epsilon_0$ arbitrarily chosen, then
\begin{displaymath}
\parallel(u_\alpha-m(\alpha)-v)_+ \parallel_{L^\infty( Q_{\alpha})} \le \parallel (u_{\alpha}-m(\alpha)-v)_+\parallel_{L^\infty(\partial Q_\alpha)}.
\end{displaymath}
We are left with taking the limit as $\alpha$ tends to $0$.\\
Fix now $\alpha_0$ and $y \in Q_{\alpha_0}$. For all $0<\alpha \le \alpha_0$ we have
\begin{equation}
\label{eq: proof local comparison the end}
(u_{\alpha}(y)-m(\alpha)-v(y))_+ \le \parallel (u_{\alpha}-m(\alpha)-v)_+ \parallel_{L^\infty(\partial Q_{\alpha})}.
\end{equation} 
As above, we have that $\limsup_{\alpha \to 0}\parallel (u_{\alpha}-m(\alpha)-v)_+ \parallel_{L^\infty(\partial Q_{\alpha})} \le \parallel (u-v)_+ \parallel_{L^\infty(\partial Q)}$. 
 Indeed, the supremum $\parallel (u_{\alpha}-m(\alpha)-v)_+ \parallel_{L^\infty(\partial Q_{\alpha})}$ is reached for some $x_{\alpha}$ in $\partial Q_{\alpha}$.
Thus, for any  subsequence such that $ \parallel (u_{\alpha}-m(\alpha)-v)_+ \parallel_{L^\infty(\partial Q_{\alpha})} $ converges to a limit $\ell$ as $\alpha$ tends to $0$, we can assume that $x_{\alpha}$ converges to $\bar x$ which belongs to $\partial Q$ when $\epsilon$ tends to $0$. But, from the properties of the sup-convolution, the fact that $u$ is upper-semi-continuous, continuous with respect to $x_0$ and the fact that $v$ is lower-semi-continuous it is easy to check that necessarily $\ell \le (u(\bar x)-v(\bar x))_+ \le \parallel (u-v)_+ \parallel_{L^\infty(\partial Q)}$.
Therefore, by the pointwise convergence of $u_\alpha$ to $u$, passing to the $\limsup$ as $\alpha$ tends to $0$ in (\ref{eq: proof local comparison the end}) we deduce
\begin{displaymath}
(u(y)-v(y))_+ \le \parallel (u-v)_+ \parallel_{L^\infty(\partial Q)}, \quad \forall y \in Q_{\alpha_0}.
\end{displaymath}
The above inequality is true $\forall y\in Q_{\alpha_0}$, with $\alpha_0$ arbitrarily chosen, then
\begin{displaymath}
\parallel (u-v)_+ \parallel_{L^\infty( Q)}\le \parallel (u-v)_+ \parallel_{L^\infty(\partial Q)}.
\end{displaymath}
\end{proof}

We are now able to prove the following global Comparison Principle.
\begin{theorem}
\label{th: global comparison normal} 
 Assume  $[\rm{H}0]$, $[\rm{H}1]$, $[\rm{H}2]$ and $[\tH3]$. Let $u$ be a bounded, usc subsolution of (\ref{HJa}) in $\cS$ and $v$ be a bounded, lsc supersolution of (\ref{HJa}) in $\cS$. Then, $u\le v$ in $\cS$.
\end{theorem}
\begin{proof}
The first step consists in a localization of the problem. For a some positive number $K$, we consider the function $\psi(x):=-K-\sqrt{1+|x|^2}$. It is easy to check that for $K\ge \frac{M_f+M_l+1}{\lambda}$,  $\psi$ satisfies the viscosity inequality
\begin{displaymath}
\lambda \psi +\sup_{(\xi,\zeta)\in \FL(x)}\{-D\psi(x).\xi-\zeta\} \le -1.
\end{displaymath}
Then, if we set, for $\mu \in (0,1)$, $u_\mu=\mu u+(1-\mu)\psi$, by convexity properties, we have that
\begin{displaymath}
\lambda u_\mu +\sup_{(\xi,\zeta)\in \FL(x)}\{-Du_\mu(x).\xi-\zeta\} \le -(1-\mu),
\end{displaymath}
where the above inequality is to be understood in the sense of the viscosity. In particular, $u_\mu$ is a subsolution of (\ref{HJa}) in $\cS$. We set
$
M_\mu:=\sup_{x\in \cS}\{ u_\mu(x)-v(x) \}
$. Since $u_\mu(x)$ is usc and tends to $-\infty$ as $|x|$ tends to $+\infty$ and since $v$ is bounded, lsc the above supremum is reached at some $x_\mu\in \cS$. We argue by contradiction, assuming that $M:=\sup_{x\in \cS}\{ u(x)-v(x) \}>0$. Then, since $M_\mu$ tends to $M$ as $\mu$ tends to $1$, for $\mu$ close enough to $1$ we have $M_\mu>0$. We fix such a $\mu$ and we distinguish two cases.
\begin{enumerate}
\item If $x_\mu \in \cP_i\setminus \Ga$ for some $i\in \{1, \ldots, N \}$, then a classical doubling variables method leads to a contradiction.
\item If $x_\mu \in \Ga$, then we are going to obtain a contradiction from Theorem \ref{th: local comparison}. Let $r>0$ be small enough such that for all $i\in \{1, \ldots, N\}$ and $x\in B(\Ga,r)\cap \cP_i$
 \begin{displaymath}
 [-\frac{\delta}{2},\frac{\delta}{2}]\subset \left\lbrace f_i(x,a).e_i : a\in A_i \right\rbrace.
 \end{displaymath}
 We set $Q_\mu:=B(x_\mu,r)\cap \cS$ and we consider the function $\bar u_\mu$ defined in $\cS$ by  $\bar u_\mu(x)=u_\mu(x)-|x-x_\mu|^2(1-\mu)^2$. It is easy to check that if $\mu$ is close enough to $1$, $\bar u_\mu$ is a subsolution of (\ref{HJa}) in $\cS$. Indeed, a direct computation gives
\begin{displaymath}
\lambda \bar u_\mu(x) + \sup_{(\xi,\zeta)\in \FL(x)}\{ -D\bar u_\mu(x,\xi)-\zeta \} \le -(1-\mu)+2rM_f(1-\mu)^2,
\end{displaymath}
and the right hand side of this inequality is clearly negative if $\mu \in (0,1)\cap [1-\frac{1}{2rM_f},1)$. Then, we apply Theorem \ref{th: local comparison} with $Q=Q_\mu$ and the pair of sub/supersolution $(u_\mu,v)$ : this leads to
\begin{equation}\label{eq: normal_controllability15}
M_\mu=u_\mu(x_\mu)-v(x_\mu)=\bar u_\mu(x_\mu)-v(x_\mu) \le \parallel (\bar u_\mu -v)_+ \parallel_{L^\infty(\partial Q_\mu)}.
\end{equation}
However, if $x\in \partial Q_\mu$
\begin{displaymath}
\bar u_\mu(x)-v(x)=u_\mu(x)-v(x)-r^2(1-\mu)^2 \le M_\mu-r^2(1-\mu)^2<M_\mu,
\end{displaymath}
in contradiction with (\ref{eq: normal_controllability15}).
\end{enumerate}
Finally, we deduce that $M\le 0$ and the proof is complete.
\end{proof}

As a consequence, we have the following result of uniqueness and regularity.

\begin{theorem}
 Assume  $[\rm{H}0]$, $[\rm{H}1]$, $[\rm{H}2]$ and $[\tH3]$. Then, the value function $v$ is continuous and is the unique viscosity solution of (\ref{HJa}) in $\cS$.
\end{theorem}
\begin{proof}
It is clear that Theorem \ref{th: global comparison normal}  implies the uniqueness for the Hamilton-Jacobi equation (\ref{HJa}). We have just to prove that $v$ is a solution of this equation. But, according to Theorem \ref{th: v_discontinuous_solution},  $v$ is a discontinuous solution of (\ref{HJa}) in $\cS$. From  Theorem \ref{th: global comparison normal} applied to the pair of sub/supersolution $(v^\star,v_\star)$, we deduce $v^\star\le v_\star$ in $\cS$. Finally, $v=v^\star=v_\star$ and $v$ is continuous.
\end{proof}

\section{Extension to a more general framework with additional dynamics and cost at the interface}
\label{sec:extens-more-gener}
With either $[\rm{H}3]$ or $[\tH3]$, it is possible to extend all the results presented above to the case when there are additional dynamics and cost at the interface. Since the framework is more general with $[\tH3]$, we only discuss this case. We keep the setting from $\S$ \ref{section_normal_controllability} except that we take into account a set of controls $A_0$, a dynamics $f_0: \Ga \times A_0 \longmapsto \R e_0$ and a running cost $\ell_0 : \Ga \times A_0 \longmapsto \R$. The assumptions made on $A_0$, $f_0$ and $\ell_0$ are the following.
\begin{itemize}
\item[(i)]$A_0$ is a non empty compact subset of the metric space $A$, disjoint from the other sets $A_i$, $i\in \{1,\ldots,N\}$.
\item[(ii)] The function $f_0$ satisfies the same boundedness and regularity properties as the functions $f_i$, $i\in \{1,\ldots, N\}$, described in $[\rm{H}0]$.
\item[(ii)] The function $\ell_0$ satisfies the same boundedness and regularity properties as the functions $\ell_i$, $i\in \{1,\ldots, N\}$, described in $[\rm{H}1]$.
\end{itemize}
We define
\begin{displaymath}
M=\left\lbrace (x,a): x\in \cS, a\in A_i \mbox{ if } x\in \cS\setminus \Ga, \mbox{ and } a\in \cup_{i=0}^N A_i \mbox{ if } x\in \Ga  \right\rbrace,
\end{displaymath}
the dynamics
\begin{displaymath}
\forall (x,a)\in M,\quad\quad   f(x, a)=\left\{
    \begin{array}[c]{ll}
      f_i(x,a) \quad &\hbox{ if } x\in \cP_i\backslash \Ga, i\in \{1,\ldots,N\} \\
      f_i(x,a) \quad &\hbox{ if } x\in \Ga \hbox{ and }     a\in A_i,  i\in \{0,1,\ldots,N\},
    \end{array}
\right.
\end{displaymath}
and the running cost
\begin{displaymath}
\forall (x,a)\in M,\quad\quad   \ell(x, a)=\left\{
    \begin{array}[c]{ll}
      \ell_i(x,a) \quad &\hbox{ if } x\in \cP_i\backslash \Ga, i\in \{1,\ldots,N\} \\
      \ell_i(x,a) \quad &\hbox{ if } x\in \Ga \hbox{ and }     a\in A_i,  i\in \{0,1,\ldots,N\}.
    \end{array}
\right.
\end{displaymath}
The infinite horizon optimal control problem is then given by \eqref{eq:3} and \eqref{eq:4}. Then, we consider the Hamilton-Jacobi equation \eqref{HJa} with the new definition of $\FL(x)$ :
\begin{displaymath}
     \FL(x)=\left\{ 
\begin{array}{ll}
\FL_i(x)\quad  &\hbox{if } x \hbox{ belongs to } \cP_i\backslash\Ga \\
\FL_0(x)\cup \bigcup_{i=1,\dots,N} \FL_i^+(x) \quad  &\hbox{if }x \in \Ga,
\end{array}\right.
\end{displaymath}
where for $x\in \Ga$, $\FL_0(x)=\left\lbrace (f_0(x,a), \ell_0(x,a)) : a\in A_0 \right\rbrace$.
The notion of viscosity sub and supersolutions of \eqref{HJa} can be also defined as in \eqref{eq:5} and \eqref{eq:6}. We obtain that the value function is discontinuous viscosity solution of \eqref{HJa} in $\cS$ in the same manner as above, by passing by the relaxed Hamilton-Jacobi equation \eqref{HJa2}. Note that the key result to pass from \eqref{HJa} to \eqref{HJa2}, Lemma \ref{sec:existence-2}, in the present framework becomes
  \begin{displaymath}
  \begin{array}[c]{lll}
      \fl(x)= &  \FL(x) & \hbox{if } x\in \cS\backslash \Ga,\\
       \fl(x)= &  \bigcup_{i=1,\dots,N}  \overline{\rm co} \left\{\FL_0(x)\cup \FL_i^+(x) \cup \bigcup_{j\not=i} \Bigl(\FL_j(x)\cap (\R e_0\times \R) \Bigr)     \right\} & \hbox{if } x\in \Ga.
  \end{array}
  \end{displaymath}
The proof of this can be made in the same way as above.\\
As previously, we have some equivalent definitions for \eqref{eq:5} and \eqref{eq:6} given by :
\begin{itemize}
\item An upper semi-continuous function $u:\cS\to\R$ is a subsolution of \eqref{HJa} in $\cS$
 if for any $x\in\cS$, any $\ph\in\cR(\cS)$ s.t. $u-\ph$ has a local maximum point at $x$, then
 \begin{displaymath}
\begin{array}[c]{ll}
\lambda u(x)+ H_i(x, D\left( \ph|_{\cP_i}\right) (x))  \le 0  \quad &\hbox{if } x\in \cP_i\backslash \Ga,\\
\lambda u(x)+ H_\Ga\left( x,D\left( \ph|_{\Ga}\right) (x), D\left( \ph|_{\cP_1}\right) (x), \dots, D\left( \ph|_{\cP_N}\right) (x) \right)  \le 0  \quad &\hbox{if } x\in \Ga.
\end{array}
 \end{displaymath}
\item  A  lower semi-continuous function $u:\cS\to\R$ is a  supersolution of \eqref{HJa} if for any $x\in\cS$,
 any $\ph\in\cR(\cS)$ s.t. $u-\ph$ has a local minimum point at $x$, then
\begin{displaymath}
\begin{array}[c]{ll}
\lambda u(x)+ H_i(x, D\left( \ph|_{\cP_i}\right) (x))   \ge 0  \quad &\hbox{if } x\in \cP_i\backslash \Ga,\\
\lambda u(x)+ H_\Ga\left( x,D\left( \ph|_{\Ga}\right) (x), D\left( \ph|_{\cP_1}\right) (x), \dots, D\left( \ph|_{\cP_N}\right) (x) \right)   \ge 0  \quad &\hbox{if } x\in \Ga.
\end{array}
\end{displaymath}
\end{itemize}
Where the new definition of $H_{\Ga}: \Ga \times \R e_0 \times \left(\prod_{i=1, \dots , N}(\R e_0 \times \R e_i) \right) \to \R$ is given by
\begin{displaymath}
H_{\Ga}(x,p_0,p_1,\ldots,p_N)= \max\left\lbrace H_0(x,p_0) \max_{i=1,\dots,N} \;H_{i}^+(x,p_i)\right\rbrace,
\end{displaymath}
where the Hamiltonians $H_i^+$ are defined in \eqref{eq:7bis} and the Hamiltonian $H_0:\Ga \times \R e_0 \longmapsto \R$, is defined by
\begin{displaymath}
H_0(x,p_0)= \max_{a\in A_0}\left(-f_0(x,a). p_0-\ell_0(x,a) \right).
\end{displaymath}
The tangential Hamiltonian at $\Ga$, $H_\Ga^T : \Ga \times \R e_0 \longmapsto \R$ is also slightly changed,
\begin{equation}
  \label{eq:9}
H_{\Ga}^T(x,p)=\max \left\lbrace   H_0(x,p), \max_{i=1,\dots,N} \;H_{\Ga,i}^T(x,p)\right\rbrace,
\end{equation}
where the Hamiltonians $H_{\Ga,i}^T$ are defined in \eqref{eq:9bis}.\\
With these new definitions, all the results proved in $\S$ \ref{section_normal_controllability} hold with obvious modifications of the proofs. In particular,
\begin{itemize}
\item[$\bullet$] a subsolution of the present problem is also a subsolution of the former problem. So Lemma \ref{lem subsol inequality on gamma} (with remark \ref{rmk subsol inequality on gamma}) and equation \eqref{eq:17} in Lemma \ref{sec:prop-visc-sub-1} hold.
\item[$\bullet$] Lemma \ref{sec:hamiltonians} holds since it only involves the Hamiltonians $H_i$, $H^+_i$ and $H_{\Ga,i}^T$.
\item[$\bullet$] The proof of Theorem \ref{sec:prop-visc-sub} can still be used. In particular, with the choice of $(q_i)_{i=1,\ldots,N}$ made in this proof, we have the identity
\begin{displaymath}
H_\Ga(\bar x, ,D(\phi|_{\Ga})(\bar x)+q_1e_1,\ldots,D(\phi|_{\cP_N})(\bar x)+q_Ne_N)=H_\Ga^T(\bar x,D(\phi|_\Gamma)(\bar x)).
\end{displaymath}
\item[$\bullet$] The proofs of the regularisation results, Lemma \ref{lem: sup_convolution} and Lemma \ref{lem: normal_controllability_approximation}, are unchanged.
\item[$\bullet$] The proofs of the Comparison principles, Theorem \ref{th: local comparison}  and Theorem \ref{th: global comparison normal} , are unchanged.
\end{itemize}

\appendix
\section{Proof of Lemma \ref{sec:existence-2}}\label{sec: appendix1}

\begin{proof}
The proof of the equality $\fl(x)=  \FL(x)$ for $x\in \cS \backslash \Ga$ is standard (see \cite{MR1484411}, Lemma 2.41, page 129), and the inclusions
$ \overline{\rm co} \left\{ \FL_i^+(x) \cup \bigcup_{j\not=i} \Bigl(\FL_j(x)\cap (\Ga\times \R) \Bigr)\right\} \subset   \fl(x)$  for $x \in \Ga$ and $i\in \{1,\dots,N\}$ are proved by explicitly 
constructing trajectories, see \cite{MR3057137}. We skip this part. This leads to 
 \begin{displaymath}
  \begin{array}[c]{rll}
     \FL(x)= &  \fl(x) & \hbox{if } x\in \cS\backslash \Ga,\\
        \bigcup_{i=1,\dots,N}  \overline{\rm co} \left\{ \FL_i^+(x) \cup \bigcup_{j\not=i} \Bigl(\FL_j(x)\cap (\Ga\times \R) \Bigr)     \right\} \subset & \fl(x) &  \hbox{if } x\in \Ga.
  \end{array}
  \end{displaymath}
  \\
We now prove the reverse inclusion.
Let $x\in \Ga$.  For any $(\zeta,\mu)\in \fl(x)$, there exists a sequence of admissible trajectories $ ( y_{n},\alpha_n)\in \cT_x$ and a sequence of times $t_n\to 0^+$ such that 
\begin{displaymath}
 \lim_{n\to \infty} \frac{1}{t_n}\int_0^{t_n}f(y_{n}(t),\a_n(t))dt=\zeta,
 \quad \hbox{and } \ds \lim_{n\to \infty} \frac{1}{t_n}\int_0^{t_n}\ell(y_{n}(t),\a_n(t))dt=\mu.
 \end{displaymath}
 First, remark that by cronstruction $\zeta$ necessarily belongs to $\R e_0 \times \R^+ e_i$, for some $i\in \{1,\ldots,N\}$.
 
 \begin{itemize}
\item
If $\zeta \notin \R e_0$, then there exists an index $i$ in $\{1,\dots, N\}$ such that $\zeta \in (\R e_0 \times \R^+ e_i) \backslash \R e_0$: in this case, $y_n(t_n)\in \cP_i \backslash \Ga$. Hence,
\begin{equation}
  \label{eq:13}
y_n(t_n)= x+ \sum_{j=1}^N  \int_0^{t_n} f_j(y_n(t),\alpha_n(t)) 1_{y_n(t)\in \cP_j\backslash \Ga }dt + \int_0^{t_n} f(y_n(t),\alpha_n(t)) 1_{y_n(t)\in \Ga }dt,
\end{equation}
with
\begin{equation}
  \label{eq:A1}
  \begin{array}[c]{ll}
  \ds  \int_0^{t_n} f_j(y_n(t),\alpha_n(t)) 1_{y_n(t)\in \cP_j\backslash \Ga}.e_jdt=0 & \hbox{ if } j\not=i,\\
   \ds \int_0^{t_n} f_i(y_n(t),\alpha_n(t)) 1_{y_n(t)\in \cP_i\backslash \Ga}.e_idt= y_n(t_n).e_i.
  \end{array}
  \end{equation}
These identities are a consequence of Stampacchia's theorem: consider for example $j\in \{1, \dots, N\}$ and the function  $ \kappa_j: y\mapsto y 1_{y\in \cP_j\backslash \Ga}.e_j$.  
It is easy to check that $t\mapsto \kappa_j(y_n(t))$ belongs to $W^{1,\infty}(0,t_n)$ and that its weak derivative coincides almost everywhere with $ t\mapsto  f_j(y_n(t),\alpha_n(t))  1_{y_n(t)\in \cP_j\backslash \Ga }.e_j$.
This implies (\ref{eq:A1}).\\
For $j=1,\dots,N$, let $t_{j,n}$ be defined by 
\begin{displaymath}
  t_{j,n}= \left|\Bigl \{ t\in [0,t_n]:   y_n(t)\in \cP_j\backslash \Ga \Bigr \}\right|.
\end{displaymath}
If $j\not=i$ and $ t_{j,n}>0$ then
\begin{displaymath}
  \begin{split}
    \frac 1 { t_{j,n}} \left(    \int_0^{t_n} f_j(y_n(t),\alpha_n(t)) 1_{y_n(t)\in \cP_j\backslash \Ga }dt,  \int_0^{t_n} \ell_j(y_n(t),\alpha_n(t)) 1_{y_n(t)\in \cP_j\backslash \Ga }dt\right) \\
= \frac 1 { t_{j,n}}  \left(    \int_0^{t_n} f_j(x,\alpha_n(t)) 1_{y_n(t)\in \cP_j\backslash \Ga }dt,  \int_0^{t_n} \ell_j(x,\alpha_n(t)) 1_{y_n(t)\in \cP_j\backslash \Ga }dt\right) +o(1)
 \end{split}
\end{displaymath}
 where $o(1)$ is a vector tending to $0$ as $n\to \infty$.  Therefore,  the distance of \\ 
$  \frac 1 { t_{j,n}} \left(   \int_0^{t_n} f_j(y_n(t),\alpha_n(t)) 1_{y_n(t)\in \cP_j\backslash \Ga}dt,  \int_0^{t_n} \ell_j(y_n(t),\alpha_n(t)) 1_{y_n(t)\in \cP_j\backslash \Ga }dt \right)
$ to the set $\FL_j(x)$ tends to $0$. Moreover, according to (\ref{eq:A1}), we have that $\int_0^{t_n} f_j(y_n(t),\alpha_n(t)) 1_{y_n(t)\in \cP_j\backslash \Ga}.e_jdt=0$. 
Hence,  the distance of 
$  \frac 1 { t_{j,n}} \left(    e_j\int_0^{t_n} f_j(y_n(t),\alpha_n(t)) 1_{y_n(t)\in \cP_j\backslash \Ga }dt,  \int_0^{t_n} \ell_j(y_n(t),\alpha_n(t)) 1_{y_n(t)\in \cP_j\backslash \Ga }dt \right)
$
 to the set $\Bigl(\FL_j(x)\cap (\R e_0\times \R) \Bigr)$ tends to zero as $n$ tends to $\infty$.
\\
 If the set $\{t: y_n(t)\in \Ga\}$ has a nonzero measure, then
\begin{displaymath}
  \begin{split}
 \frac 1 {|\{t: y_n(t)\in \Ga\} |}  \left(  \int_0^{t_n} f(y_n(t),\alpha_n(t)) 1_{ y_n(t)\in \Ga } dt      ,   \int_0^{t_n} \ell(y_n(t),\alpha_n(t)) 1_{y_n(t)\in \Ga } dt   \right)\\
  = \frac 1 {|\{t: y_n(t)\in \Ga\} |}  \left(   \int_0^{t_n} f(x,\alpha_n(t)) 1_{ y_n(t)\in \Ga }dt,  \int_0^{t_n} \ell(x,\alpha_n(t)) 1_{y_n(t)\in \Ga }dt\right) +o(1)
   \end{split}
\end{displaymath}
Therefore, the distance of $\frac 1 {|\{t: y_n(t)\in \Ga\} |}  \left(  \int_0^{t_n} f(y_n(t),\alpha_n(t)) 1_{ y_n(t)\in \Ga } dt      ,   \int_0^{t_n} \ell(y_n(t),\alpha_n(t)) 1_{y_n(t)\in \Ga } dt   \right)$ to the set $ \overline{\rm co}
 \left\{ \bigcup_{j=1}^N \FL_j(x)\right\}$ tends to zero as n tends to $\infty$. Moreover, from \mbox{theorem \ref{sec:optim-contr-probl}},  $f(y_n(t), \alpha_n(t))\in \R e_0$ almost everywhere on $\{t: y_n(t)\in \Ga\}$. Therefore, the distance of $\frac 1 {|\{t: y_n(t)\in \Ga\} |}  \left(  \int_0^{t_n} f(y_n(t),\alpha_n(t)) 1_{ y_n(t)\in \Ga } dt      ,   \int_0^{t_n} \ell(y_n(t),\alpha_n(t)) 1_{y_n(t)\in \Ga } dt   \right)$ to the set \\ $ \overline{\rm co}
 \left\{ \bigcup_{j=1}^N \Bigl( \FL_j(x) \cap (\R e_0 \times \R) \Bigr)\right\}$ tends to zero as n tends to $\infty$.
\\
Finally, we know that $ T_{i,n}>0$.
\begin{displaymath}
  \begin{split}
    \frac 1 { t_{i,n}} \left(    \int_0^{t_n} f_i(y_n(t),\alpha_n(t)) 1_{y_n(t)\in \cP_i\backslash \Ga }dt,  \int_0^{t_n} \ell_i(y_n(t),\alpha_n(t)) 1_{y_n(t)\in \cP_i\backslash \Ga }dt\right) \\ =
 \frac 1 { t_{i,n}}   \left( \int_0^{t_n} f_i(x,\alpha_n(t)) 1_{y_n(t)\in \cP_i\backslash \Ga}dt ,   \int_0^{t_n} \ell_i(x,\alpha_n(t)) 1_{y_n(t)\in \cP_i\backslash \Ga }dt \right)+o(1)
 \end{split}
\end{displaymath}
so  the distance of\\ 
$  \frac 1 { t_{i,n}} \left(   \int_0^{t_n} f_i(y_n(t),\alpha_n(t)) 1_{y_n(t)\in \cP_i\backslash \Ga }dt,  \int_0^{t_n} \ell_i(y_n(t),\alpha_n(t)) 1_{y_n(t)\in \cP_i\backslash \Ga}dt \right)
$
 to the set $\FL^+_i(x)$ tends to zero as $n$ tends to $\infty$.
\\
Combining all the observations above, we see that  the distance of  \\
$\left( \frac{1}{t_n}\int_0^{t_n}f(y_{n}(t),\a_n(t))dt,  \frac{1}{t_n}\int_0^{t_n}\ell(y_{n}(t),\a_n(t))dt \right) $ to 
$ \overline{\rm co} \left\{ \FL_i^+(x) \cup \bigcup_{j\not=i} \Bigl(\FL_j(x)\cap (\R e_0\times \R) \Bigr)     \right\}$ tends to $0$ as $n\to \infty$.
Therefore $(\zeta,\mu)\in  \overline{\rm co} \left\{ \FL_i^+(x) \cup \bigcup_{j\not=i} \Bigl(\FL_j(x)\cap (\R e_0 \times \R) \Bigr)     \right\}$.
\item If $\zeta\in \R e_0$, either there exists $i$ such that $y_n(t_n)\in \cP_i\backslash\Ga$ or $y_n(t_n)\in \Ga$:
 \\
$\bullet$ If $y_n(t_n)\in \cP_i\backslash \Ga$, then we
 can make exactly the same argument as above and conclude that 
 $(\zeta,\mu)\in  \overline{\rm co} \left\{ \FL_i^+(x) \cup \bigcup_{j\not=i} \Bigl(\FL_j(x)\cap (\R e_0 \times \R) \Bigr)     \right\}$. 
Since $\zeta\in \R e_0$, we have in fact that $(\zeta,\mu)\in  \overline{\rm co} \bigcup_{j=1}^N \Bigl(\FL_j(x)\cap (\R e_0 \times \R) \Bigr)$. 

\bigskip

 $\bullet$ if  $y_n(t_n)\in \Ga$, according to Stampacchia theorem, we have that \\ $ \ds  \int_0^{t_n} f_j(y_n(t),\alpha_n(t)) 1_{y_n(t)\in \cP_j\backslash \Ga}.e_jdt=0$ for all $j=1,\dots, N$. We can repeat the argument above, and obtain that
 $(\zeta,\mu)\in  \overline{\rm co}  \left\{ \bigcup_{j=1}^N \Bigl(\FL_j(x)\cap (\R e_0 \times \R) \Bigr)\right\}$.
\end{itemize}
\end{proof}

\section{Proof of Theorem \ref{th: v_discontinuous_solution}}\label{appendix: v_discontinuous_solution}

\begin{proof}
First, remark that in this proof, the notation $o_\epsilon(1)$ will denote an application independent of $t$, which tends to $0$ as $\epsilon$ tends to $0$ and that for $k\in \N^\star$ the notation $O(t^k)$ will denote an application independent of $\epsilon$, such that $\frac{O(t^k)}{t^k}$ remains bounded as $t$ tends to $0$.
\textbf{Show that $v_\star$ is a supersolution of \eqref{HJa} :}   for any $x\in\cS$, let $(x_\epsilon)_{\epsilon>0}$ be a sequence such that $x_\epsilon$ tends to $x$ when $\epsilon$ tends to $0$ and $v(x_\epsilon)$ tends to $v_\star(x)$ when $\epsilon$ tends to $0$. Let $\varphi$ be in $\cR(\cS)$ such that $v_\star-\varphi$ has a local minimum at $x$, i.e. there exists $r>0$ such that 
\begin{equation}\label{eq: proof_discontinuous_existence_1}
\forall y\in B(x,r)\cap \cS, \quad \quad v_\star(x)-\varphi(x) \le v_\star(y)-\varphi(y).
\end{equation}
From the dynamic programming principle (Proposition \ref{prop:dpp}), for any $\epsilon>0$ and $t>0$, there exists $(\bar y_{\epsilon,t},\bar \alpha_{\epsilon,t})\in \cT_{x_\epsilon}$ such that
\begin{displaymath}
\begin{array}{lll}
v(x_\epsilon) & \ge & \int_0^t\ell(\bar y_{\epsilon,t}(s),\bar \alpha_{\epsilon,t}(s))e^{-\lambda s}ds+e^{-\lambda t}v(\bar y_{\epsilon,t}(t))-\epsilon\\
 & \ge & \int_0^t\ell(\bar y_{\epsilon,t}(s),\bar \alpha_{\epsilon,t}(s))e^{-\lambda s}ds +e^{-\lambda t}v_\star(\bar y_{\epsilon,t}(t))-\epsilon.\\
\end{array}
\end{displaymath}
Then, according to \eqref{eq: proof_discontinuous_existence_1}, for $\epsilon$ and $t>0$ small enough we have
\begin{displaymath}
v(x_\epsilon)-v_\star(x)    \ge  \int_0^t\ell(\bar y_{\epsilon,t}(s),\bar \alpha_{\epsilon,t}(s))e^{-\lambda s}ds+ v_\star(x)(e^{-\lambda t}-1)+ (\varphi(\bar y_{\epsilon,t}(t))-\varphi(x))e^{-\lambda t} -\epsilon.
\end{displaymath}
Using that $v(x_\epsilon)-v_\star(x)=o_{\epsilon}(1)$, that $\int_0^t\ell(\bar y_{\epsilon,t}(s),\bar \alpha_{\epsilon,t}(s))e^{-\lambda s}ds= \int_0^t\ell(\bar y_{\epsilon,t}(s),\bar \alpha_{\epsilon,t}(s))ds +O(t^2) $ and that $(\varphi(\bar y_{\epsilon,t}(t))-\varphi(x))e^{-\lambda t}=(\varphi(\bar y_{\epsilon,t}(t))-\varphi(x))+t o_{\epsilon}(1)+O(t^2)$,  we finally obtain
\begin{equation}\label{eq: proof_discontinuous_existence_3}
0   \ge  \int_0^t\ell(\bar y_{\epsilon,t}(s),\bar \alpha_{\epsilon,t}(s))ds+ v_\star(x)(e^{-\lambda t}-1)+ \varphi(\bar y_{\epsilon,t}(t))-\varphi(x)+t o_{\epsilon}(1)+O(t^2)+o_\epsilon(1).
\end{equation}
\begin{itemize}
\item \textbf{If $x\in \cP_i\setminus \Ga$ :} since $x_\epsilon$ tends  to $x$ belonging to $\cP_i\setminus \Ga$ as $\epsilon$ tends to $0$ and since the dynamic $f$ is bounded, see remark \ref{rmk: f bounded continuous}, there exists $\bar t>0$ such that for any $t\in (0,\bar t)$, $\bar y_{\epsilon,t}(s)\in \cP_i\setminus\Ga$ for any $s\in (0,t)$. Then, the inequality \eqref{eq: proof_discontinuous_existence_3} can be rewritten as follows
\begin{displaymath}
\begin{array}{lll}
0  & \ge & \displaystyle{ \int_0^t\ell_i(\bar y_{\epsilon,t}(s),\bar \alpha_{\epsilon,t}(s))+D(\varphi|_{\cP_i})(\bar y_{\epsilon,t}(s)).f_i(\bar y_{\epsilon,t}(s),\bar \alpha_{\epsilon,t}(s))ds+v_\star(x)(e^{-\lambda t}-1)}\\
& & \displaystyle{ +\varphi(x_\epsilon)-\varphi(x)+t o_{\epsilon}(1)+O(t^2)+o_\epsilon(1)}.\\
\end{array}
\end{displaymath}
Using that $\varphi(x_\epsilon)-\varphi(x)=o_{\epsilon}(1)$ and that $D(\varphi|_{\cP_i})(\bar y_{\epsilon,t}(t))=D(\varphi|_{\cP_i})(x)+o_{\epsilon}(1)+O(t)$, we get that
\begin{equation}\label{eq: proof_discontinuous_existence_2}
\begin{array}{lll}
0 & \ge &  \displaystyle{\int_0^t\ell_i(\bar y_{\epsilon,t}(s),\bar \alpha_{\epsilon,t}(s))+ D(\varphi|_{\cP_i})(x).f_i(\bar y_{\epsilon,t}(s),\bar \alpha_{\epsilon,t}(s))ds+ v_\star(x)(e^{-\lambda t}-1)}\\
& &\displaystyle{ +t o_{\epsilon}(1)+O(t^2)+o_\epsilon(1)}.
\end{array}
\end{equation}
It is easy to check that $\frac{1}{t}\left(\int_0^tf_i(\bar y_{\epsilon,t}(s),\bar \alpha_{\epsilon,t}(s))ds,\int_0^t\ell_i(\bar y_{\epsilon,t}(s),\bar \alpha_{\epsilon,t}(s))ds \right)$ is at a distance to $\FL_i(x)$ of the order of $o_{\epsilon}(1)+O(t)$. Thus,
\begin{displaymath}
\begin{array}{lll}
0 & \ge &\displaystyle{ v_\star(x)(e^{-\lambda t}-1)-t\left(\max_{(\xi,\zeta)\in \FL_i(x)}\{-D(\varphi|_{\cP_i})(x,\xi)-\zeta\}\right) +t o_{\epsilon}(1)+O(t^2)+o_\epsilon(1)}.
\end{array}
\end{displaymath}
Finally, dividing the latter inequality by $t$, taking the limit as $\epsilon$ tends to $0$ and  in a second time the limit as $t$ tends to $0$ we get the desired inequality
\begin{displaymath}
\lambda v_\star(x)+\max_{(\xi,\zeta)\in \FL_i(x)}\{-D(\varphi|_{\cP_i})(x,\xi)-\zeta\} \ge 0.
\end{displaymath}
\item \textbf{If $x\in \Ga$ and $x_\epsilon\in \cP_i\setminus \Ga$ :} Let $\tau_{\epsilon,t}>0$ be the exit time of $\bar y_{\epsilon,t}$ from $\cP_i\setminus \Ga$. Up to the extraction of a subsequence, we may assume that either $\tau_{\epsilon,t}\le t$ for all $\epsilon>0$ or that $\tau_{\epsilon,t}> t$ for all $\epsilon$.

\textbf{If $\tau_{\epsilon,t} > t$ :} The same calculations as in the case where $x\in \cP_i\setminus \Ga$ give us \eqref{eq: proof_discontinuous_existence_2}. As above $\frac{1}{t}\left(\int_0^tf_i(\bar y_{\epsilon,t}(s),\bar \alpha_{\epsilon,t}(s))ds,\int_0^t\ell_i(\bar y_{\epsilon,t}(s),\bar \alpha_{\epsilon,t}(s))ds \right)$ is at a distance to $\FL_i(x)$ of the order of $o_{\epsilon}(1)+O(t)$. But this time, we have that $\int_0^tf_i(\bar y_{\epsilon,t}(s),\bar \alpha_{\epsilon,t}(s))ds.e_i \ge -x_\epsilon.e_i=o_\epsilon(1)$ for all $t,\epsilon>0$ and then we have the more specific information that

$\frac{1}{t}\left(\int_0^tf_i(\bar y_{\epsilon,t}(s),\bar \alpha_{\epsilon,t}(s))ds,\int_0^t\ell_i(\bar y_{\epsilon,t}(s),\bar \alpha_{\epsilon,t}(s))ds \right)$ is at a distance to $\FL^+_i(x)$ of the order of $\displaystyle{\frac{o_{\epsilon}(1)}{t}}+o_{\epsilon}(1)+O(t)$. Finally, \eqref{eq: proof_discontinuous_existence_2} give us as desired
\begin{displaymath}
\lambda v_\star(x)+\max_{i\in \{1,\ldots, N\}}\max_{(\xi,\zeta)\in \FL^+_i(x)}\{-D(\varphi|_{\cP_i})(x,\xi)-\zeta\} \ge 0.
\end{displaymath}
\textbf{If $\tau_{\epsilon,t} \le t$ :} Then, the inequality \eqref{eq: proof_discontinuous_existence_3} can be written as follow
\begin{equation}\label{eq: proof_discontinuous_existence_4}
\begin{array}{lll}
0  & \ge & \displaystyle{ \int_0^{\tau_{\epsilon,t}}\ell_i(\bar y_{\epsilon,t}(s),\bar \alpha_{\epsilon,t}(s))+D(\varphi|_{\cP_i})(\bar y_{\epsilon,t}(s)).f_i(\bar y_{\epsilon,t}(s),\bar \alpha_{\epsilon,t}(s))ds}\\
& & \displaystyle{ +\sum_{j=1}^N\int_{\tau_{\epsilon,t}}^t 1_{\{\bar y_{\epsilon,t}(s)\in \cP_j\backslash \Ga \}}\left[ \ell_j(\bar y_{\epsilon,t}(s),\bar \alpha_{\epsilon,t}(s))+D(\varphi|_{\cP_j})(\bar y_{\epsilon,t}(s)).f_j(\bar y_{\epsilon,t}(s),\bar \alpha_{\epsilon,t}(s))\right]ds}\\
& & \displaystyle{+ \int_{\tau_{\epsilon,t}}^t 1_{\{\bar y_{\epsilon,t}(s)\in \Ga \}}\left[\ell(\bar y_{\epsilon,t}(s),\bar \alpha_{\epsilon,t}(s))+D\left(\varphi|_\Ga\right)(\bar y_{\epsilon,t}(s)).f(\bar y_{\epsilon,t}(s),\bar \alpha_{\epsilon,t}(s))\right]ds}\\
& & +v_\star(x)(e^{-\lambda t}-1)\displaystyle{ +\varphi(x_\epsilon)-\varphi(x)+t o_{\epsilon}(1)+O(t^2)+o_\epsilon(1)}.
\end{array}
\end{equation}
Note that to obtain the third line of this inequality, we use Point 3 of Theorem \ref{sec:optim-contr-probl_normal}. To obtain the supersolution inequality, we have to deal with each term of the inequality \eqref{eq: proof_discontinuous_existence_4} individually.

$\rightarrow$ Let $j\in \{1,\ldots,N\}$ be such that $\bar y_{\epsilon,t}(t)\not\in \cP_j\setminus \Ga$. Then $\displaystyle{\int_{\tau_{\epsilon,t}}^t 1_{\{\bar y_{\epsilon,t}(s)\in \cP_j\backslash \Ga \}}f_j(\bar y_{\epsilon,t}(s),\bar \alpha_{\epsilon,t}(s)).e_jds=0 }$ and

$\displaystyle{\frac{1}{|\{s:\bar y_{\epsilon,t}(s)\in \cP_j\backslash \Ga \}|}\left(\int_{\tau_{\epsilon,t}}^t 1_{\{\bar y_{\epsilon,t}(s)\in \cP_j\backslash \Ga\} }f_j(\bar y_{\epsilon,t}(s),\bar \alpha_{\epsilon,t}(s))ds, \int_{\tau_{\epsilon,t}}^t 1_{\{\bar y_{\epsilon,t}(s)\in \cP_j\backslash \Ga\} }\ell_j(\bar y_{\epsilon,t}(s),\bar \alpha_{\epsilon,t}(s))ds\right)}$ is at a distance to $\FL_i(x)\cap \left(\R e_0 \times \R \right)$ of the order of $o_\epsilon(1)+O(t)$. Thus, we get
\begin{equation}\label{eq: proof_discontinuous_existence_5}
\begin{array}{l}
\displaystyle{\int_{\tau_{\epsilon,t}}^t 1_{\{\bar y_{\epsilon,t}(s)\in \cP_j\backslash \Ga \}}\left[ \ell_j(\bar y_{\epsilon,t}(s),\bar \alpha_{\epsilon,t}(s))+D(\varphi|_{\cP_j})(\bar y_{\epsilon,t}(s)).f_j(\bar y_{\epsilon,t}(s),\bar \alpha_{\epsilon,t}(s))\right]ds} \\
 \ge |\{s:\bar y_{\epsilon,t}(s)\in \cP_j\backslash \Ga \}|\left(-\max_{(\xi,\zeta)\in \FL_j(x)\cap \left(\R e_0 \times \R \right)}\left\lbrace -D\varphi(x,\xi)-\zeta \right\rbrace +o_\epsilon(1)+O(t)\right).
\end{array}
\end{equation}
$\rightarrow$ If there exists one $k\in \{1,\ldots,N\}$, such that $\bar y_{\epsilon,t}(t)\in \cP_k\setminus \Ga$. In this case, we have that $\displaystyle{\int_{\tau_{\epsilon,t}}^t 1_{\{\bar y_{\epsilon,t}(s)\in \cP_k\backslash \Ga \}}f_k(\bar y_{\epsilon,t}(s),\bar \alpha_{\epsilon,t}(s)).e_kds=\bar y_{\epsilon,t}(t).e_k>0 }$ and then that

$\displaystyle{\frac{1}{|\{s:\bar y_{\epsilon,t}(s)\in \cP_k\backslash \Ga \}|}\left(\int_{\tau_{\epsilon,t}}^t 1_{\{\bar y_{\epsilon,t}(s)\in \cP_k\backslash \Ga \}}f_k(\bar y_{\epsilon,t}(s),\bar \alpha_{\epsilon,t}(s))ds, \int_{\tau_{\epsilon,t}}^t 1_{\{\bar y_{\epsilon,t}(s)\in \cP_k\backslash \Ga \}}\ell_k(\bar y_{\epsilon,t}(s),\bar \alpha_{\epsilon,t}(s))ds\right)}$ is at a distance  to $\FL_k^+(x)$ of the order of $o_\epsilon(1)+O(t)$. Therefore, we get that
\begin{equation}\label{eq: proof_discontinuous_existence_6}
\begin{array}{l}
\displaystyle{\int_{\tau_{\epsilon,t}}^t 1_{\{\bar y_{\epsilon,t}(s)\in \cP_k\backslash \Ga \}}\left[ \ell_k(\bar y_{\epsilon,t}(s),\bar \alpha_{\epsilon,t}(s))+D(\varphi|_{\cP_k})(\bar y_{\epsilon,t}(s)).f_k(\bar y_{\epsilon,t}(s),\bar \alpha_{\epsilon,t}(s))\right]ds} \\
 \ge |\{s:\bar y_{\epsilon,t}(s)\in \cP_k\backslash \Ga \}|\left(-\max_{(\xi,\zeta)\in \FL_k^+(x)}\left\lbrace -D\varphi(x,\xi)-\zeta \right\rbrace +o_\epsilon(1)+O(t)\right).
\end{array}
\end{equation}
$\rightarrow$  From Point 3 of Theorem \ref{sec:optim-contr-probl_normal}, we know that $f(\bar y_{\epsilon,t}(s),\bar \alpha_{\epsilon,t}(s))\in \R e_0$ almost everywhere on $\{s:\bar y_{\epsilon,t}(s)\in \Ga \}$. Therefore, 

$\displaystyle{\frac{1}{|\{s:\bar y_{\epsilon,t}(s)\in \Ga \}|}\left(\int_{\tau_{\epsilon,t}}^t 1_{\{\bar y_{\epsilon,t}(s)\in \Ga\} }f(\bar y_{\epsilon,t}(s),\bar \alpha_{\epsilon,t}(s))ds, \int_{\tau_{\epsilon,t}}^t 1_{\{\bar y_{\epsilon,t}(s)\in \Ga\} }\ell(\bar y_{\epsilon,t}(s),\bar \alpha_{\epsilon,t}(s))ds\right)}$ is at a distance to $\overline{\rm co} \left\lbrace \cup_{j=1}^N \FL_j(x)\cap\left(\R e_0\times \R \right)\right\rbrace$ of the order of $o_\epsilon(1)+O(t)$ and then
\begin{equation}\label{eq: proof_discontinuous_existence_7}
\begin{array}{l}
\displaystyle{\int_{\tau_{\epsilon,t}}^t 1_{\{\bar y_{\epsilon,t}(s)\in \Ga \}}\left[ \ell(\bar y_{\epsilon,t}(s),\bar \alpha_{\epsilon,t}(s))+D_\Ga\varphi(\bar y_{\epsilon,t}(s)).f(\bar y_{\epsilon,t}(s),\bar \alpha_{\epsilon,t}(s))\right]ds} \\
 \ge |\{s:\bar y_{\epsilon,t}(s)\in \Ga \}|\left(-\max_{(\xi,\zeta)\in \overline{\rm co} \left\lbrace \cup_{j=1}^N \FL_j(x)\cap\left(\R e_0\times \R \right)\right\rbrace}\left\lbrace -D\varphi(x,\xi)-\zeta \right\rbrace +o_\epsilon(1)+O(t)\right).
\end{array}
\end{equation}
However, from the piecewise linearity of the function $(\xi,\zeta)\mapsto-D\varphi(x,\xi)-\zeta$ we have that
\begin{displaymath}
\max_{(\xi,\zeta)\in \overline{\rm co} \left\lbrace \cup_{j=1}^N \FL_j(x)\cap\left(\R e_0\times \R \right)\right\rbrace}\left\lbrace -D\varphi(x,\xi)-\zeta \right\rbrace =\max_{j\in \{1,\ldots,N\}}\left\lbrace\max_{(\xi,\zeta)\in  \FL_j(x)\cap\left(\R e_0\times \R \right)}\left\lbrace -D\varphi(x,\xi)-\zeta \right\rbrace\right\rbrace.
\end{displaymath}
Therefore, \eqref{eq: proof_discontinuous_existence_7} give us finally
\begin{equation}\label{eq: proof_discontinuous_existence_8}
\begin{array}{l}
\displaystyle{\int_{\tau_{\epsilon,t}}^t 1_{\{\bar y_{\epsilon,t}(s)\in \Ga \}}\left[ \ell(\bar y_{\epsilon,t}(s),\bar \alpha_{\epsilon,t}(s))+D_\Ga\varphi(\bar y_{\epsilon,t}(s)).f(\bar y_{\epsilon,t}(s),\bar \alpha_{\epsilon,t}(s))\right]ds} \\
\ge  \displaystyle{|\{s:\bar y_{\epsilon,t}(s)\in \Ga \}|\left(-\max_{j\in \{1,\ldots,N\}}\left\lbrace\max_{(\xi,\zeta)\in  \FL_j(x)\cap\left(\R e_0\times \R \right)}\left\lbrace -D\varphi(x,\xi)-\zeta \right\rbrace\right\rbrace + o_\epsilon(1)+O(t)\right)}.
\end{array}
\end{equation}
$\rightarrow$ It remains to deal with $\displaystyle{ \int_0^{\tau_{\epsilon,t}}\ell_i(\bar y_{\epsilon,t}(s),\bar \alpha_{\epsilon,t}(s))+D(\varphi|_{\cP_i})(\bar y_{\epsilon,t}(s)).f_i(\bar y_{\epsilon,t}(s),\bar \alpha_{\epsilon,t}(s))ds}$.
It is the term the most tricky one because $\displaystyle{\int_0^{\tau_{\epsilon,t}} f_i(\bar y_{\epsilon,t}(s),\bar \alpha_{\epsilon,t}(s)).e_ids=-x_\epsilon.e_i<0 }$ generates some outgoing directions. To conclude, we have to use that $x_\epsilon.e_i=o_{\epsilon}(1)$. As a consequence, up to the extraction of a subsequence, we may assume that either $\displaystyle{\lim_{\epsilon \to 0}\frac{|x_\epsilon.e_i|}{\tau_{\epsilon,t}}=0}$ or $\displaystyle{C_1 \le\frac{|x_\epsilon.e_i|}{\tau_{\epsilon,t}}\le C_2}$, for all $\epsilon$, for some positive constants $C_1,C_2$.

If $\displaystyle{\lim_{\epsilon \to 0}\frac{|x_\epsilon.e_i|}{\tau_{\epsilon,t}}=0}$, it is simple to see that $\displaystyle{\frac{1}{\tau_{\epsilon,t}} \left(\int_0^{\tau_{\epsilon,t}}f_i(\bar y_{\epsilon,t}(s),\bar \alpha_{\epsilon,t}(s))ds, \int_0^{\tau_{\epsilon,t}}\ell_i(\bar y_{\epsilon,t}(s),\bar \alpha_{\epsilon,t}(s))ds\right)}$ is at a distance to $ \FL_i(x)\cap\left(\R e_0\times \R \right)$ of the order of $o_\epsilon(1)+O(t)$. Then, we get
\begin{equation}\label{eq: proof_discontinuous_existence_9}
\begin{array}{l}
\displaystyle{\int_0^{\tau_{\epsilon,t}}\ell_i(\bar y_{\epsilon,t}(s),\bar \alpha_{\epsilon,t}(s))+D(\varphi|_{\cP_i})(\bar y_{\epsilon,t}(s)).f_i(\bar y_{\epsilon,t}(s),\bar \alpha_{\epsilon,t}(s))ds} \\
 \ge \tau_{\epsilon,t}\left(-\max_{(\xi,\zeta)\in \FL_i(x)\cap\left(\R e_0\times \R \right)}\left\lbrace -D\varphi(x,\xi)-\zeta \right\rbrace +o_\epsilon(1)+O(t)\right).
\end{array}
\end{equation}
If $\displaystyle{C_1 \le \frac{|x_\epsilon.e_i|}{\tau_{\epsilon,t}}\le C_2}$, we can directly prove that
\begin{equation}\label{eq: proof_discontinuous_existence_10}
\displaystyle{ \int_0^{\tau_{\epsilon,t}}\ell_i(\bar y_{\epsilon,t}(s),\bar \alpha_{\epsilon,t}(s))+D(\varphi|_{\cP_i})(\bar y_{\epsilon,t}(s)).f_i(\bar y_{\epsilon,t}(s),\bar \alpha_{\epsilon,t}(s))ds= O(x_\epsilon.e_i)=o_{\epsilon}(1).}
\end{equation}
Finally, if we put together all the informations of \eqref{eq: proof_discontinuous_existence_5}, \eqref{eq: proof_discontinuous_existence_6}, \eqref{eq: proof_discontinuous_existence_8}, \eqref{eq: proof_discontinuous_existence_9}, \eqref{eq: proof_discontinuous_existence_10} and the fact that  $\varphi(x_\epsilon)-\varphi(x)=o_{\epsilon}(1)$, the inequality \eqref{eq: proof_discontinuous_existence_3} gives us
\begin{displaymath}
0 \ge -t\max_{j\in \{1,\ldots,N\}}\{\max_{(\xi,\zeta)\in \FL^+_j(x)}\{-D(\varphi|_{\cP_j})(x).\xi-\zeta\}\}
+v_\star(x)(e^{-\lambda t}-1)+to_{\epsilon}(1)+O(t^2)+o_\epsilon(1).
\end{displaymath}
Dividing this last inequality by $t$, taking the limit as $\epsilon$ tends to $0$ and  in a second time the limit as $t$ tends to $0$ we get the wanted inequality
\begin{displaymath}
\lambda v_\star(x)+\max_{(\xi,\zeta)\in \FL(x)}\{-D\varphi(x,\xi)-\zeta\} \ge 0.
\end{displaymath}
\item \textbf{If $x\in \Ga$ and $x_\epsilon\in  \Ga$ :} In this case, the inequality \eqref{eq: proof_discontinuous_existence_3} can be written as follow
\begin{displaymath}
\begin{array}{lll}
0  & \ge & \displaystyle{ +\sum_{j=1}^N\int_0^t 1_{\{\bar y_{\epsilon,t}(s)\in \cP_j\backslash \Ga \}}\left[ \ell_j(\bar y_{\epsilon,t}(s),\bar \alpha_{\epsilon,t}(s))+D(\varphi|_{\cP_j})(\bar y_{\epsilon,t}(s)).f_j(\bar y_{\epsilon,t}(s),\bar \alpha_{\epsilon,t}(s))\right]ds}\\
& & \displaystyle{+ \int_0^t 1_{\{\bar y_{\epsilon,t}(s)\in \Ga \}}\left[\ell(\bar y_{\epsilon,t}(s),\bar \alpha_{\epsilon,t}(s))+D_\Ga\varphi(\bar y_{\epsilon,t}(s)).f(\bar y_{\epsilon,t}(s),\bar \alpha_{\epsilon,t}(s))\right]ds}\\
& & +v_\star(x)(e^{-\lambda t}-1)\displaystyle{ +\varphi(x_\epsilon)-\varphi(x)+to_{\epsilon}(1)+O(t^2)+o_\epsilon(1)}.\\
\end{array}
\end{displaymath}
The same study term by term that in the previous case gives us the desired result.
\end{itemize}

\textbf{Show that $v^\star$ is a subsolution of \eqref{HJa} :}   for $x\in \cS$ let $(x_\epsilon)_{\epsilon>0}$ be a sequence such that $x_\epsilon$ tends to $x$ as $\epsilon$ tends to $0$ and $v(x^\epsilon)$ tends to $v^\star(x)$ as $\epsilon$ tends to $0$. Let $\varphi$ be in $\cR(\cS)$ such that $v^\star-\varphi$ has a local maximum at $x$, i.e. there exists $r>0$ such that 
\begin{equation}\label{eq: proof_discontinuous_existence_11}
\forall y\in B(x,r)\cap \cS, \quad \quad v^\star(x)-\varphi(x) \ge v^\star(y)-\varphi(y).
\end{equation}
To prove that $v^\star$ is a subsolution, according to Corollary \ref{cor: discontinuous_existence_1}, it is enough to show that for any $k\in \{1,\ldots,N\}$,
\begin{equation}\label{eq: proof_discontinuous_existence_13}
v^\star(x)+\sup_{a\in A_k \hbox{ s.t. }  f_k(x,a).e_k> 0} (-D(\varphi|_{\cP_k})(x). f_k(x,a) -\ell_k(x,a)) \le 0.
\end{equation}
Let $\bar a\in A_k$ be such that $f_k(x,a).e_k=\delta_{\bar a}>0$.
\begin{itemize}
\item \textbf{If $x_\epsilon \in \cP_i\setminus \Ga$ :}
For all $\epsilon>0$, let $(\bar y_\epsilon,\bar \alpha_\epsilon)\in \cT_{x_\epsilon}$ be an admissible controlled trajectory given by Lemma \ref{lem: discontinuous_existence} and  $\bar \tau_\epsilon$ $(\le C x_\epsilon.e_i)$ its exit time from $\cP_i\setminus \Ga$. So, we consider $\tilde y : [0, \tilde t)\to \cP_k$ the maximal solution of
the integral equation $\bar y(t) = \bar y_\epsilon(\bar \tau_\epsilon)+\int_0^t f_k(\bar y(s),\bar a) ds$. According to the assumptions [H0] and [$\tH$3], we can check that $\tilde t\ge \frac{\delta_{\bar a}}{2L_fM_f}$. Then, we introduce $(y_\epsilon,\alpha_\epsilon)$ the admissible controlled trajectory of $\cT_{x_\epsilon}$ defined in $[0,\bar \tau_\epsilon+\tilde t)$ as follow
\begin{displaymath}
(y_\epsilon(t),\alpha_\epsilon(t))= \left\lbrace
\begin{array}{lll}
(\bar y_\epsilon(t),\bar \alpha_\epsilon(t)) & \mbox{if} & t < \bar \tau_\epsilon,\\
 (\bar y(t-\bar \tau_\epsilon),\bar \a) & \mbox{if} &T\le  t \ge \bar \tau_\epsilon.\\
\end{array}
\right.
\end{displaymath}
Note that by construction, for $t\in (\bar \tau_\epsilon,\tilde t)$ small enough, we have $y_\epsilon(s)\in \cP_i\setminus \Ga$ for $s\in (0,\bar \tau_\epsilon)$ and $y_\epsilon(s)\in \cP_k\setminus \Ga$ for $s\in (\bar \tau_\epsilon, t)$. In the sequel we will consider such a $t$.

From the dynamic programming principle, Proposition \ref{prop:dpp}, we have 
\begin{displaymath}
\begin{array}{lll}
v(x_\epsilon) & \le & \int_0^t\ell(y_{\epsilon}(s), \alpha_{\epsilon}(s))e^{-\lambda s}ds+e^{-\lambda t}v(y_{\epsilon}(t))\\
 & \le &\int_0^t\ell(y_{\epsilon}(s), \alpha_{\epsilon}(s))e^{-\lambda s}ds+e^{-\lambda t}v^\star(y_{\epsilon}(t)).\\
\end{array}
\end{displaymath}
Then, according to \eqref{eq: proof_discontinuous_existence_11}, for $\epsilon$ small enough we have
\begin{displaymath}
v(x_\epsilon)-v^\star(x)    \le  \int_0^t\ell(y_{\epsilon}(s), \alpha_{\epsilon}(s))e^{-\lambda s}ds+ v^\star(x)(e^{-\lambda t}-1)+ (\varphi(y_{\epsilon}(t))-\varphi(x))e^{-\lambda t}.
\end{displaymath}
By similar arguments as above, we can deduce from this inequality the following
\begin{displaymath}
0   \le  \int_0^t\ell( y_{\epsilon,t}(s),\alpha_{\epsilon,t}(s))ds+ v^\star(x)(e^{-\lambda t}-1)+ \varphi( y_{\epsilon,t}(t))-\varphi(x)+t o_{\epsilon}(1)+O(t^2)+o_\epsilon(1).
\end{displaymath}
And by construction of the admissible controlled trajectory $(y_\epsilon,\alpha_\epsilon)$, this inequality can be written as follows
\begin{equation}\label{eq: proof_discontinuous_existence_12}
\begin{array}{lll}
0  & \le & \displaystyle{ \int_0^{\bar \tau_{\epsilon}}\ell_i(y_{\epsilon}(s),\alpha_{\epsilon}(s))+D(\varphi|_{\cP_i})(y_{\epsilon}(s)).f_i(y_{\epsilon}(s),\alpha_{\epsilon}(s))ds}\\
& &\displaystyle{ +\int^t_{\bar \tau_{\epsilon}}\ell_k(y_{\epsilon}(s),\bar a)+D(\varphi|_{\cP_k})(y_{\epsilon}(s)).f_k(y_{\epsilon}(s),\bar a)ds}\\
& & +v^\star(x)(e^{-\lambda t}-1)\displaystyle{ +\varphi(x_\epsilon)-\varphi(x)+t o_{\epsilon}(1)+O(t^2)+o_\epsilon(1)}.\\
\end{array}
\end{equation}
So, using that
$\displaystyle{ \int_0^{\bar \tau_{\epsilon}}\ell_i(y_{\epsilon}(s),\alpha_{\epsilon}(s))+D(\varphi|_{\cP_i})(y_{\epsilon}(s)).f_i(y_{\epsilon}(s),\alpha_{\epsilon}(s))ds=o_{\epsilon}(1)}$

and
$\displaystyle{\frac{1}{t-\bar \tau_\epsilon}\int^t_{\bar \tau_{\epsilon}}\ell_k(y_{\epsilon}(s),\bar a)+D(\varphi|_{\cP_k})(y_{\epsilon}(s)).f_k(y_{\epsilon}(s),\bar a)ds}$

$\displaystyle{=D(\varphi|_{\cP_k})(x).f_k(x,\bar a)+\ell_k(x,\bar a)+o_{\epsilon}(1)+O(t)}$.
Thus, the inequality \eqref{eq: proof_discontinuous_existence_12} gives us
\begin{displaymath}
\begin{array}{lll}
0 & \ge &\displaystyle{ v^\star(x)(e^{-\lambda t}-1)+t\left(D(\varphi|_{\cP_k})(x).f_k(x,\bar a)+ \ell_k(x,\bar a)\right)+t\circ_{\epsilon}(1)+O(t^2)+o_\epsilon(1)}.
\end{array}
\end{displaymath}
Dividing this last inequality by $t$, taking the limit as $\epsilon$ tends to $0$ and  in a second time the limit as $t$ tends to $0$ we give us
\begin{displaymath}
\lambda v^\star(x)-D(\varphi|_{\cP_k})(x).f_k(x,\bar a)- \ell_k(x,\bar a) \le 0.
\end{displaymath}
Since this inequality is true for any $\bar a \in A_k$ such that $f_k(x,a).e_k>0$, we finally get \eqref{eq: proof_discontinuous_existence_13}.
\item \textbf{If $x_\epsilon \in \Ga$ :} Then, the same argument as above can be used. The only difference is that for the construction of the admissible controlled trajectory $(y_\epsilon,\alpha_\epsilon)$ we do not need to join $x_\epsilon$ to $\Ga$. Consequently, the calculations that follow are slightly simpler.
\end{itemize}
\end{proof}

\paragraph{Acknowledgement}
The author was partially funded  by the ANR  project ANR-12-BS01-0008-01.

\bibliographystyle{amsplain}
\bibliography{uniqueness}

\end{document}